\documentclass[11pt,oneside]{amsart}

   \usepackage[english]{babel}
   \usepackage{color}
\usepackage{graphicx,color}

 \usepackage{amssymb}

\usepackage{geometry}
\usepackage{enumerate}
\usepackage{xcolor}
\usepackage{color}
\usepackage{multicol}
\usepackage{tikz}
\usetikzlibrary{matrix,arrows.meta}
\usepackage{amsmath}

\usepackage[colorlinks=true, pdfstartview=FitV, linkcolor=blue, citecolor=blue, urlcolor=blue]{hyperref}
\newcommand{\red}{\textcolor{red}}

\usepackage[normalem]{ulem}

			\usepackage{amsmath,amscd}
			\usepackage{hyperref}
            \usepackage{fontenc}
            \usepackage{textcomp}
\usepackage{inputenc}
      \newtheorem{theorem}{Theorem}[section]
             
       \newtheorem{lemma}[theorem]{Lemma}
       
          \newtheorem{example}[theorem]{Example}

      \newtheorem{coro}[theorem]{Corollary}
      \newtheorem{prop}[theorem]{Proposition}
      \newtheorem{defi}[theorem]{Definition}

       \numberwithin{equation}{section}

            \newtheorem{remark}[theorem]{Remark}

      \makeatletter
      \def\@setcopyright{}
      \def\serieslogo@{}
      \makeatother
   
\begin{document}

%



   \author{Adrien Boyer and Jean-Claude Picaud}
   \address{Universit\'e de Paris - Paris - France }
   \email{adrien.boyer@imj-prg.fr}
	\address{Universit\'e de Tours and Z\"urich University}
	\email{jean-claude.picaud@lmpt.univ-tours.fr}

  
   \title{Riesz operators and some spherical representations for hyperbolic groups}


   \begin{abstract}
   We introduce the Riesz operator in the context of Gromov hyperbolic groups in order to investigate a one parameter family of non unitary and non tempered boundary Hilbertian representations of hyperbolic groups deforming the trivial representation to the standard boundary representation. We prove that asymptotic Schur's relations occur for the representations we build. Moreover, up to normalization, the Riesz operator plays the role in the context of hyperbolic groups of the Knapp-Stein intertwiner for complementary series of Lie groups.
   \end{abstract}

   \subjclass{Primary 43; Secondary 22}


 
   \dedicatory{To L. and N. }

   \date{\today}


   \maketitle


\section{Introduction}
\subsection{Representations theory basics}\label{basics}
Let $G$ be a  group,  $\mathcal{V}$ be a complex vector space and  $GL(\mathcal{V})$ be the group of complex invertible linear maps acting on $\mathcal{V}$.  \emph{A linear (resp. Hilbertian - unitary) representation} of $G$ is a couple $(\pi,\mathcal{V})$ where $\pi: G \rightarrow GL(\mathcal{V})$  is a homomorphism with values  in  invertible (resp.  bounded, unitary) operators acting on $\mathcal{V}$.  

Given  $\langle \cdot,\cdot \rangle$ a non degenerate $\mathbb{C}$-pairing for $(\mathcal{V},\mathcal{W})$ two $\mathbb{C}$-vector spaces  and a representation $(\pi,\mathcal{V})$  of $G$, the contragredient representation of $(\pi,\mathcal{V})$ with respect to  $\langle \cdot,\cdot \rangle$ is defined as $(\tilde{\pi},\mathcal{W})$  where for all $(v,w)\in \mathcal{V}\times \mathcal{W}$ and  for all $g\in G$: 
\begin{align}
 \langle\pi(g)v,w\rangle=  \langle v,  \tilde{\pi}(g^{-1})w  \rangle.
\end{align}

   Given two \emph{linear representations} (\emph{Hilbertian,  resp. unitary})  of $G$ : $(\pi_{1},\mathcal{V}_{1})$ and $(\pi_{2},\mathcal{V}_{2})$, an intertwining  operator is  (\emph{bounded, resp. isometric}) operator  $\mathcal{I}:\mathcal{V}_{1} \rightarrow \mathcal{V}_{2}$, satisfying for all $g\in G$ :
 \begin{equation}
  \mathcal{I} \pi_{1}(g) = \pi_{2}(g) \mathcal{I}.
 \end{equation}

We say that two \emph{representations} are equivalent if there exists an intertwining operator that is an isomorphism in the category concerned.
Given a representation $(\pi,\mathcal{V})$, a subspace $\mathcal{W}\subset \mathcal{V}$ is invariant by $\pi$ if $\pi(\gamma)\mathcal{W}\subset \mathcal{ W}$ for all $\gamma \in G$. A Hilbertian representation $(\pi,\mathcal{V})$ is \emph{irreducible} if $\pi$ has no closed proper invariant subspace.  

\subsection{General discussion about the dual of a group}  A basic problem in harmonic analysis on a locally compact group $G$ is to describe the set $\widehat{G}$ of all irreducible unitary representations, up to equivalence, called the unitary dual of $G$. 
For example, in the setting of reductive Lie groups this problem can be solved and we refer to the introduction of \cite{Ta} for more details in this context. 
In the setting of countable discrete groups the situation is drastically different. Except for virtually abelian groups, the dual space of a discrete group $\Gamma$ is  not even a Hausdorff space for the standard topology structure (see for instance Glimm \cite{Gl} and Thoma \cite{Th}). Quoting the authors of \cite{BD} (see also \cite{FiPi}) : ``\textit{In all other cases there is no natural Borel coding of $\widehat{\Gamma}$, i.e. $\widehat{\Gamma}$ is not
countably separated; for lack of a systematic procedure of constructing all irreducible
representations, a natural problem is to construct large classes of irreducible
representations.''} \\
 According to \cite{Fe}, this is of interest to understand the \emph{non unitary-dual} of a group. 
 In this paper we will construct a large class of irreducible Hilbertian representations for non-elementary \emph{Gromov hyperbolic groups} \cite{Gr} and exhibit a subclass of those groups for which those irreducible non-unitary Hilbertian representations  can be unitarized.\\

\subsection{Construction of representations for Gromov hyperbolic groups: the spherical boundary representations} Let $\Gamma$ be a non-elementary discrete Gromov hyperbolic group. We will consider $\delta$-hyperbolic proper metric spaces $(X,d)$, \emph{roughly geodesic and $\epsilon$-good} (see Section \ref{sh} for the definitions  and Section \ref{sec14} for the motivation) on which $\Gamma$  acts properly and cocompactly by isometries. Fix such $(X,d)$, pick $o\in X$ and set $B(o,R)=\{ x\in X|d(o,x)<R\}$. 
Recall that the Gromov boundary $\partial X$ of $(X,d)$ is endowed with a family of conformal visual metrics $(d_{x,\epsilon})_{x\in X}$ associated with a parameter $\epsilon$. 
The compact metric space $(\partial X,d_{o,\epsilon})$  admits a Hausdorff measure of dimension  \begin{equation}\label{DQ} D:={Q\over \epsilon} \end{equation} 
where 
\begin{equation}\label{volumegrowth}
Q=Q_{\Gamma,d}:=\limsup_{R\to +\infty}\frac{1}{R}\log |\Gamma.o\cap B(o,R)|,
\end{equation}
 is the critical exponent of $\Gamma$ (w.r.t. its action on $(X,d)$).
This $D$-Hausdorff measure is nonzero, finite, unique up to a constant, and denoted by  \(\nu_{o}\) when we normalize it to be a probability. In the context of compact quotients of Hadamard manifolds $\nu_o$ is known as the \emph{Patterson-Sullivan} measure.  
The \emph{class} of measures $[\nu_o]$ is invariant under the action of $\Gamma$ and independent of the choice of \(\epsilon\). Hence, it induces a one parameter family of representations which we call \emph{spherical boundary representations,} already studied in \cite{Boy} and defined by :

  \begin{equation}\label{repre}
\pi_{t}(\gamma) v(\xi)=\bigg(\frac{d\gamma_{*}\nu}{d\nu}(\xi)\bigg)^{\frac{1}{2}+t}v(\gamma^{-1}\xi),\;\;\;\;\;  (t\in\mathbb R,\; v\in L^{2}(\partial X,\nu_{o}))
\end{equation}
whose class  $[\pi_t]$ depends \emph{a priori} on $(X,d)$.\\

The adjoint operator  $\pi^{*}_{t}(\gamma)$  of $\pi_{t}(\gamma)$ is given for any $\gamma \in \Gamma$ by:
\begin{equation}
\pi^{*}_{t}(\gamma)=\pi_{-t}(\gamma^{-1}).
\end{equation}
 In the following, we will denote $\langle\cdot, \cdot\rangle$  the inner product on $L^{2}(\partial X,\nu_o)$, and if $
\textbf{1}_{\partial X}$ is the characteristic function of $\partial X$ we will consider  \emph{the spherical function :}
\begin{equation}
\phi_{t}:\gamma \in \Gamma \mapsto \langle \pi_{t}(\gamma)\textbf{1}_{\partial X},\textbf{1}_{\partial X}\rangle.
\end{equation}

 The representation $\pi_{0}$ is  unitary  and might be thought as an analog of the endpoint of principal series for $SL(2,\mathbb{R})$. It appears in different contexts as a boundary representation and it is also called quasi-regular representation (or Koopman representation) - see \cite{BM}, \cite{BM2}, \cite{G}, \cite{Boy}, \cite{BoyMa},
\cite{BoyP}, \cite{BG}, \cite{BLP}, \cite{Fink}, \cite{KS} and \cite{KS2} for boundary representations and see \cite{Du} and \cite{Du2} for other quasi-regular representations. More recently in \cite{Ca}, boundary representations of hyperbolic groups have been used to obtain results in the general setting of locally compact groups.  \\

\subsection{Intertwiner and Riesz operator}\label{sec14}

 Define the operator for $t>0$: 
 \begin{equation}\label{intertwiner}
\mathcal{I}_{t}(v)(\xi):=\int_{\partial X} \frac{v(\eta)}{d^{(1-2t)D}_{o,\epsilon}(\xi,\eta)}d\nu_{o}(\eta).
\end{equation}
(see the end of Section \ref{HSG} for the visual metrics $d_{o,\epsilon}$).  We call this operator the Knapp-Stein operator associated with the metric $d$ in reference to \cite{KSt}. It is a self-adjoint compact operator on $L^{2}(\partial X,\nu_{o})$. The fundamental observation is that, when the metric space $(X,d)$ is an \emph{$\epsilon$-good hyperbolic space}, (see Nica and \v{S}pakula  \cite{NS} and Section \ref{sh} below), the operator  $\mathcal{I}_{t}$ intertwines $\pi_{t}$ and $\pi_{-t}$ on $L^{2}(\partial X,\nu_{o})$, namely for all $\gamma \in \Gamma$ and for all $v\in L^{2}(\partial X,\nu_{o}):$

\begin{equation}
\mathcal{I}_{t}\pi_{t}(\gamma)v=\pi_{-t}(\gamma)\mathcal{I}_{t}v.
\end{equation}
 In the following, it will be convenient to introduce the function 
\begin{align}\label{lafonction}
\sigma_{t}:\xi\in \partial X \mapsto \mathcal{I}_{t}(\textbf{1}_{\partial X})(\xi)\in \mathbb{R}^{+}\;\;\;\; (t>0).
\end{align}
 as well the (spherical) Riesz operator of order $t>0$ defined by  :
\begin{equation}\label{Rieszpot}
\mathcal{R}_{t}(v)(\xi):=\frac{1}{\sigma_{t}(\xi)}\int_{\partial X} \frac{v(\eta)}{d^{(1-2t)D}_{o,\epsilon}(\xi,\eta)}d\nu_{o}(\eta).
\end{equation}

\subsection{Several representations}

Let $\mathbb C[\Gamma]$ the group algebra and for each $t\in \mathbb R_+$ set :

\begin{equation}\label{algrepre}
\mathcal{F}_{t,\textbf{1}_{\partial X}}:= \pi_{t}(\mathbb C[\Gamma])(\textbf{1}_{\partial X})  \subset L^{2}(\partial X ,\nu_{o}).
\end{equation} 

 The mapping $(v,w)\mapsto \langle v,\mathcal{I}^{2}_{t}(w)\rangle$ is a positive symmetric sesquilinear form on $\mathcal{F}_{t,\textbf{1}_{\partial X}}$. Set $N_{t}:=\lbrace v\in\mathcal{F}_{t,\textbf{1}_{ \partial X}}| \mathcal{I}_{t}(v)=0\rbrace$. Endow the quotient space $\mathcal{F}_{t,\textbf{1}_{\partial X}}/N_{t} $ with the inner product $\langle \cdot,\mathcal{I}^{2}_{t}(\cdot)\rangle$.
 Let $(\mathcal{K}_t, \langle \cdot,\cdot\rangle_{\mathcal{K}_{t}})$ be the Hilbert space completion of $\mathcal{F}_{t,\textbf{1}_{\partial X}}/N_{t}$ with respect to $\langle \cdot,\mathcal{I}^{2}_{t}(\cdot)\rangle$. In an analogous way, with the hypothesis that  the Riesz operator is positive for some $0<t\leq \frac{1}{2}$ (and we will show that this is equivalent for ${\mathcal I}_t$ to be positive), the bilinear form 
 $\langle \cdot,\mathcal{I}_{t}(\cdot)\rangle $ induces another Hilbert space which will be denoted by $(\mathcal{H}_t,\langle \cdot,\cdot\rangle_{\mathcal{H}_{t}})$. We have:
 \begin{prop}\label{posintro}
 If for some $t_0>0$ the operator $\mathcal{R}_{t_0}$ is positive, then $(\pi_{t_0},\mathcal{H}_{t_0})$ is a unitary representation.
 \end{prop}
 
We shall study in the following the families of \emph{Hilbertian representations} $(\pi_{t},\mathcal{K}_{t})_{t>0}$ and \emph{unitary representations} $(\pi_{t},\mathcal{H}_{t})$ for values of $t\in ]0,{1\over 2}]$ such that  the Riesz operator is  positive (and we will see that a necessary condition for positivity is $t\leq 1/2$). The latter are the analogous of complementary series.\\

\subsection{Results}

Let $\Gamma$ be a non elementary Gromov hyperbolic group acting cocompactly and properly on  \emph{a roughly geodesic,  $\epsilon$-good, $\delta$-hyperbolic space} $(X,d)$.  Let $\partial X$ be the Gromov boundary of $X$ and set $\overline{X}:=X\cup \partial X$.  Define for any $R>0$ and for any non-negative integers $n\geq 1$, the spheres $S^{\Gamma}_{n,R}:=\{ \gamma \in \Gamma|(n-1)R\leq d(o,\gamma o)<nR\}$. The cardinal of $S^{\Gamma}_{n,R}$ is denoted by $|S^{\Gamma}_{n,R}|$.  \\
To write the next statements in a synthetic form, we adopt the following notations : for $t>0$, we set $\mathcal{E}_{t,0}=L^{2}(\partial X ,\nu_{o})$ and $\mathcal{E}_{t,2}=\mathcal{K}_{t}$. Then we can state our main result, the following asymptotic Schur's  orthogonality relations, that involve weight functions $\beta_{n,R}$ for which we refer to \cite[Theorem 3.5]{BG}:

\begin{theorem}\label{theo1} 
For any  $R>0$ large enough, there exists a sequence of  measures $\beta_{n,R}:\Gamma \rightarrow \mathbb{R}^{+}$, supported on $S^{\Gamma}_{n,R}\subset \Gamma$, satisfying $\beta_{n,R}(\gamma)\leq C /|S^{\Gamma}_{n,R}|$ for some $C>0$ independent of $n$
such that   for all $t,s > 0  $, for $i,j=0,2$, for all $f,g\in C(\overline{X})$,  for all $v,w\in \mathcal{E}_{t,i}$, $v',w'\in \mathcal{E}_{s,j}$
\begin{align*}
\lim_{n\to \infty}\sum_{\gamma \in S^{\Gamma}_{n,R}}\beta_{n,R}(\gamma) f(\gamma   o) g(\gamma^{-1}   o)\frac{\langle \pi_{t}(\gamma)v,w\rangle_{\mathcal{E}_{t,i}} }{\phi_{t}(\gamma)} &\frac{ \overline{\langle \pi_{s}(\gamma)v',w'\rangle}_{\mathcal{E}_{s,j}}}{\phi_{s}(\gamma)}\\ 
&=\langle g_{|_{\partial X}} \mathcal{R}_{t}(v),\mathcal{R}_{s}(v')\rangle \overline{\langle \mathcal{I}^{i}_{t}(w),f_{|\partial X}\mathcal{I}^{j}_{s}(w')\rangle}.
\end{align*}
 
 \end{theorem}

  Assume that there exists $t\in ]0, {1\over 2}]$ so that $\mathcal{R}_{t}$ is positive. Under this assumption and for this value of $t$, we set $(\pi_{t},\mathcal{H}_{t})$ is a unitary representation. The convergence for the unitary representations is indeed slightly different. We denote by $\mathcal{H}'_{t}$ its dual space that can be viewed as Hilbert space embedded in $L^{2}(\partial X,\nu_{o})$.

  \begin{theorem}\label{theo1'} 
Assume there exist $0<s,t< 1/2$ such that $\mathcal{R}_{t}$ and $\mathcal{R}_{s}$ are positive.
For any  $R>0$ large enough, there exists a sequence of  measures $\beta_{n,R}:\Gamma \rightarrow \mathbb{R}^{+}$, supported on $S^{\Gamma}_{n,R}\subset \Gamma$, satisfying $\beta_{n,R}(\gamma)\leq C /|S^{\Gamma}_{n,R}|$ for some $C>0$ independent of $n$
such that   for all $t,s > 0  $, for all $f,g\in C(\overline{X})$, we have for all $(v_1,w_1)\in \mathcal{H}_{t}\times \mathcal{H}_{t}'$, $(v_2,w_2)\in \mathcal{H}_{s}\times \mathcal{H}'_{s}$
\begin{align*}
\lim_{n\to \infty}\sum_{\gamma \in S^{\Gamma}_{n,R}}\beta_{n,R}(\gamma) f(\gamma   o) g(\gamma^{-1}   o)\frac{\langle \pi_{t}(\gamma)v_1,w_1\rangle_{\mathcal{H}_{t}} }{\phi_{t}(\gamma)} &\frac{ \overline{\langle \pi_{s}(\gamma)v_2,w_2\rangle}_{\mathcal{H}_{s}}}{\phi_{s}(\gamma)}\\ 
&=\langle g_{|_{\partial X}} \mathcal{R}_{t}(v_1),\mathcal{R}_{s}(v_2)\rangle \overline{\langle \mathcal{I}_{t}(w_1),f_{|\partial X}\mathcal{I}_{s}(w_2)\rangle}.
\end{align*}
 
 \end{theorem}

\begin{remark} 

In the following cases,  $\beta_{n,R}$ is the uniform probability measure on $S^{\Gamma}_{n,R}$ (namely $\beta_{n,R}=\dfrac{1}{|S^{\Gamma}_{n,R}|} \sum_{\gamma \in S^{\Gamma}_{n,R}}D_{\gamma o }$ where $D_{\gamma o}$ stands for the unit Dirac mass  at $\gamma o$):\\
\begin{enumerate}
\item $\Gamma=\mathbb F_r$ is the (non abelian) free group  on $r\geq 2$  generators acting on  its standard Cayley graph with the standard left invariant word metric.
\item $\Gamma$ is the fundamental group of a compact negatively curved manifold acting on $(X,d)$ its universal cover.
\item $\Gamma$ is a non elementary discrete group of isometries acting cocompactly on a CAT(-1) space with  non arithmetic length spectrum.
\end{enumerate}

\end{remark}

\begin{remark} The asymptotic Schur relations have been studied as such for tempered representations of Lie groups very recently in \cite{EKY}. It is worth noting that the case $(\mathcal{H}_{t},\mathcal{H}_{s})$ with $t,s< \frac{1}{2}$ of Theorem \ref{theo1'} provides an example of  Schur relations for unitary representations that are \emph{not tempered}, namely the complementary series.
\end{remark}

We deduce three theorems {\it  \`a la Bader-Muchnik} from the above theorem. The main point is that the convergence (and the irreducibility) of  $\pi_{t}$ does not really depend on the Hilbert spaces on which $\Gamma$ acts {\it via} $\pi_{t}$.
 We set $\mathcal{E}_{1,t}:=\mathcal{H}_{t}$ when $\mathcal{R}_{t}$ is assumed to be positive.
 \begin{coro}\label{BML2}
For any  $R>0$ large enough, there exists  $\beta_{n,R}:\Gamma \rightarrow \mathbb{R}^{+}$ as above  (satisfying $\beta_{n,R}(\gamma)\leq C /|S^{\Gamma}_{n,R}|$ for some $C>0$ independent of $n$)  such that
for all $t> 0  $, for $i=0,2$,  for all $f,g\in C(\overline{X})$, for all $v,w\in \mathcal{E}_{t,i}$ and for $i=1$ for all $(v,w)\in \mathcal{E}_{t,1} \times \mathcal{E}'_{t,1}$:
$$\sum_{\gamma \in S^{\Gamma}_{n,R}}\beta_{n,R}(\gamma) f(\gamma   o) g(\gamma^{-1} o) \frac{\langle \pi_{t}(\gamma)v,w\rangle_{\mathcal{E}_{t,i}} }{\phi_{t}(\gamma)}\to \langle  g_{|_{\partial X}}\mathcal{R}_{t}(v),\textbf{1}_{\partial X}\rangle \langle f_{|_{\partial X}}  ,w\rangle_{\mathcal{E}_{t,i}}, $$
as $n\to +\infty$.
\end{coro}

 From which we obtain the following results of irreducibility of Hilbertian representations:
\begin{coro}\label{coro1} Let $t>0$. 
\begin{enumerate}
 \item The Riesz operator is injective if and only if the Hilbertian representation $(\pi_{t},L^{2}(\partial X,\nu_{o}))$ is irreducible.
 \item If the Riesz operator is not injective, then the Hilbert space $L^{2}(\partial X,\nu_{o})$ splits as: $$L^{2}(\partial X,\nu_{o})=\mathcal{V}_{t} \oplus  \mathcal{W}_{-t},$$ 
 where  $(\pi_{t},\mathcal{V}_{t})$ and $(\pi_{-t},\mathcal{W}_{-t})$ are Hilbertian representations. Moreover:
 \begin{itemize}
 
 \item If $0<t<\frac{1}{2}$, the  representations are  $(\pi_{t},\mathcal{V}_{t})$ and $(\pi_{-t},\mathcal{W}_{-t})$ are infinite dimensional. 
 \item Let $t>0$. The Hilbertian representation $(\pi_{-t},\mathcal{W}_{-t})$ is irreducible  and thus the quotient representation $(\pi_{t},L^{2}(\partial X,\nu_{o})/\mathcal{V}_{t} )$ is irreducible as the contragredient representation of $(\pi_{-t},\mathcal{W}_{-t})$ with respect to the $L^{2}$-pairing.  Moreover $(\pi_{-t},\mathcal{W}_{-t})$ and $(\pi_{-t'},\mathcal{W}_{-t'})$ are inequivalent for positive real numbers $t\neq t'$.

 \item If $t=\frac{1}{2}$, then $(\pi_{-\frac{1}{2}},\mathcal{W}_{-\frac{1}{2}})$ and $(\pi_{\frac{1}{2}},L^{2}(\partial X,\nu_{o})/\mathcal{V}_{\frac{1}{2}} )$ are the trivial representation.
  \end{itemize}
 \end{enumerate}
\end{coro}

\begin{coro}\label{coro2}
 Let $t>0$.
The Hilbertian representation $(\pi_{t},\mathcal{K}_{t})$ is irreducible as well as its contragredient representation with respect to the $L^{2}$-pairing. Moreover $(\pi_{t},\mathcal{K}_{t})$ and $(\pi_{t'},\mathcal{K}_{t'})$ are inequivalent for positive real numbers $t\neq t'$.
\end{coro}

\begin{coro}\label{coro3}
We have:
\begin{enumerate}
\item It $t=\frac{1}{2}$, then $(\pi_{t},\mathcal{H}_{t})$ is the trivial representation.
\item If $t>\frac{1}{2}$, then $\mathcal{R}_{t}$ cannot be positive.
 \item \label{3ofcoro} If for some $t\in ]0,\frac{1}{2}[$, the operator $\mathcal{R}_{t}$ is positive, then $(\pi_{t},\mathcal{H}_{t})$  is an infinite dimensional irreducible unitary representation and its contragredient representation with respect to the $L^{2}$-pairing is irreducible as well. Moreover $(\pi_{t},\mathcal{H}_{t})$ and $(\pi_{t'},\mathcal{H}_{t'})$ are inequivalent for positive real numbers $t\neq t'$ in $[0,\frac{1}{2}]$.
\end{enumerate}
 
\end{coro}

\begin{remark}
In this paper, the families $(\pi_{t})_{t\in [0,1/2]}$ defined in (\ref{repre}) (for different Hilbert spaces) can be thought as  one parameter non-unitary deformations from $\pi_{o}$, the standard boundary representation to the trivial representation. The Hilbertian representations $(\pi_{t},\mathcal{K}_{t})_{t\in ]0,1/2[}$ might be thought as an Hilbertian analog of complementary series for general hyperbolic groups.

\end{remark}

\subsection{Comments on conditionally negative hyperbolic groups}
An hyperbolic group $\Gamma$ may satisfy property $(T)$ or in an opposite direction Haagerup property. \\
In the first case, the trivial representation (which is realized by $(\pi_{\frac{1}{2}},\mathcal{H}_{\frac{1}{2}})$) is isolated and this is an obstruction for the Riesz operator ${\mathcal R}_t$ to be positive in a neighborhood of $1/2$. We suspect the existence of $\varepsilon>0$ depending on $(X,d)$ such that $\mathcal{R}_{t}$ is positive on $]0,\frac{1}{2}-\varepsilon]$ generalizing  the positivity of ${\mathcal R}_t$ obtained for uniform lattices of $Sp(n,1)$  in \cite{CJV}  (see Theorem 3.4.3 in \cite{CJV} and also \cite{Ko}).\\
In the second case, amongst several definitions, Haagerup property is characterized by the existence of a proper and conditionally negative function on $\Gamma$, that is a map $\psi\,:\, \Gamma\to \mathbb R_+$ satisfying for all $n\geq 1$, for all $\gamma_1,\cdots,\gamma_n\in \Gamma$ and for all real numbers $\lambda_1,\cdots,\lambda_n$ 
such that $\sum\lambda_i=0$,
$$ \sum_{i,j}\lambda_i\lambda_j\psi(\gamma_j^{-1}\gamma_i)\leq 0.$$
An example of such function is given by $\psi (\gamma)=d(x,\gamma x)$ when $\Gamma$ acts by isometries on an hyperbolic space $(X,d)$, with the additional hypothesis that $d$ is of conditionally negative type, that is $\sum_{i,j}\lambda_i\lambda_jd(x_i,x_j)\leq 0$ for all $x_1,\cdots, x_n\in X$ and  $\sum\lambda_i=0$. Then, the argument appearing in 
 \cite{Bou} for CAT(-1) spaces can be extended to the hyperbolic framework to deduce the positivity of ${\mathcal R}_t$ for all $t\in]0,{1\over 2}]$. Namely :\\


\begin{prop}\label{positivcondneg}
If $\Gamma$ acts on a proper $\delta$-hyperbolic, roughly geodesic and $\epsilon$-good space $(X,d)$ with compact quotient and $d$ of conditionally negative type, then $\mathcal{R}_{t}$ is positive for all $0<t \leq \frac{1}{2}$.
\end{prop}

 We will see (Proposition \ref{Haagerup}) that if $\Gamma$  acts with compact quotient on $(X,d)$ roughly geodesic, $\delta$-hyperbolic and $\epsilon$-good, positivity of   $({\mathcal R}_t)$ for all ${t\in ]0,{1\over 2}]}$ implies Haagerup property for $\Gamma$. Conversely, it would be interesting to know if Haagerup property for $\Gamma$ implies the existence of a $\Gamma$-hyperbolic space  $(X,d)$ such that the corresponding Riesz operators are positive.\\

\subsection{Structure of the paper}
In Section \ref{section2}, we introduce all the geometrical concepts concerning  hyperbolic spaces and groups. The notions of strongly hyperbolic  and $\epsilon$-good spaces are discussed. We also recall equidistribution results in the spirit of Margulis-Roblin. Then all the basics of Patterson-Sullivan measure theory are developed. We start  Section \ref{section3} by studying  properties of the function $\sigma_{t}$ introduced in (\ref{lafonction}), we prove elementary properties of the Riesz operator as well as for the intertwiner $\mathcal{I}_{t}$, such as its compacity. We discuss the existence of a left inverse operator of the intertwiner.  
In Section \ref{section4}, we define and study the Hilbert spaces on which the group acts. In particular we study their pairing structure. In Section \ref{section5}, we study and define the notion of spherical functions and make connections with positive definite functions. We also study the positivity of $\mathcal{I}_{t}$ assuming the metric on the space $X$ defines a negative definite kernel. In Section \ref{section6} we prove spectral inequalities for $\pi_{t}$ and we give several consequences of those inequalities as mixing properties of the boundary representations. We also introduce the generalized Poisson transform and the Riesz transform. Those operators are fundamental to establish all  convergence theorems. Section \ref{section7} is devoted to the proofs. Last, we add an appendix providing a way of unitarization based on $\ell^{2}$-cohomology of discrete groups.

\subsection*{Acknowledgement}
We are grateful to Nigel Higson to have strengthened our interest in intertwining operators and would also like to thank Jean-Martin Paoli for having suggested to study the square of the intertwining operator. We wish to thank Vincent Millot for useful discussion and for the reference \cite{Sam}. Eventually, we would like to thank Antoine Pinochet-Lobos for helpful comments on a first version of this manuscript. Finally, we are grateful to Marc Bourdon for discussions on Besov spaces and $\ell^{2}$-Betti numbers. \\
The first author thanks warmly Christophe Pittet for his support.
The second author is grateful  to Institut f\"ur Mathematik Universit\"at Z\"urich -  specially to Viktor Schr\"oder and also for the stimulating ambiance in Dynamical Systems seminar -   for their kind hospitality  during the past two years when this work was partially completed. \\
 The authors thank the
MFO  for RIP program providing excellent research environment where this work has been initiating.

\section{Preliminaries}\label{section2}

In the sequel we adopt a universal notation $C$ for constants independent of other variables than the global geometry. 
 \subsection{Hyperbolic metric spaces and boundaries}\label{HSG} (A good reference for this section is \cite{BH}).\\ 
A metric space \((X,d)\) is said to be \emph{hyperbolic} if there exists $\delta\geq 0$  and a\footnote{if the condition holds for some \(o\) and \(\delta\), then it holds for any \(o\) and \(2\delta\)} basepoint \(o\in X\) such that for any \(x,y,z\in X\)  one has
\begin{equation}\label{hyp}
  (x,y)_{o}\geq \min\{ (x,z)_{o},(z,y)_{o}\}-\delta,
\end{equation}
where \((x,y)_{o}\) stands for the \emph{Gromov product} of \(x\) and \(y\) from \(o\), that is
\begin{equation}
  (x,y)_{o}=\frac{1}{2}(d(x,o)+d(y,o)-d(x,y)).
\end{equation}
We will call indifferently $(X,d)$ a $\delta$-hyperbolic space when the constant $\delta$ needs to be precised and we will only consider in the sequel \emph{proper} hyperbolic metric spaces (metric spaces for which closed balls are compact).\\
A map  $\phi:(X,d_{X}) \to (Y,d_{Y})$ between metric spaces is a \emph{quasi-isometry} if there exist positive constants $L,\lambda>0$ so that 
\begin{equation}
\frac{1}{L}d_{X}(x,y)-\lambda \leq d_{Y}(\phi(x),\phi(y))\leq L d_{X}(x,y)+\lambda.
\end{equation}

Within the class of \emph{geodesic metric spaces},  hyperbolicity is invariant under quasi-isometries, which is not the case for arbitrary metric spaces.\\

A sequence $(a_{n})_{n\in  \mathbb{N}}$ in $X$ converges at infinity if $(a_{i},a_{j})_{o}\rightarrow +\infty$ as $i,j$ goes to $+\infty$. Set  $(a_n)\sim(b_n)\Leftrightarrow (a_i,b_j)_o\to \infty$ as $i,j\to \infty$ : It defines an equivalence relation and the set of equivalence classes  (independant on the base point) is denoted by $\partial X$. The topology on $X$ naturally extends to  $\overline{X}:=X\cup \partial X$ so that $\overline{X}$ and $\partial X$ are compact sets. The formula 
\begin{equation}\label{gromovextended}
(\xi,\eta)_{o}:= \sup \liminf_{i,j}(a_{i},b_{j})_{o} 
\end{equation}
(where the supremum is taken over all sequences $(a_n), (b_n)$ which represent $\xi$ and $\eta$ respectively)
allows to extend the Gromov product on $\overline{ X}\times \overline{ X}$ but in a \emph{non continuous way} in general (it is obviously continuous on $X\times X$ as the distance function to a point is).\\
We refer to \cite[3.17 Remarks, p. 433]{BH} for the following.\\
 Let $(X,d)$ be an hyperbolic metric space and fix a base point $o$ in $ X$.
 \begin{enumerate}
\item For all $\xi,\eta$ in $\overline{X}$ there exist sequences $(a_n)$ and $(b_n)$ such that $\xi=\lim a_{n}$, $\eta=\lim b_{n}$ and $(\xi,\eta)_o=\lim_{n \rightarrow +\infty} (a_n,b_n)_o$.

\item For all $\xi,\eta$ and $u$ in $\overline{X}$ by taking limits we still have $$(\xi,\eta)_o \geq \min{\{(\xi,u)_{o},(u,\eta)_{o} \} }-2\delta .$$
\item For all $\xi,\eta$ in $\partial X$ and all sequences $(a_i)$ and $(b_j)$ in X with $\xi= \lim a_ i$ and $\eta= \lim b_j$, we have:

$$(\xi,\eta)_{o} -2\delta \leq  \liminf_{i,j}(a_i ,b_j)_{o} \leq (\xi,\eta)_{o}.$$ 

\end{enumerate}

We refer to \cite[8.- Remarque, Chapitre 7, p. 122]{G} for a proof of the statement (3).  The boundary $\partial X$ carries a family of \emph{visual metrics}, depending on \(d\) and a real parameter \(\epsilon > 0\) denoted from now by $d_{o,\epsilon}$. For $1<\epsilon \leq \frac{\log(2)}{4\delta} $ then \(d_{o,\epsilon}\) is a metric, relative to a base point o on \(\partial X\) satisfying
\begin{equation}\label{eq:visual-metric-def}
 e^{-\epsilon (\xi,\eta)_o-\epsilon c_{v}}\leq   d_{o,\epsilon}(\xi,\eta)\leq  e^{-\epsilon (\xi,\eta)_o},
\end{equation}
for some $c_{v}>0$ such that $e^{-\epsilon c_{v}}=(3-2e^{\delta \epsilon}).$

We refer to \cite[Chapter III-H pp. 435-437]{BH} for those visual metrics.\\

\subsection{Hyperbolic groups}
\subsubsection{Definition}

A group $\Gamma$ acts properly discontinuously on a proper metric space if for every compact sets $K,L\subset X$, the set $\{\gamma\in \Gamma\;;\; \gamma K\cap L\neq \emptyset\}$ is finite. The action is said to be geometric if furthermore $\Gamma$ acts by isometries and  $X/\Gamma$ is compact. A {\it group $\Gamma$ is said to be hyperbolic} if it acts  geometrically on a proper hyperbolic space. A hyperbolic group is necessarily finitely generated (by \v{S}varc-Milnor's lemma). For such $\Gamma$,  any  finite set of generators $\Sigma$  gives rise to a so called metric Cayley graph  $({\mathcal G}(\Gamma, \Sigma),d_{\Sigma})$  whose set of vertices are the elements of $\Gamma$, linked by length-one edges if and only if they differ by an element of $\Sigma$.  Every {\it geodesic} hyperbolic metric space $(X,d)$ on which $\Gamma$ acts geometrically is quasi-isometric to a (any of the) Cayley graph(s) of a hyperbolic group. All constructions below will depend {\it a priori} on hyperbolic metric spaces $(X,d)$ on which $\Gamma$ acts geometrically  and we will consider as well {\it non geodesic} hyperbolic spaces. \\
\subsubsection{Notations}
As far as the geometry of $(X,d)$ or that of $\Gamma$ for its action  as above is concerned, we will fix the following notations :\\
Endow $\Gamma$ with the length function  $|\cdot|: \Gamma \rightarrow \mathbb{R}^{+}$ defined as $|\gamma |=d(o,\gamma o)$ for a basepoint $o$ and for any $\gamma \in \Gamma.$\\
The open (resp. closed) balls in $X$ are denoted by $B_X(o,R)$ (resp. $\bar B_X(o,R))$ and define the (thickened) spheres of $X$ associated with $n\in \mathbb N$ and fixed $R>0$  as $S_{n,R}(o)=B_X(o,(n+1)R)\setminus B_X(o,nR)$. We will need as well spheres of a $\Gamma$-orbit :
$S^{\Gamma}_{n,R}(o)=\{\gamma\in \Gamma| \Gamma o\cap S_{n,R}(o)\neq\varnothing\}$ and will use the notation $S_{n,R}$ and $S^{\Gamma}_{n,R}$ rather than $S_{n,R}(o)$ and $S^{\Gamma}_{n,R}(o)$ (as $o$ will be fixed).\\

  \subsection{Roughly geodesic spaces,  good and strong hyperbolic spaces and Gromov boundary}\label{sh}
  
The classical theory of $\delta$-hyperbolic spaces works under the assumption that the spaces are geodesic. In general, the Gromov product associated with a word metric on a Cayley graph of a Gromov hyperbolic group does not extend continuously to the bordification. In the following we will need  metrics $d$ on $\Gamma$ such that the Gromov product extends continuously to the bordification : the price to pay is that $(\Gamma, d)$ won't be geodesic but only   \emph{roughly geodesic}  :

  \begin{defi}\label{roughiso}
A metric space $(X,d)$ is roughly geodesic  if there exists $C=C_X>0$ so that for all $x,y\in X$ there exists a rough geodesic joining $x$ and $y$,  that is map $r:[a,b]\subset \mathbb{R}\rightarrow X$ with $r(a)=x$ and $r(b)=y$ such that $ |t-s|-C_X \leq d(r(t),r(s))\leq  |t-s|+C_X$ for all $t,s\in [a,b]$.
\end{defi}
 
We say that two rough geodesic rays 
$r,r':[0,+\infty)\rightarrow X$ are equivalent if \\ $\sup_{t}d(r(t),r'(t))<+\infty$. We  write $\partial_{r} X$ for the set of equivalence classes of rough geodesic rays. When $(X,d)$ is a proper roughly geodesic space,  $\partial X$ and $\partial_{r} X$ coincide. If $\Gamma$ is a hyperbolic group
 endowed with a left invariant hyperbolic metric quasi-isometric to a word metric, it turns out that the metric space $(\Gamma,d)$ is a proper roughly geodesic $\delta$-hyperbolic metric space, see for example \cite[Section 3.1]{Ga}. Hence a hyperbolic group acts geometrically on a proper, roughly geodesic $\delta$-hyperbolic metric space.\\

\begin{defi}(Nica-\v{S}pakula)
 We say that a hyperbolic space $X$ is $\epsilon$-good, where $\epsilon>0$, if the following
two properties hold for each base point $o\in X$:
\begin{itemize}
\item The Gromov product $(\cdot,\cdot)_{o}$ on $X$ extends continuously to the bordification $X\cup \partial X$.

\item  The map $d_{o,\epsilon}:(\xi,\eta)\in \partial X \mapsto  e^{-\epsilon(\xi,\eta)_{o}}$ is a metric on  $\partial X$.
\end{itemize}
\end{defi}

 The metric space  $(\partial X,d_{o,\epsilon})$ is a compact subspace of the bordification $\overline{X}:=\partial X \cup X$ (also compact) and the open ball  centered at $\xi$ of radius $r$ with respect to $d_{o,\epsilon}$ will be denoted by $B(\xi,r)$.\\

In \cite{NS} the authors introduce the notion of strong hyperbolicity that we recall here:
\begin{defi}
A metric space $(X,d)$ is \emph{$\epsilon$-strongly hyperbolic} if for all $x,y,z,o \in X$ we have 
$$
e^{-\epsilon (x,y)_{o}}\leq  e^{-\epsilon (x,z)_{o}}+e^{-\epsilon (z,y)_{o}}.
$$
\end{defi}
Then, the authors prove the following
\begin{theorem}(Nica-\v{S}pakula)
An $\epsilon$-strongly hyperbolic space is  
 $\epsilon$-good and $\frac{\log(2)}{\epsilon}$-hyperbolic.
\end{theorem}


To sum up, we state the  following theorem, that is a combination of results due to Blach\`ere, Ha\" issinsky and Matthieu \cite{BHM} and of Nica and \v{S}pakula \cite{NS}. 

\begin{theorem}
A hyperbolic group acts by isometries, properly discontinuously and cocompactly on a roughly geodesic  $\epsilon$-good $\delta$-hyperbolic space.
\end{theorem}

As indicated in the introduction, we will consider in the following geometric actions of $\Gamma$ (hyperbolic) on roughly geodesic, $\epsilon$-good hyperbolic spaces $(X,d)$. Two  examples of such spaces are given 1) by the group $\Gamma$ itself endowed with the Mineyev metric \cite{Min} and 2) again by $\Gamma$ endowed with Green metrics $d_{\mu}$   associated with symmetric finitely supported random walks $\mu$. Roughly geodesic property of the second (family of) example is proved by Blach\`ere, Ha\" issinsky and Matthieu in \cite{BHM} and Nica and  \v{S}pakula in \cite{NS} prove that $(\Gamma,d_{\mu})$ is a $\epsilon$-strongly hyperbolic space.\\

 
 \textbf{From now on, except if indicated, we consider $(X,d)$ a proper roughly geodesic $\epsilon$-good $\delta$-hyperbolic space}.

\subsection{Busemann functions}

In the context of $\epsilon$-good spaces, the Busemann function is defined as :
\begin{equation}\label{buseman}
\beta_{\cdot}(\cdot,\cdot):(\xi,x,y)\in \partial X \times X \times X \mapsto \beta_{\xi}(x,y)= \lim_{z\to \xi}d(x,z)-d(z,y)\in \mathbb{R}.
\end{equation}

Note that:
\begin{equation}\label{buseman}
 \beta_{\xi}(x,y)= 2(\xi,y)_{x}-d(x,y).
\end{equation}


We have for all $\xi \in \partial X$ and for all $x,y\in X$:
\begin{equation}\label{buseman'}
(\xi,y)_{x}\leq d(x,y),
\end{equation}
and thus 
\begin{equation}\label{buseman''}
\beta_{\xi}(x,y)\leq d(x,y).
\end{equation}
The visual metrics $(d_{x,\epsilon})_{x\in X}$ of parameter $\epsilon>0$ on the boundary satisfy the following conformal relation, that is for all $x,y\in X$ and for all $\xi,\eta \in \partial X$ we have:
\begin{equation}\label{visualrough}
d_{y,\epsilon}(\xi,\eta)= \exp \big(\frac{\epsilon}{2} (\beta_{\xi}(x,y)+\beta_{\eta}(x,y))\big)d_{x,\epsilon}(\xi,\eta).
\end{equation}

\subsection{The Patterson-Sullivan measure}
\label{sec:patt-sull-meas}
Recall that the volume growth of $\Gamma$-orbits  on $(X,d)$ is controlled as
\begin{equation}\label{volumegrowth}
C^{-1} e^{Q  R}\leq |\Gamma.o\cap B(o,R)|\leq C e^{Q  R},
\end{equation}
for some constant $C$ independent of $R$.\\ 
The limit set of $\Gamma$ denoted by $\Lambda_{\Gamma}$ is the set of  accumulation points in $\partial X$ of an (actually any) orbit. Namely $\Lambda_{\Gamma}:=\overline{\Gamma . o}\cap \partial X$, with the closure in $\overline{X}$. \\
We say that $\Gamma$ is non-elementary if $\sharp\Lambda_{\Gamma}>2$ (and in this case, $\sharp\Lambda_{\Gamma}=\infty$). In the following $\Gamma$ acts geometrically so that $\Lambda_{\Gamma}=\partial X$.\\


The Patterson-Sullivan measure \(\nu_{o}\) is quasi-invariant under the action of  \(\Gamma\)  on $\partial X$ and since we are working in the setting of hyperbolic groups actings on $\epsilon$-good $\delta$-hyperbolic space we have for $\nu_o$-almost every $\xi \in \partial X$:

 \begin{equation}
\frac{d\gamma_{*}\nu_{o}}{d\nu_{o}}(\xi)= e^{\large Q \beta_{\xi}(o,\gamma o)}.
\end{equation}

It turns out that the support of $\nu_{o}$ is  $\partial X$ and moreover $\nu_{o}$ is Ahlfors regular of dimension \(D\), which means that we have the following estimate for  volume of balls: there exists $C\geq 1$ so that for all $\xi \in \Lambda_{\Gamma}$ for all \(r \leq Diam (\partial X)\):
\begin{equation}
 C^{-1} r^{D}\leq  \nu_{o}(B(\xi,r)) \leq C r^{D}.
\end{equation}

Finally, the Patterson-Sullivan measure is ergodic for the action of \(\Gamma\), and thus unique up to a constant.
The foundations of Patterson-Sullivan measures theory are in the important papers \cite{Pa} and \cite{Su} (see also \cite{Bou},\cite{BM} and \cite{Ro} for more general results in the context of CAT(-1) spaces). These measures are also called \textit{conformal densities}. When considering hyperbolic spaces  (and discrete groups acting on it), analogous constructions give rise to  \emph{quasi-conformal densities}.  For a non-elementary hyperbolic group  $\Gamma$ acting on a proper hyperbolic space, Coornaert   \cite[Th\'eor\`eme 8.3]{Co}  proves the existence of $\Gamma$-invariant quasi-conformal densities of dimension $Q $. This construction has been extended to the case of roughly geodesic metric spaces in \cite{BHM}.

\subsection{Some properties on the boundary}
\subsubsection{Upper Gromov bounded by above}
We say that a space $X$ is \emph{upper Gromov bounded by above} with respect to $o$, if there exists a constant $M_{X}>0$ such that for all $x\in X$,
 \begin{equation} \sup_{\xi \in \partial X}(\xi,x)_{o}\geq d(o,x)-M_{X}.
 \end{equation}
 This definition appears  in \cite{CM} and \cite{Ga} and is relevant when $(X,d)$ is not geodesic. In the setting of proper roughly geodesic $\epsilon$-good $\delta$-hyperbolic space on can take $M_{X}=\delta$: Let $\xi \in \partial X$ and pick $\eta \in B(\xi,r)=\{\eta'| d_{o,\epsilon}(\xi,\eta')\leq r \}$  with $r>0$ such that $\log(1/r)>\epsilon d(o,x)$. Thus, for all $\eta\in \partial X$ we have $(\xi,x)_{o}\geq \min \{(\xi,\eta)_{o},(\eta,x)_{o} \}-\delta\geq d(o,x)-\delta$ since $(\eta,x)_{o}\leq d(o,x)$.

  Hence, whenever $x\in X$ is given, we denote by $\hat{x}_{o}$ a point in the boundary satisfying 
 
 \begin{equation} \label{hat}
 (\hat{x}_{o},x)_{o}\geq d(o,x)-M_{X}.
 \end{equation}
  If $\gamma$ is an element of $\Gamma$, we denote by $\widehat{\gamma}_{o}$ the point in the boundary defined as  $\widehat{\gamma o}_{o}$ and $\check{\gamma}_{o}$ the point $\widehat{\gamma^{-1} o}_{o}$.

\subsubsection{Partitions of the boundary}
Let $(X,d)$ be a roughly geodesic, $\epsilon$-good, $\delta$-hyperbolic space
and pick $o \in X$ a base point. We  introduce a partition of the boundary as in \cite[Section 3.2]{Boy}: \\

 For $n\geq 2$, for $R>0$, for all $y\in S_{n,R}\subset X$  and for all $k\in \{0,\dots,n\}$,  define the Borel sets $A_{k,R}(o,y)$ of $\partial X$ as :
 
 \begin{equation}
 A_{0,R}(o,y):=\{  \xi\in\partial X |    (\xi,y)_{o}< R \},\;\;\; A_{n,R}(o,y):=\{  \xi\in\partial X |    nR\leq(\xi,y)_{o} \}
 \end{equation}
  and for $k\in \{1,\dots, n-1\}: $
\begin{equation}
A_{k,R}(o,y):=\{  \xi\in\partial X |    kR \leq(\xi,y)_{o}<(k+1)R \}.
\end{equation}

 In the following, we set $A_{k,R}(y )=A_{k,R}(o,y)$. For all $y\in X$,  the sets $(A_{k,R}(y))_{k =0,\dots,n}$ provide a partition of $\partial X$.
 
 If $y=\eta\in \partial X$, a similar partition  $(A_{k,R}(\eta))_{k\in \mathbb{N}}$ of $\partial  X\setminus \{\eta\}$ is defined by :
 \begin{equation}
 A_{0,R}(\eta):=\{  \xi\in\partial X |    (\xi,\eta)_{o}< R \}
 \end{equation}
  and for $k\in \mathbb{N}^{*}: $
\begin{equation}
A_{k,R}(\eta):=\{  \xi\in\partial X |    kR \leq(\xi,\eta)_{o}<(k+1)R \}.
\end{equation}

 \begin{lemma}\label{Ahlfors} There exist $R>0$ and $C\geq 1$ large enough such that  for all $y$ in $\overline{X}$  for all $k \in \mathbb{N}$: 
$$C^{-1}e^{-Q k R}\leq  \nu_{o}(A_{k,R}(y))\leq C e^{-Q k R}.$$
\end{lemma}
We refer to \cite[Lemma 3.5]{Boy} for a proof of the above fact.

\subsection{Equidistribution \`a la Roblin-Margulis} We refer to \cite[Theorem 3.2]{BG} for a proof of the following theorem inspired by Roblin's Theorem \cite[Th\'eor\`eme 4.1.1 ]{Ro}.
The  Dirac mass  at  $x\in X$ is denoted by $D_{x}$. 

\begin{theorem}\label{equi}
For any  $R>0$ large enough, there exists a sequence of  measures $\beta_{n,R}:\Gamma \rightarrow \mathbb{R}^{+}$ such that:
\begin{enumerate}
 \item There exists $C>0$ satisfying  for all $n \in \mathbb{N}$ and all $\gamma \in S^{\Gamma}_{n,R}$ that $$\beta_{n,R}(\gamma)\leq C /  |S^{\Gamma}_{n,R}|.$$
 \item 
  We have the following convergence: $$\sum_{\gamma \in S^{\Gamma}_{n,R}}\beta_{n,R}(\gamma) D_{\gamma o} \otimes D_{\gamma^{-1} o} \rightharpoonup  \nu_{o}\otimes \nu_{o},$$ 
as $n\to +\infty$, for the weak* topology in $C(\overline{X}\times \overline{X})$.
\end{enumerate}

\end{theorem}
\begin{coro}\label{equi'}
We deduce immediately from Theorem \ref{equi} the following equidistribution theorem:
$$\sum_{\gamma \in S^{\Gamma}_{n,R}}\beta_{n,R}(\gamma) D_{\gamma o}  \rightharpoonup  \nu_{o},$$ 
as $n\to +\infty$, for the weak* topology in $C(\overline{X})$.

\end{coro}
We also need the ``boundary version" of Theorem \ref{equi} (see \cite[Lemma 3.1]{BG}):
\begin{theorem}\label{equib}
For any $R>0$ large enough, there exists a sequence of measures $\beta_{n,R}:\Gamma \rightarrow \mathbb{R}^{+}$  such that :

\begin{enumerate}
 \item There exists $C>0$ satisfying  for all $n \in \mathbb{N}$ and all $\gamma \in S^{\Gamma}_{n,R}$ that $$\beta_{n,R}(\gamma)\leq C /  |S^{\Gamma}_{n,R}|.$$
 \item 
  We have the following convergence :$$\sum_{\gamma \in S^{\Gamma}_{n,R}}\beta_{n,R}(\gamma) D_{\hat{\gamma }_{o}} \otimes D_{\check{\gamma }_{o} } \rightharpoonup \nu_{o}\otimes \nu_{o},$$ 
as $n\to +\infty$, for the weak* convergence in $C(\partial X\times \partial X)$.
\end{enumerate}

\end{theorem}





\section{ The Riesz operators on the boundary of hyperbolic groups}\label{section3}
 
\subsection{The kernel $k_{t}$, the function $\sigma_{t}$.}
  
 To define the Riesz operators,  we consider for any $t\in \mathbb{R}$ the symmetric positive  (singular for $t<1/2$) kernel with values in 
 $\mathbb{R}^{+}\cup\{ +\infty\}$:
\begin{equation}\label{kernel}
k_{t}:(x,y)\in \overline{X} \times \overline{X}  \mapsto e^{(1-2t)Q(x,y)_{o} }.
\end{equation}

Note that for $(\xi,\eta)\in \partial X\times \partial X$ with $\xi\neq \eta$ we have

$$k_{t}(\xi,\eta)=\frac{1}{d^{(1-2t)D}_{o,\epsilon}(\xi,\eta)}.$$

For $A>0$, consider the open (resp. compact) subsets : $$U^{A}:={\{(x,y)\in \overline{X}\times \overline{X}|(x,y)_{o}<A \} },\;\;\; U_{A}:={\{(x,y)\in \overline{X}\times \overline{X}|(x,y)_{o}\geq A \} }$$ 
and the corresponding  truncatures of $k_{t}$  : $k^{A}_{t}=\textbf{1}_{U^{A}}.k_t$ and $k_{t,A}=\textbf{1}_{U_{A}}.k_t$.\\
Define $\sigma_{t}, \;  \sigma^{A}_{t},\; \sigma_{t,A}$ on the whole space $\overline{X}$ integrating $k_t$ (resp. $k_t^A$, $k_{t,A}$) :
\begin{equation}\label{sigma}
\sigma_{t}(x)=\int_{\partial X}k_t(x,\xi)d\nu_{o}(\xi),\;\; \sigma_{t}^A(x)=\int_{\partial X}k_t^A(x,\xi)d\nu_{o}(\xi),\;\; \sigma_{t,A}(x)=\int_{\partial X}k_{t,A}(x,\xi)d\nu_{o}(\xi)
\end{equation}
We will see that $\sigma_{t_{|_{X}}}$ is continuous on $X$ for any $t\in \mathbb{R}$. If $t>0$, we shall see in the following that $\sigma_{t}$ extends to a continuous function on $\partial X$ and thus extend to a continuous function on the whole space $\overline{X}$. \\


First, we start by an elementary lemma.

\begin{lemma}\label{cont1}
Let $t\in\mathbb{R}$. For all $A>0$, the function $\sigma^{A}_{t}$ is continuous on $\overline{X}$. \end{lemma}
\begin{proof}
For $A>0$ and $y\in  \overline{X}$, the map $\xi \in \partial X \mapsto k^{A}_{t}(\xi,y)$ is measurable  and integrable: for all $(\xi,y)\in \partial X \times \overline{X}$ we have $|k^{A}_{t}(\xi,y)|\leq e^{(1-2t)Q A}.$ Since we are working in the class of $\epsilon$-good spaces, the map $y\in \overline{X}\mapsto k^{A}_{t}(\xi,y)$ is continuous for all $\xi \in \partial X$. By continuous dependance on integral's Theorem, the function $y \in \overline{X}\mapsto \sigma^{A}_{t}(y)\in \mathbb{R}^{+}$ is continuous on $\overline{X}$.
\end{proof}

\begin{lemma}\label{Cauchy}
Let $t>0$. The function $\sigma_{t_{|_{ \partial X}}}$ is continuous on $\partial X$.\end{lemma}

\begin{proof}

Let $A_{N}=NR$ with $N\in\mathbb N^*$ and $R$ large enough so that Lemma \ref{Ahlfors} is valid, namely there exists $C\geq 1$ such that for all $\eta\in \partial X $ and for all $k\in \mathbb{N}$: $C^{-1}e^{-kQR}\leq \nu_{o}(A_{k,R}(\eta))\leq Ce^{-kQR}$. Consider the sequence $(\sigma^{A_N}_{t})_{N\in \mathbb{N}}$  in $C(\partial X)$. Use the partition  of $\partial X$ defined above to write for all $\eta\in \partial X$ and for $M >N:$

\begin{align*}
\sigma^{A_M}_{t}(\eta)-\sigma^{A_N}_{t}(\eta)&=\int_{ NR<(\xi,\eta)_{o}\leq  MR}e^{(1-2t)Q(\xi,\eta)_{o}} d\nu_{o}(\xi)\\ 
&\leq\sum_{k=N}^{M-1}\int_{\xi \in A_{k,R}(\eta)}e^{(1-2t)Q(\xi,\eta)_{o}} d\nu_{o}(\xi)\leq \sum^{M-1}_{k=N}\nu_{o}(A_{k,R}(\eta))e^{(1-2t)Q kR} \;\;(t>0)\\
&\leq C\sum^{M-1}_{k=N}e^{-2tQ kR}\leq \dfrac{Ce^{-2tQNR}}{1-e^{-2tQ R}}.
\end{align*}
 
Therefore $(\sigma^{A_N}_{t})_{N\in \mathbb{N}}$ is a Cauchy sequence in $(C(\partial X),\|\cdot\|_{\infty})$ complete. Since $\sigma^{A_N}_{t}\to \sigma_{t_{|_{\partial X}}}$, the latter is continuous (and  $k_{t}(.,\eta)$ is measurable and in $L^{1}(\partial X,\nu_{o})$ for every $t>0$).
\end{proof}

\begin{lemma}\label{limsup}
 Let $R>0$ large enough. There exists $C>0$, such that for all $t>0$, for all $\eta\in \partial X$ and for all $N\in\mathbb N^*$ and $A=NR$, we have: $$\limsup_{y\to \eta,\; y\in X}\sigma_{t,A}(y)\leq C\frac{e^{-2tQA}}{1-e^{-2tQR}}.$$ 
\end{lemma}
\begin{proof}
The proof is similar to that of the previous lemma :
set $A=NR$ (with $R$ as above), $M>N$ and $y\in S_{M,R}$ :

\begin{align*}
\sigma_{t,A}(y)&=\int_{ NR\leq (\xi,y)_{o}}e^{(1-2t)Q(\xi,y)_{o}} d\nu_{o}(\xi)\\ 
&=\sum_{k=N}^{M}\int_{\xi \in A_{k,R}(y)}e^{(1-2t)Q(\xi,y)_{o}} d\nu_{o}(\xi)\leq \sum^{M}_{k=N}\nu_{o}(A_{k,R}(y))e^{(1-2t)Q kR}\; (t>0)\\
&\leq C\sum^{M}_{k=N}e^{-2tQ kR}\leq C
\bigg(\dfrac{e^{-2tQNR}}{1-e^{-2tQ R}}\bigg),
\end{align*}
  
so that $\limsup_{y\to \eta}\sigma_{t,A}(y) \leq C
\bigg(\dfrac{e^{-2tQNR}}{1-e^{-2tQ R}}\bigg),$ and the conclusion follows.
\end{proof}

\begin{prop}\label{cont2}
Let $t>0$ ; the function $\sigma_{t}$ is continuous on $\overline{X}$, and moreover there exists $R>0$ and $C\geq 1$ such that for all $t>0$ and for all $\eta \in \partial X$ :

$$ C^{-1} \dfrac{1}{1-e^{-2tQ R}} \leq  \sigma_{t}(\eta)\leq C  \dfrac{1}{1-e^{-2tQ R}} \cdot$$

\end{prop}

\begin{proof} 

By Lemma \ref{Cauchy} $\sigma_{t_{|_{ \partial X}}}$ is continuous and
we also know that  $\sigma_{t_{|_{ X}}}$ is continuous.
So, it is only  to prove that $\sigma_{t}(y)\to \sigma_{t}(\eta)$ when $y\in X\to \eta\in \partial X$.
For $A_N=NR$ as above :
\begin{align*}
|\sigma_{t}(y)-\sigma_{t}(\eta)|&\leq |\sigma_{t}(y)-\sigma^{A_{N}}_{t}(y)|+|\sigma^{A_{N}}_{t}(y)-\sigma^{A_{N}}_{t}(\eta)|+|\sigma^{A_{N}}_{t}(\eta)-\sigma_{t}(\eta)|
\end{align*}
The first term on the right hand side is $ |\sigma_{t,A_{N}}(y)|$ and Lemma  \ref{limsup} implies $\limsup_{y\to \eta, \, y\in X} |\sigma_{t,A_{N}}(y)|\leq Ce^{-2tQNR}$ for some constant $C=C(R,t,Q)$. Continuity of $\sigma_t^{A_{N}}$ on $\overline{X}$ (Lemma  \ref{cont1}) implies $\limsup_{y\to \eta, \, y\in X} |\sigma^{A_{N}}_{t}(y)-\sigma^{A_{N}}_{t}(\eta)|=0$ and  uniform convergence of $\sigma^{A_N}$ on $\partial X$ together with the two previous estimates gives the desired conclusion when taking first  limsup and $N\to \infty$. Lastly estimates for $\sigma_t$ follow from the proof of Lemma \ref{Cauchy}.
\end{proof}


\begin{coro}\label{sigmat}
$\| \sigma_t\|_{\infty}.\| \sigma_t^{-1}\|_{\infty}\leq C^2$ where $C$ as above is independant of $t>0$.
\end{coro}

Analogous computations as in Lemma \ref{Cauchy} give :
\begin{lemma}\label{ineqsigma}

For $R$ large enough as above and $A_N=NR,\;  (N\in\mathbb N$), there exists $C\geq 1$ such that we have for all $t\in \mathbb{R}^*$, all $N\in\mathbb N^*$ and for all $\eta \in \partial X$ :

\begin{enumerate}
\item $$C^{-1} \frac{1-e^{-2tQNR}}{1-e^{-2tQR}} \leq \sigma^{A_N}_{t}(\eta) \leq C \frac{1-e^{-2tQNR}}{1-e^{-2tQR}},$$
\item $$C^{-1} \frac{e^{-2tQNR}}{1-e^{-2tQR}}\leq \sigma_{t,A_N}(\eta)\leq C  \frac{e^{-2tQNR}}{1-e^{-2tQR}}.$$\\
\end{enumerate}


\end{lemma}


  \begin{remark}
If $(X,d)$ is a rank one Riemmanian symmetric space, $\sigma_{t}$ is constant as it is $K$-invariant, where $K$  the maximal compact subgroup (i.e. the stabilizer of $o$) acts transitively on $\partial X$. \\
\end{remark}
We state a counting estimate related to $k_{t}$ that already appears in the context of CAT(-1) space in \cite{Bou}, that is the discrete version of Lemma \ref{limsup}. 
\begin{lemma}\label{ineqsigma2}
Let $R>0$ large enough and let  $(\beta_{n,R})_{n\in \mathbb{N}}$ be the sequence of weights provided by Theorem \ref{equi}.
 There exist a constant $C>0$ and $R>0$ large enough such that for all $A>0$
  and for any $h\in \Gamma$ we have for all $n$ large enough:
 $$\limsup_{n\to \infty}\sum_{\{g\in S_{n,R}|(go,ho)\geq A\}}\beta_{n,R}(g)k_{t}(ho,go)\leq C \frac{e^{-2tQA}}{1-e^{-2tQR}}.$$
 \end{lemma}
 \begin{proof}
 Set $R>0$ large enough.  Fix $h \in \Gamma$. For all $n\in \mathbb{N}$
and for all $k\in \mathbb{N}^*$, define $U_{k,R}(h) \subset S^{\Gamma}_{n,R}$ as : $$U_{k,R}(h):=\{ g\in S^{\Gamma}_{n,R}|(k-1)R<(go,ho)\leq kR  \}. $$ 
Set $A=NR$ with $R$ large enough. By Lemma 4.3 in \cite{Ga},  there exists $C>0$ satisfying for all $k\leq n$ : 
\begin{equation}\label{countestim}
|U_{k,R}(h)|\leq Ce^{(n-k)RQ},
\end{equation}
 and for $k\geq n$ : 
 
$$
|U_{k,R}(h)|\leq C.$$ 
Hence, we have for $n$ large enough :
\begin{align*}
\sum_{\{g\in S_{n}|(go,ho)\geq A\}}\beta_{n,R}(g)k_{t}(ho,go)&=\sum^{n}_{k=N}\sum_{g\in U_{k,R}(h) }\beta_{n,R}(g) e^{(1-2t)Q(go,ho)}\\
&\leq C\sum^{n}_{k=N}\big(\beta_{n,R}(g)\cdot |U_{k,R}(h)|\cdot e^{Q(1-2t)k}\big)\\
&\leq C\sum^{n}_{k=N}e^{-2tQRk},
\end{align*}
where the last inequality follows from (\ref{countestim}) together with Item 1 of Theorem \ref{equi}.
Thus, we have for all $A=NR>0$:
$$ \limsup_{n\to +\infty} \sum_{\{g\in S_{n}|(go,ho)\geq A\}}\beta_{n,R}(g)k_{t}(ho,go) \leq  C \frac{e^{-2tQA}}{1-e^{-2tQR}},$$
and the proof is done.
 \end{proof}
 

 \subsection{Definition and elementary properties of the  Riesz operator $\mathcal{R}_{t} $ for $t>0$.}
 Riesz potentials are  classical in standard geometric framework where  spectral analysis can be performed (see for instance \cite[Section 3.10]{Hei}, \cite[Chapter 6,  Section 2.1, Definition 6.13]{Sam} and  \cite[Ch. V, Section IV]{SU}).
 

In the present context, recall that the spherical Riesz potential of order $t>0$ is defined as :
\begin{equation}\label{Rieszpot}
\mathcal{R}_{t}(v)(\xi):=\frac{1}{\sigma_{t}(\xi)}\mathcal{I}_{t}(v)(\xi),\;\; {\rm with}\;\; \mathcal{I}_{t}(v)(\xi)=\int_{\partial X} \frac{v(\eta)}{d^{(1-2t)D}_{o,\epsilon}(\xi,\eta)}d\nu_{o}(\eta).
\end{equation}
We call  $\mathcal{I}_{t}$ the Knapp-Stein operator (associated with a metric $d_{o,\epsilon}$) in reference to \cite{KSt} together with Proposition \ref{intertwin}.

Multiplication by $\sigma_{t}$ ($t>0$) defines a bounded self-adjoint and invertible operator $M_{\sigma_{t}}$ on $L^{2}(\partial X,\nu_{o})$ with $\|M_{\sigma_{t}}\|_{L^{2}\to L^{2}}=\|\sigma_{t}\|_{\infty}$ and
 $M^{-1}_{\sigma_{t}}=M_{\sigma^{-1}_{t}}$.\\



The first properties of the intertwiner $\mathcal{I}_{t}$ and $\mathcal{R}_{t}=M_{\sigma^{-1}_{t}}\mathcal{I}_{t}$ for $t>0$  are in the next propositions.
  \begin{prop}\label{bounded}
 For $t>0$, $\mathcal{I}_{t}$ is bounded  self adjoint and $\mathcal{R}_{t}$ is bounded  : $$\|\mathcal{R}_{t}\|_{L^{2}\to L^{2}}\leq  \|\sigma_{t}\|_{\infty}\|\sigma^{-1}_{t}\|_{\infty} \mbox{ and }\|\mathcal{I}_{t}\|_{L^{2}\to L^{2}}\leq \|\sigma_{t}\|_{\infty}.$$ 
  \end{prop}
 \begin{proof}
 This is based on the so-called Schur's test. By Cauchy-Schwarz we obtain for almost every $\xi \in \partial X:$
\begin{align*}
|\mathcal{I}_{t}(v)(\xi)|&\leq \int_{\partial X}\dfrac{|v(\eta)|}{d_{o,\epsilon}^{(1-2t) D}(\xi,\eta)}d\nu_{o}(\eta) \\
&\leq \bigg(\int_{\partial X}\dfrac{1}{d_{o,\epsilon}^{(1-2t) D}(\xi,\eta)}d\nu_{o}(\eta) \bigg)^{\frac{1}{2}}\bigg(\int_{\partial X}\dfrac{|v^{2}(\eta)|}{d_{o,\epsilon}^{(1-2t) D}(\xi,\eta)}d\nu_{o}(\eta)\bigg)^{\frac{1}{2}}.
\end{align*}
Then, we have:
\begin{align*}
\|\mathcal{I}_{t}v\|^{2}_{2}&\leq \int_{\partial X} \sigma_{t}(\xi)\bigg(\int_{\partial X}\dfrac{|v^{2}(\eta)|}{d_{o,\epsilon}^{(1-2t) D}(\xi,\eta)}d\nu_{o}(\eta) \bigg)d\nu_{o}(\xi)\\
\mbox{Lemma \ref{Cauchy} }&\leq \|\sigma_{t}\|_{\infty} \int_{\partial X} \bigg(\int_{\partial X}\dfrac{|v^{2}(\eta)|}{d_{o,\epsilon}^{(1-2t) D}(\xi,\eta)}d\nu_{o}(\eta)\bigg) d\nu_{o}(\xi)\\
\mbox{by Fubini's Theorem }&\leq \|\sigma_{t}\|^{2}_{\infty}   \bigg(\int_{\partial X}|v^{2}(\eta)|d\nu_{o}(\eta) \bigg)\\
&=
\|\sigma_{t}\|^{2}_{\infty} \|v\|^{2}_{2}.
\end{align*}
Hence, we have $\|\mathcal{I}_{t}v\|_{2}\leq \|\sigma_{t}\|_{\infty} \|v\|_{2}$ and
\begin{align*}
\|\mathcal{R}_{t}\|_{L^{2}\to L^{2}}=\|M_{\sigma^{-1}_{t}}\mathcal{I}_{t}\|_{L^{2}\to L^{2}}=&\leq  \|\mathcal{I}_{t}\|_{L^{2}\to L^{2}} \|M_{\sigma^{-1}_{t}}\|_{L^{2}\to L^{2}}\\
&\leq \|\sigma_{t}\|_{\infty}\|\sigma^{-1}_{t}\|_{\infty}.
\end{align*}

Moreover, since $k_{t}$ is symmetric, Fubini's Theorem implies that $\mathcal{I}_{t}$ is self adjoint.
 \end{proof}

The following statement will also be useful :
 \begin{lemma}\label{positiv}
 Let $t>0$.
The Riesz operator $\mathcal{R}_{t}$ is positive if and only if the Knapp-Stein operator $\mathcal{I}_{t} $ is positive.
\end{lemma}
\begin{proof}
Assume that $\mathcal{R}_{t}$ is positive. Then, the operator $M_{\sigma^{1/2}_{t}}\mathcal{R}_{t}M_{\sigma^{1/2}_{t}}^{*}=M_{\sigma^{-1/2}_{t}}\mathcal{I}_{t}M_{\sigma^{1/2}_{t}}$ is positive and thus $sp(M_{\sigma^{-1/2}_{t}}\mathcal{I}_{t}M_{\sigma^{1/2}_{t}})\subset \mathbb{R}^{+}$, where $sp(\cdot)$ denotes the spectrum of an operator. Moreover $sp(M_{\sigma^{-1/2}_{t}}\mathcal{I}_{t}M_{\sigma^{1/2}_{t}})=sp(\mathcal{I}_{t}M_{\sigma^{1/2}_{t}}M_{\sigma^{-1/2}_{t}})=sp(\mathcal{I}_{t})$. Thus $sp(\mathcal{I}_{t})\subset \mathbb{R}^{+}$ and  the operator $\mathcal{I}_{t}$ is positive. \\
Assume that $\mathcal{I}_{t}$ is positive. Then, the operator $M_{\sigma^{-1/2}_{t}}\mathcal{I}_{t}M_{\sigma^{-1/2}_{t}}^{*}=M_{\sigma^{-1/2}_{t}}\mathcal{I}_{t}M_{\sigma^{-1/2}_{t}}$ is positive and thus $sp(M_{\sigma^{-1/2}_{t}}\mathcal{I}_{t}M_{\sigma^{-1/2}_{t}})\subset \mathbb{R}^{+}$. Then $sp(M_{\sigma^{-1/2}_{t}}\mathcal{I}_{t}M_{\sigma^{-1/2}_{t}})=sp(M_{\sigma^{-1}_{t}}\mathcal{I}_{t})$. Since $M_{\sigma^{-1}_{t}}\mathcal{I}_{t}=\mathcal{R}_{t}$,  thus the operator $\mathcal{R}_{t}$ is positive.
\end{proof}

  \subsection{Truncature of operators}
\subsubsection{Truncature of $\mathcal{I}_{t}$}
  Define for any $A>0$ and $t>0$ the bounded operators  acting on $ L^{2}(\partial X,\nu_{o})$ : 
 $$\mathcal{I}^{A}_{t}(v)(\xi):=\int_{\{\eta|(\xi,\eta)_{o}< A\}}\frac{v(\eta)}{d^{(1-2t)D}_{o,\epsilon}(\xi,\eta)}d\nu_{o}(\eta)$$
  and   $$ \mathcal{I}_{t,A}(v)(\xi):=(\mathcal{I}_{t}-\mathcal{I}^{A}_{t})(v)(\xi)=\int_{\{\eta|(\xi,\eta)_{o}\geq A\}}\frac{v(\eta)}{d^{(1-2t)D}_{o,\epsilon}(\xi,\eta)}d\nu_{o}(\eta).$$



\begin{lemma}\label{spectralgap}
For any $R>0$ large enough there exists $C>0$ such that for all $t>0:$ 
$$
\|\mathcal{I}^{A}_{t}\|_{L^{2}\to L^{2}}\leq C\cdot \frac{1-e^{-2tQA}}{1-e^{-2tQR}}\;\;\;{\rm  and }\;\;\;
 \|\mathcal{I}_{A,t}\|_{L^{2}\to L^{2}}\leq C\cdot \frac{e^{-2tQA}}{1-e^{-2tQR}}.$$ 
\end{lemma}
\begin{proof}
Following the proof of Proposition \ref{bounded}, we obtain for all $A>0:$
$$\|\mathcal{I}^{A}_{t}\|_{L^{2}\to L^{2}}\leq \|\sigma^{A}_{t}\|_{\infty}.$$
Use now Lemma \ref{ineqsigma} to estimate $\|\sigma^{A}_{t}\|_{\infty}$: for $R>0$ large enough there exists $C>0$ such that
$$\|\mathcal{I}^{A}_{t}\|_{L^{2}\to L^{2}}\leq C\cdot \frac{1-e^{-2tQA}}{1-e^{-2tQR}}.$$
The proof dealing with $\mathcal{I}_{A,t}$ is exactly the same.
\end{proof}

\begin{prop}\label{compactope}
For $t>0$, the operator $\mathcal{I}_{t}$ is compact 
from $L^{2}(\partial X)$ to $L^{2}(\partial X)$. \end{prop}
\begin{proof}

 Let us prove that  $\mathcal{I}_{t}$ is a compact operator from $L^{2}\to L^{2}$. Let $N>0$ and note that 
$$\mathcal{I}^{N}_{t}(v)(\xi)=\int _{\partial X} v(\eta)h_{N}(\xi,\eta)d\nu_{o}(\eta),$$
is a Hilbert-Schmidt operator since $h_{N}$ is in $L^{2}(\partial X \times \partial X,\nu_{o} \otimes \nu_{o})$ where :
$$
h_{N}(\xi,\eta) = \left\{
    \begin{array}{ll}
        \frac{1}{d^{(1-2t)D}_{o,\epsilon}(\xi,\eta)}& \mbox{if $(\xi,\eta)_{o}<N,$ }  \\
        0 & \mbox{if not.}
    \end{array}
\right.
$$
Hence, $\mathcal{I}^{N}_{t}$ is compact for all $N$ and moreover, Lemma \ref{spectralgap} implies that
$$\|\mathcal{I}_{t}-\mathcal{I}^{N}_{t}\|_{L^{2}\to L^{2}}=\|\mathcal{I}_{N,t}\|_{L^{2}\to L^{2}} \to 0,$$ as $N\to +\infty$. So, $\mathcal{I}_{t}$ is compact.\\

\end{proof}

\subsection{A one parameter family of operators}

Set for $v\in L^{2}(\partial X)$ :
$$(\delta_{t}(v))(\xi,\eta):=\frac{1}{d_{o,\epsilon}^{(\frac{1}{2}-t) D} (\xi,\eta)}\big( v(\xi)-v(\eta)\big).$$
The purpose of this section is to define a Laplacian related to the Riesz operator $\mathcal{R}_{t}$ for each $t>0$.

\begin{prop}\label{gdelta} Let $t>0$.
 \begin{enumerate}
 \item \label{1'''}The operator $\delta_{t}$ is a well defined bounded operator from $L^{2}(\partial X,\nu_{o})$ to $L^{2}(\partial X \times \partial X,\nu_{o} \otimes \nu_{o})$ whose adjoint $\delta^{*}_{t}:L^{2}(\partial X \times \partial X,\nu_{o} \otimes \nu_{o})\rightarrow L^{2}(\partial X,\nu_{o})$ is defined by the following formula:   $$ \delta^{*}_{t}F(\cdot):= \int_{\partial X}\frac{ F(\cdot,\eta)-F(\eta,\cdot)}{d^{(\frac{1}{2}-t)D}_{o,\epsilon}(\cdot,\eta)}d\nu_{o}(\eta).$$
  \item \label{2'''} The operator $\Delta_{t}:=\frac{\delta_{t}^{*}\delta_{t}}{2} $ is a bounded and positive operator acting on $L^{2}(\partial X,\nu_{o})$  defined by   : $$ \Delta_{t}(v)(\xi)=\int_{\partial X}\frac{  v(\xi)-v(\eta) }{d_{o,\epsilon}^{(1-2t)D} (\xi,\eta)}d\nu_{o}(\eta)\;\;\;\; (v\in L^{2}(\partial X,\nu_{o}),\;\; \xi\in \partial X).$$
 \item \label{3'''} We have: $ \mathcal{I}_{t}=M_{\sigma_{t}}-\Delta_{t}.$
  
\end{enumerate}

    \end{prop}

  \begin{proof} We prove (\ref{1'''}). We have

  $$\begin{array}{l}
  \|\delta_{t}v\|^{2}_{L^{2}\otimes L^{2}}=\displaystyle\int_{\partial X \times \partial X}\frac{|v(\xi)-v(\eta)|^{2}}{d^{(1-2t)D}_{o,\epsilon}(\xi,\eta)}d\nu_{o}(\xi)d\nu_{o}(\eta)\\
  \\
 \;\;\;\;\;\;\;\hspace{1cm}\displaystyle \leq \int_{\partial X \times \partial X}\frac{|v(\xi)|^{2}}{d^{(1-2t)D}_{o,\epsilon}(\xi,\eta)}d\nu_{o}(\xi)d\nu_{o}(\eta)+\int_{\partial X \times \partial X}\frac{|v(\eta)|^{2}}{d^{(1-2t)D}_{o,\epsilon}(\xi,\eta)}d\nu_{o}(\xi)d\nu_{o}(\eta)\\
 \hspace{1,8cm}\displaystyle +2\bigg(\int_{\partial X \times \partial X}\frac{|v(\xi)|^{2}}{d^{(1-2t)D}_{o,\epsilon}(\xi,\eta)}d\nu_{o}(\xi)d\nu_{o}(\eta)\bigg)^{\frac{1}{2}} \bigg(\int_{\partial X \times \partial X}\frac{|v(\eta)|^{2}}{d^{(1-2t)D}_{o,\epsilon}(\xi,\eta)}d\nu_{o}(\xi)d\nu_{o}(\eta)\bigg)^{\frac{1}{2}}  \\
 \\
\hspace{2cm} \displaystyle =4\int_{\partial X }|v(\xi)|^{2}\sigma_{t}(\xi)d\nu_{o}(\xi) \;\;\;  \mbox{  (Fubini Theorem) } \\
 \\
 \hspace{2cm}\displaystyle \leq 4\|\sigma_{t}\|_{\infty}\|v\|^{2}_{2},
 \end{array}$$

\noindent   hence $\delta_{t}$ is bounded from $L^{2}(\partial X,\nu_{o})$ to $L^{2}(\partial X \times \partial X,\nu_{o} \otimes \nu_{o})$.
Now observe for all $v\in L^{2}(\partial X,\nu_{o})$ and $F\in L^{2}(\partial X \times \partial X, \nu_{o}\otimes \nu_{o})$: 
  \begin{align*}
  \langle \delta_{t}v,F \rangle &=\int_{\partial X \times \partial X}\frac{(v(\xi)-v(\eta)) \overline{F}(\xi,\eta)}{d^{(\frac{1}{2}-t)D}_{o,\epsilon}(\xi,\eta)}d\nu_{o}(\xi)d\nu_{o}(\eta)\\
  &=\int_{\partial X \times \partial X}\frac{v(\xi) \overline{F}(\xi,\eta)}{d^{(\frac{1}{2}-t)D}_{o,\epsilon}(\xi,\eta)}d\nu_{o}(\xi)d\nu_{o}(\eta)-\int_{\partial X \times \partial X}\frac{v(\eta) \overline{F}(\xi,\eta)}{d^{(\frac{1}{2}-t)D}_{o,\epsilon}(\xi,\eta)}d\nu_{o}(\xi)d\nu_{o}(\eta)\\
  &=\int_{\partial X}v(\xi)\bigg(\int_{\partial X}\frac{ \overline{F}(\xi,\eta)}{d^{(\frac{1}{2}-t)D}_{o,\epsilon}(\xi,\eta)}d\nu_{o}(\eta) \bigg)d\nu_{o}(\xi)-\int_{\partial X }v(\eta)\bigg(\int_{\partial X}\frac{ \overline{F}(\xi,\eta)}{d^{(\frac{1}{2}-t)D}_{o,\epsilon}(\xi,\eta)}d\nu_{o}(\xi)\bigg)d\nu_{o}(\eta)\\
  &=\int_{\partial X}v(\xi)\bigg(\int_{\partial X}\frac{ \overline{F}(\xi,\eta)-\overline{F}(\eta,\xi)}{d^{(\frac{1}{2}-t)D}_{o,\epsilon}(\xi,\eta)}d\nu_{o}(\eta) \bigg)d\nu_{o}(\xi)\\
  &=\langle v,\delta^{*}_{t}F\rangle \\
  \end{align*}
  where for  almost all $\xi\in \partial X$ $$\delta^{*}_{t}F(\xi):= \bigg(\int_{\partial X}\frac{ F(\xi,\eta)-F(\eta,\xi)}{d^{(\frac{1}{2}-t)D}_{o,\epsilon}(\xi,\eta)}d\nu_{o}(\eta) \bigg),$$
  and (\ref{1'''}) is proved.\\
  Expression (\ref{2'''}) is straightforward from that of  $\delta_t$ and  $\delta^{*}_{t}$.
  
As for (\ref{3'''}), write for $v\in L^{2}(\partial X,\nu_{o})$:
\begin{align*}
\mathcal{I}_{t}(v)(\xi)&=\int_{\partial X}\frac{v(\eta)}{d^{(1-2t)D}_{o,\epsilon}(\xi,\eta)}d\nu_{o}(\eta)  \\
&=\int_{\partial X}\frac{v(\eta)-v(\xi)}{d^{(1-2t)D}_{o,\epsilon}(\xi,\eta)}d\nu_{o}(\eta)+\int_{\partial X}\frac{v(\xi)}{d^{(1-2t)D}_{o,\epsilon}(\xi,\eta)}d\nu_{o}(\eta) \\
&=-\Delta_{t}(v)(\xi)+\sigma_{t}(\xi)v(\xi),
\end{align*}
  and the proof is complete.
  \end{proof}
  
  We close this section by a characterization of the positivity of $\mathcal{R}_{t}$ is terms of spectral radius of $\Delta_{t}$. Consider
   \begin{equation}\label{deltanormalise}
  \widetilde{\Delta_{t}}:=M_{\sigma^{-1}_{t}}\Delta_{t}.
  \end{equation}


 A spectrum analysis of ${\mathcal I}_t$ is hopeless for general hyperbolic groups and general hyperbolic metric spaces $(X,d)$.
Nevertheless, we obtain the following :
\begin{prop} We have:
\begin{enumerate}
\item \label{1''''}The eigenspace $\lbrace v\in L^{2}(\partial X,\nu_{o})|  \mathcal{R}_{t}v=v \rbrace$ is nothing but the space of constant functions.
\item  \label{2''''}The operator $\mathcal{R}_{t}$ is positive if and only if $\|\widetilde{\Delta_{t}}\|_{L^{2}\to L^{2}}\leq 1$.
\end{enumerate}
\end{prop}
\begin{proof}

We prove (\ref{1''''}). Indeed pick $v\in L^{2}(\partial X, \nu_{o})$ such that $\mathcal{R}_{t}v=v$. It follows 
by (\ref{3'''}) of Proposition \ref{gdelta} that $M_{\sigma^{-1}_{t}}\Delta_{t}v=0$. Thus $\Delta_{t}v=0$. Therefore $\langle \Delta_{t} v,v \rangle=0$, which writes  :

$$ \int_{\partial X \times \partial X} \frac{|v(\xi)-v(\eta)|^{2}}{d^{(1-2t)D}_{o,\epsilon}(\xi,\eta)}d\nu_{o}(\xi)d\nu_{o}(\eta)=0.$$ Hence, for almost every $(\xi,\eta)\in \partial X \times \partial X$ we have:

$$|v(\xi)-v(\eta)|=0.$$ Therefore, $v$ has to be constant.\\

The proof of  (\ref{2''''}) follows from (\ref{3'''}) of  Proposition \ref{gdelta}.\\ If $\mathcal{R}_{t}$ is positive, then for all $v \in L^{2}(\partial X,\nu_{o})$, $\langle \mathcal{R}_{t}v,v\rangle=\|v\|_{2}^{2}-\langle \widetilde{\Delta_{t}}v,v \rangle \geq 0.$ Since $\widetilde{\Delta_{t}}$ is a positive operator, we have $\|\widetilde{\Delta_{t}}\|_{L^{2}\to L^{2}}\leq 1.$ \\
Now, if $\|\widetilde{\Delta_{t}}\|_{L^{2}\to L^{2}}\leq 1$ then for all   $v \in L^{2}(\partial X,\nu_{o})$, $$\langle \mathcal{R}_{t}v,v\rangle=\|v\|_{2}^{2}-\langle \widetilde{\Delta_{t}}v,v \rangle \geq \|v\|_{2}^{2}(1- \|\widetilde{\Delta_{t}}\|_{L^{2}\to L^{2}})\geq 0,$$ and the proof is done.

\end{proof}

  \begin{lemma} With respect to the weak operator topology, $\mathcal{R}_{t}\to Id$ when $t\to 0$.\end{lemma}

\begin{proof}
It is only to prove that $ \langle \widetilde{\Delta_{t}} v ,w\rangle\to 0$ for $v,w$ in the dense subspace of Lipschitz functions since $\sup_{t>0}\|\mathcal{R}_{t}\|_{L^{2}\to L^{2}}<+\infty$ (Proposition \ref{bounded} together with Corollary \ref{sigmat}). For such functions,  
\begin{align*}
| \langle \widetilde{\Delta_{t}} v ,w\rangle |&\leq C\|v\|_{Lip}\|w\|_{Lip} \big(1-e^{-2tQR}\big) \bigg(\int_{\partial X \times \partial X}\frac{d^{2}_{o,\epsilon}(\xi,\eta)}{d^{(1-2t)D}_{o,\epsilon}(\xi,\eta)}d\nu_{o}(\xi)d\nu_{o}(\eta)\bigg)\\
 &\leq C \|v\|_{Lip}\|w\|_{Lip} \big(1-e^{-2tQR}\big) \bigg(\frac{1}{1-e^{(-2tQ -2\epsilon )R}}\bigg),
 \end{align*} 
 where the second inequality is based on the method of computations of Lemma \ref{Cauchy}. Eventually, let $t\to 0$ to conclude the proof.
\end{proof}

\subsection{Inverse of $\mathcal{I}_{t}$ (and $\mathcal{R}_{t}$)}
Finding a formula for the inverse of Riesz potentials is a broad subject (see \cite[Section 3, Ch. 6]{Sam}). In our setting, nothing is known about kernel and image of ${\mathcal R}_t$. We first prove an abstract Lemma (probably well known from the specialists) as we were not able to find it in the literature, and apply this result to operators ${\mathcal I}_t$.\\
\begin{lemma}\label{inverse}
Let $T$ be a self adjoint compact operator acting on a Hilbert space $H$ with infinite rank. There exists an unbounded closed self adjoint operator $S:\overline{Im(T)}\to H$  whose domain is $\mbox{Dom}(S)=Im(T)$ such that $S$ satisfies $ST(u)=u$ for any $u\in H\ominus \ker(T)$ and $TS(v)=v$ for any $v\in Im(T).$ Moreover,  if $T$ is a positive operator, we can write $S=c\cdot I+D$ where $D:\overline{Im(T)}\to H$ is an unbounded  closed positive self adjoint operator and $c$ a positive constant. 

 \end{lemma}
  
  \begin{proof}
    
 Since $T$ is a compact self adjoint operator, consider a system of orthonormal vectors $(v_{i})$ associated with corresponding real eigenvalues $\big(\lambda_{i}(T)\big)_{i}$ such that $|\lambda_{1}(T)|\geq |\lambda_{2}(T)|\geq \dots  $  with $\lambda_n(T)\to 0$ as $n\to \infty$. Write  $Tu=\sum^{+\infty}_{i=1}\lambda_{i}(T) \langle u,v_{i}\rangle v_{i}$ for all $u\in H$ with respect the $\|\cdot\|_{H}$-convergence. For all $n$, define the sequence of bounded operators $S_{n}$ as :
$$S_{n}(\cdot):=\sum_{k=1}^{n}\frac{1}{\lambda_{k}(T)}\langle \cdot,v_{k}\rangle v_{k}.$$
 We prove first that for all $w\in Im(T)$, $S_{n}(w)$ converges with respect to  $\|\cdot\|_{H}$. 
Let $w\in Im(T)$. There exists a unique $u\in H\ominus \ker(T)$ such that $w=T(u)$. Therefore, for all $n:$
$$
S_{n}(w)=S_{n}T(u)
=S_{n}(\sum^{+\infty}_{i=1}\lambda_{i}(T) \langle u,v_{i}\rangle v_{i})
=\sum_{k=1}^{n}\langle u,v_{k}\rangle v_{k}.$$
Since $u\in H$, we have the following $\|\cdot\|_{H}$-convergence: $S_{n}(w)\to u$. Thus, for all $w\in Im(T)$ define $$S(w):=\lim_{n\to +\infty}S_{n}(w).$$ For all $u \in H\ominus \ker(T)$, we have $ST(u)=u$.
  By construction of $S$, we have $Im(T)= \mbox{Dom}(S)$ and obviously $\mbox{Dom}(S)$ is dense in $\overline{Im(T)}.$\\
Since $\lambda_{i}(T) $ goes to $0$ the operator $S$ is unbounded on $H$.
Furthermore $S$ is closed: let $w_{p} \in Im(T)$ converges strongly to $w$ and $S(w_{p})$ converges strongly to $z$. Then, there exists a  unique $u_{p}\in H\ominus \ker(T)$ so that $w_{p}=T(u_{p})$. Thus, $Sw_{p}=u_{p}$. So, $u_{p}$ converges strongly to $z$. Since $T$ is continuous we have $T(u_{p})\to T(z)=w$, so  $w\in Im(T)=\mbox{Dom}(S)$ and $Sw=z$.\\
We prove the second assertion. Assume now that $T$ is positive and assume also that $\|T\|_{H \to H}\leq 1$. Consider $p_{n}$ the orthogonal projection onto the space $\bigoplus^{n}_{k=1} \ker(T-\lambda_{k}I)$. The definition of $S_{n}$ reads as follows :
 $$S_{n}(\cdot):=p_{n}(\cdot)+\sum_{k=1}^{n}\bigg(\frac{1}{\lambda_{k}(T)}-1\bigg)\langle \cdot,v_{k}\rangle v_{k},$$
and since $0\leq\lambda_{k}(T)\leq 1$, it follows that the sequence of operators $$D_{n}(\cdot)=\sum_{k=1}^{n}\big(\frac{1}{\lambda_{k}(T)}-1\big)\langle \cdot,v_{k}\rangle v_{k}$$ is a sequence of positive operators. Besides, $p_{n}$ converges strongly to the identity operator. Since $S$ is defined as the limit of $S_{n}$ on $Im(T)$, we define for all $w\in Im(T)$, $$D(w):=\lim_{n\to +\infty}D_{n}(w)=\lim_{n\to +\infty}S_{n}(w)-p_{n}(w).$$ Note that $D$ is an unbounded positive closed operator. By construction, we have $S=I+D$ on $Im(T)$. \\
If $\|T\|_{H \to H} $ is not bounded by $1$, perform the above construction with $ \dfrac{T}{ \|T\|_{H \to H}}$ and thus $S= c\cdot I+D$ where $D$ is a unbounded positive closed operator on $Im(T)$ and $c=\|T\|_{H \to H}>0$. And the proof is done.\\

   \end{proof}

 \subsection{Riesz potentials and intertwiners}
If $\Gamma\subset SU(1,1)\simeq Isom^+(\mathbb D^2)$ is a uniform lattice acting on  the Poincar\'e disk $\mathbb D^2$, the operator $\mathcal{I}_t$ (proportional to ${\mathcal R}_t$)  reads :
\begin{equation}
\mathcal{I}_{t}(v)(\omega)=\frac{1}{2\pi}\int _{[0,2\pi]}\frac{v(\theta)}{\big(1-\cos(\theta-\omega)\big)^{\frac{1}{2} -t}}d\theta,\;\;\; (v\in L^{2}(\partial \mathbb{D},\frac{d\theta}{2\pi})).
\end{equation}
 From that expression, it can be derived  positivity of ${\mathcal I}_t$ for all $t\in [0,{1\over 2}]$ and a straightforward calculus gives intertwining of complementary series representations of $SU(1,1)$, that is $\mathcal{I}_{t}\pi_t=\pi_{-t}\mathcal{I}_t$, (see  \cite[Ch. V, Section IV]{SU}).

The same intertwining relation occurs in our context when assuming strong hyperbolicity of $(X,d)$. Namely, we get


 

 
 \begin{prop}\label{intertwin}
For any $t>0$, we have for all $\gamma \in \Gamma$,
$\mathcal{I}_{t}\pi_{t}(\gamma)=\pi_{-t}(\gamma)\mathcal{I}_{t}$ on $L^{2}(\partial X, \nu_{o})$.
\end{prop}
 
 \begin{proof}

Let $v\in L^{2}(\partial X, \nu_{o})$ and write $\pi_{t}(\gamma)v=e^{ (\frac{1}{2}+t)Q\beta_{\xi}(o,\gamma o)}v(\gamma^{-1}\xi)$. We have for all real numbers $s>0$ and $t\in \mathbb{R}$ and for almost all $\eta \in \partial X$
 \begin{align*}
(\mathcal{I}_{s}\pi_{t}(\gamma)v)(\eta)&=\int_{\partial X}\frac{e^{ (\frac{1}{2}+t)Q\beta_{\xi}(o,\gamma o)}v(\gamma^{-1}\xi)}{d_{ o,\epsilon}^{(1-2s)D}(\xi, \eta)} d\nu_{o}(\xi)\\
&=\int_{\partial X}\frac{v(\gamma^{-1}\xi)}{d_{ o,\epsilon}^{(1-2s)D}(\xi, \eta)} e^{ (\frac{1}{2}+t)Q\beta_{\gamma^{-1} \xi}( \gamma^{-1} o, o)}d\nu_{o}(\xi)\\
&=\int_{\partial X}\frac{v(\gamma^{-1}\xi)}{d_{ o,\epsilon}^{(1-2s)D}(\xi, \eta)} e^{ -(\frac{1}{2}+t)Q\beta_{\gamma^{-1} \xi}( o,\gamma^{-1} o)}d\nu_{o}(\xi)\\
&=\int_{\partial X}\frac{v(\xi)}{d_{ o,\epsilon}^{(1-2s)D}(\gamma \xi, \eta)} e^{ (\frac{1}{2}-t)Q\beta_{ \xi}(o, \gamma^{-1} o)}d\nu_{o}(\xi)\\
\mbox{Thanks to (\ref{visualrough}) and (\ref{DQ}) }&=\int_{\partial X}\frac{v(\xi)e^{ (\frac{1}{2}-t)Q\beta_{ \xi}(o, \gamma^{-1} o)}}{e^{(\frac{1}{2}-s)Q (\beta_{ \xi}(o,\gamma^{-1} o )+\beta_{\gamma^{-1} \eta}(o,\gamma^{-1} o ))}d_{ o,\epsilon}^{(1-2s)D}( \xi, \gamma^{-1}\eta)}d\nu_{o}(\xi)\\
&=e^{(\frac{1}{2}-s)Q\beta_{\eta}(o,\gamma o)} \int_{\partial X}\frac{v(\xi)e^{ (s-t)Q\beta_{ \xi}(o, \gamma^{-1} o)}}{d_{ o,\epsilon}^{(1-2s)D}( \xi, \gamma^{-1} \eta)}d\nu_{o}(\xi)\\
&=\pi_{-s}(\gamma)\mathcal{I}_{s} \big(M_{e^{ (s-t)Q\beta_{ \cdot}(o, \gamma^{-1} o)}}v\big)(\eta).
\end{align*}

 
 Then the proof is done by taking $s=t$.
 \end{proof}

\section{Boundary Dual system representations and Spherical functions}\label{section4}
\subsection{Dual system representations}

A $\mathbb{C}$-pairing is a pair of $\mathbb C$-vector spaces $(H_{1},H_{2})$ together with a non-degenerate bilinear form $\langle\cdot,\cdot \rangle : H_{1}\times H_{2}\rightarrow \mathbb{C}$.
 Such bilinear form induces weak topologies $\sigma(H_1,H_2)$ and $\sigma(H_2,H_1)$ on $H_1$ and $H_2$ respectively  generated by linear forms fixing one variable. We refer  to \cite[Section 8]{Jar} for a general discussion.

 A \emph{dual system representation of $\Gamma$} (called rather a ``linear system representation" in \cite{Fe}) is the data of:
\begin{enumerate}
\item A $\mathbb{C}$-pairing of Banach spaces : $(H_{1},H_{2},\langle\cdot ,\cdot\rangle).$

\item Two representations $(\pi_{1},H_{1})$ and $(\pi_{2},H_{2})$ of $\Gamma$ such that $\langle \pi_{1}(\gamma)v,\pi_{2}(\gamma) w \rangle=\langle v, w \rangle$ for all $(v,w)\in H_{1}\times H_{2},$ for all $\gamma \in \Gamma$.
 
\end{enumerate}
We refer to \cite[Section 2]{Fe} for dual systems and we adapt this notion here for a group rather than for an algebra.\\ 
\begin{example}
If $H_{2}=H^{*}_{1}$, the representation $\pi_{2}$ is nothing but the contragredient representation of $\pi_{1}$ for the natural dual pairing.\textbf{We say that $(H_{1},H_{2},\langle\cdot ,\cdot\rangle)$ is the natural dual pairing associated with $H_{1}$}.
\end{example}
For a $\mathbb C$-vector space $V$, we denote in what follows by $\overline{V}$ its conjugate, that is multiplication by scalars on $\overline{V}$ is given by $\lambda \cdot   v:= \bar \lambda v, \; (v\in V)$.\\
\begin{example}\label{dual}Consider the $\mathbb{C}$-pairing $(L^{2}(\partial X,\nu_{o}),\overline{L^{2}(\partial X,\nu_{o})})$ with respect to the $L^{2}$-inner product.
 The conjugate representation $\overline{\pi_{-t}}$ satisfies: $\overline{\pi_{-t}}(\gamma)(\lambda \cdot v)=\lambda \cdot \overline{\pi_{-t}}(\gamma)( v),$
 so that  representations $(\pi_{t},L^{2}(\partial X,\nu_{o}))$ and $(\overline{\pi_{-t}},\overline{L^{2}(\partial X,\nu_{o}))}$ are a dual linear system of $\Gamma$. 
 \end{example}
 

We say that a dual system representations $(\pi_{1},H_{1})$ and $(\pi_{2},H_{2})$ of $\Gamma$ is irreducible if $(\pi_{1},H_{1})$ has no proper $\sigma(H_1,H_2)$-closed subspace stable by $\pi_{1}$, \emph{and} $(\pi_{2},H_{2})$ has no  proper $\sigma(H_2,H_1)$-closed subspace stable by $\pi_{2}$. Since norm closure and weak closure of linear spaces are the same, the weak closure in the above definitions can be replaced by norme closure. Observe that in the case of a natural dual pairing of a reflexiv Banach space then a dual system representations $(H_{1},H_{2})$ of $\Gamma$ is irreducible if $(\pi_{1},H_{1})$ has no proper closed invariant subspace {\emph{equivalently}} $(\pi_{2},H_{2})$ has no proper closed invariant subspace.
\\

The aim of this section is to construct three dual system representations of $\Gamma$ for which we will prove irreducibility. \textbf{They are all natural dual system associated with some Hilbert spaces.}


 \subsection{Linear representations}
 We fix $t>0$. A priori the operator $\mathcal{I}_{t}$ is not injective, thus consider $N_{t}:=\{ v\in \mathcal{F}_{t,\textbf{1}_{\partial X}}| \mathcal{I}_{t}(v)=0\}$ that is a subrepresentation of $(\pi_{t}, \mathcal{F}_{t,\textbf{1}_{\partial X}})$.
 The algebraic representation defined in (\ref{algrepre}) provides another algebraic representation of $\Gamma$ via $\pi_{-t}$, defined as follows :

$$ \begin{tikzpicture}\label{isoalg}
  \matrix (m)
    [
      matrix of math nodes,
      row sep    = 3em,
      column sep = 4em
    ]
    {
     \mathcal{F}_{t,\textbf{1}_{\partial X}}              & \mathcal{I}_{t}( \mathcal{F}_{t,\textbf{1}_{\partial X}}) \\
 \mathcal{F}_{t, \textbf{1}_{\partial X}}  / N_{t}&             \\
    };
  \path
    (m-1-1) edge [->>] node [left] {$p$} (m-2-1)
    (m-1-1.east |- m-1-2)
      edge [->] node [above] {$\mathcal{I}_{t}$} (m-1-2)
    (m-2-1.east) edge [->] node [below] {$[\mathcal{I}_{t} ]$} (m-1-2);
\end{tikzpicture} $$

 \begin{equation}
\mathcal{F}_{-t,\sigma_{t}}:=[\mathcal{I}_{t}](\mathcal{F}_{t,\textbf{1}_{\partial X}}/N_{t})=\mathcal{I}_{t}(\mathcal{F}_{t,\textbf{1}_{\partial X}}) \subset L^{2}(\partial X ,\nu_{o}) .
\end{equation}
Hence $\mathcal{F}_{t,\textbf{1}_{\partial X}}/N_{t}$ and $\mathcal{F}_{-t,\sigma_{t}}$ are isomorphic as vector spaces.
Any element $v \in \mathcal{F}_{-t,\sigma_{t}}$ can be represented by $v=\mathcal{I}_{t}(x)$ where $x\in \mathcal{F}_{t,\textbf{1}_{\partial X}}$ an unique element modulo $N_{t}$.
 Since $\mathcal{I}_{t}$ intertwines $\pi_{t}$ and $\pi_{-t}$ on $L^{2}(\partial X,\nu_{o})$, we obtain a representation given by $(\pi_{-t},\mathcal{F}_{-t,\sigma_{t}})$.\\

 \subsection{A first Hilbertian dual system representations of $\Gamma$ associated with $\mathcal{I}_{t}$} Associated to the intertwiner $\mathcal{I}_{t}$, the Hilbert space $L^{2}(\partial X,\nu_{o})$ decomposes as:

\begin{equation}
L^{2}(\partial X,\nu_{o})=\ker(\mathcal{I}_{t})\oplus \overline{Im(\mathcal{I}_{t})}^{\|\cdot\|_{2}}.
\end{equation}

The closed subspace $\mathcal{V}_{t}:=\ker(\mathcal{I}_{t})$ of $L^{2}(\partial X,\nu_{o})$ is a $\pi_{t}$ subrepresentation of $\Gamma$ although $\mathcal{W}_{t}:=\overline{Im(\mathcal{I}_{t})}^{\|\cdot\|_{2}}$ is a $\pi_{-t}$ subrepresentation.  Hence, if $\mathcal{I}_{t}$ is not injective, then neither $(\pi_{t},L^{2}(\partial X ,\nu_{o}))$ nor $(\overline{\pi_{-t}},\overline{L^{2}(\partial X,\nu_{o})})$ is irreducible. Consider the representations  $(\pi_{t},L^{2}(\partial X ,\nu_{o})/\mathcal{V}_{t})$ and $(\overline{\pi_{-t}},\overline{\mathcal{W}_{t}})$ and  notice that it is an example of a dual system representations of $\Gamma$ with respect to the standard $L^{2}$-inner product.

\begin{example}
The dual space $L^{2}(\partial X,\nu_{o})/\mathcal{V}_{t}$ is nothing but $\overline{\mathcal{W}_{t}}$. Hence the contragredient representation of $(\pi_{t},L^{2}(\partial X,\nu_{o})/\mathcal{V}_{t})$ with respect to the standard $L^{2}$-inner product is $(\overline{\pi}_{-t},\overline{\mathcal{W}_{t}})$. Note that if $\mathcal{I}_{t}$ is injective, we recover Example \ref{dual}.
 \end{example}

\begin{remark}
 The space $\mathcal{F}_{t,\textbf{1}_{ \partial X}}$ endowed with $\langle \cdot,\cdot\rangle$ is a pre-Hilbert space. The representation \begin{equation}\label{dense}
\overline{\mathcal{F}_{t,\textbf{1}_{ \partial X}}}^{\|\cdot\|_{2}}=\mathcal{L}_{t},
\end{equation}
 is well defined since there exists $C>0$ such that for all $\gamma\in \Gamma$, for all $v\in\mathcal{L}_{t}$, we have: $$ \|\pi_{t}(\gamma)v\|_{2}\leq Ce^{Qt|\gamma|}\|v\|_{2}.$$
\end{remark}


\subsection{A non unitary Hilbertian analog of complementary series}

In the following, we will make use that $\mathcal{L}_{t}=L^{2}(\partial X,\nu_{o})$ for $t>0$ and postpone  the proof   to Section \ref{section6} (Proposition \ref{density}).\\

 Consider the sesquilinear form $$\langle \cdot ,\cdot \rangle_{\mathcal{K}_{t}}:(v,w)\in \mathcal{F}_{t,\textbf{1}_{ \partial X}}\times \mathcal{F}_{t,\textbf{1}_{ \partial X}}\mapsto \langle     v, \mathcal{K}_{t}(w)\rangle.$$ 
Turn first $\mathcal{F}_{t,\textbf{1}_{ \partial X}}$ into a pre-Hilbert space. Write $N_{t}=\mathcal{V}_{t}\cap \mathcal{F}_{t,\textbf{1}_{ \partial X}}$. Observe that one can equip  $\mathcal{F}_{t,\textbf{1}_{ \partial X}}/ N_{t}$  with $\langle\cdot ,\cdot \rangle_{\mathcal{K}_{t}}$. Hence,    $(\mathcal{F}_{t,\textbf{1}_{ \partial X}}/ N_{t},\langle\cdot ,\cdot \rangle_{\mathcal{K}_{t}})$ becomes a pre-Hilbert space. Eventually consider the Hilbert completion of $\mathcal{F}_{t,\textbf{1}_{ \partial X}}/N_{t}$ and define the Hilbert space
\begin{equation}
\overline{\mathcal{F}_{t,\textbf{1}_{ \partial X}}/N_{t}}^{\|\cdot\|_{\mathcal{K}_{t}}}:=\mathcal{K}_{t},
\end{equation}
\begin{lemma}
The Hilbertian representation $(\pi_{t},\mathcal{K}_{t})$ is well defined.
\end{lemma}
\begin{proof}
First of all, note that $(\pi_{t},\mathcal{F}_{t,\textbf{1}_{ \partial X}} /N_{t})$ is still a linear representation of $\Gamma$ because $N_{t}$ is a linear subrepresentation of $\mathcal{F}_{t,\textbf{1}_{ \partial X}}$. Moreover we have:
$$
\|\pi_{t}(\gamma)v\|_{\mathcal{K}_{t}}=\|\mathcal{I}_{t}(\pi_{t}(\gamma)v)\|_{2}
=\|\pi_{-t}(\gamma)\mathcal{I}_{t}(v)\|_{2}
\leq Ce^{tQ|\gamma|}\|\mathcal{I}_{t}(v)\|_{2}
=Ce^{tQ|\gamma|}\|v\|_{\mathcal{K}_{t}},$$
where the second equality follows from Proposition \ref{intertwin}. 

\end{proof}

Endow now $\mathcal{F}_{-t,\sigma_{t}} $ with the sesquilinear form $\langle\cdot,\cdot\rangle_{\mathcal{K}^{'}_{t}}:= \langle \mathcal{J}_{t} (\cdot), \mathcal{J}_{t}(\cdot)\rangle$ where  $\mathcal{J}_{t}$ is the left inverse of $\mathcal{I}_{t}$. Note that   $(\mathcal{F}_{-t,\sigma_{t}},\langle \cdot , \cdot \rangle_{\mathcal{K}^{'}_{t}})$ is a pre-Hilbert space.\\
Then turn $\mathcal{F}_{-t,\sigma_{t}}$ into a Hilbert space by considering

$$\mathcal{K}'_{t}:=\overline{\mathcal{F}_{-t,\sigma_{t}}}^{ \| \cdot \|_{\mathcal{K}^{'}_{t} } }.$$
We go on in this subsection by the following observation.
\begin{lemma}\label{intertwin2}
The left inverse $\mathcal{J}_{t}$ of $\mathcal{I}_{t}$ intertwines $\pi_{-t}$ and $\pi_{t}$ i.e. $$ \pi_{t} \mathcal{J}_{t}=\mathcal{J}_{t} \pi_{-t}$$ on $\mathcal{F}_{-t,\sigma_{t}}.$
\end{lemma}

\begin{proof}
Let $v=\mathcal{I}_{t}(\pi_{t}(h)\textbf{1}_{\partial X})\in \mathcal{F}_{-t,\sigma_{t}}=\mathcal{I}_{t}(\mathcal{F}_{t,\textbf{1}_{ \partial X}})$ with $h\in \mathbb{C}[\Gamma]$. If $\pi_{t}(h)\textbf{1}_{\partial X}\in N_{t}$ there is nothing to do, thus assume that $h$ is such that $ p( \pi_{t}(h)\textbf{1}_{\partial X})\neq 0$ in $\mathcal{F}_{t,\textbf{1}_{\partial X}}/N_{t}$.  Thus: 
$$
\mathcal{J}_{t}\pi_{-t}(\gamma)v=\mathcal{J}_{t}\mathcal{I}_{t}(\pi_{t}(\gamma)\pi_{t}(h)\textbf{1}_{\partial X})
=(\pi_{t}(\gamma)\pi_{t}(h)\textbf{1}_{\partial X})
=(\pi_{t}(\gamma)\mathcal{J}_{t}\mathcal{I}_{t}\pi_{t}(h)\textbf{1}_{\partial X})
=\pi_{t}(\gamma)\mathcal{J}_{t}v.
$$

\end{proof}

\begin{lemma}
The representation $(\pi_{-t},\mathcal{K}'_{t})$ is well defined.
\end{lemma}

\begin{proof}
We have for all $\gamma \in \Gamma$, for all $v\in \mathcal{F}_{-t,\sigma_{t}}$ :
$$
\|\pi_{-t}(\gamma)v\|_{\mathcal{K}^{'}_{t}}=\|\mathcal{J}_{t}(\pi_{-t}(\gamma)v)\|_{2}
=\|\pi_{t}(\gamma)\mathcal{J}_{t}(v)\|_{2}
\leq Ce^{tQ|\gamma|}\|\mathcal{J}_{t}(v)\|_{2}
=Ce^{tQ|\gamma|}\|v\|_{\mathcal{K}^{'}_{t}},
$$
where the second equality follows from Lemma \ref{intertwin2}. 

\end{proof}

Using the general theory of dual pairings, several properties are obvious and well known. We sum up the properties of the dual pairs we shall consider, and we give a proof of these facts. 

\begin{prop}\label{duality}
Let $t>0$. Assume that $\mathcal{F}_{t,\textbf{1}_{\partial X}}$ is dense in $L^{2}(\partial X,\nu_o)$.\\

\begin{enumerate}
\item \label{1''}The pair $(\mathcal{K}_{t},\overline{\mathcal{K}'_{t}},\langle \cdot,\cdot\rangle)$ is a $\mathbb{C}$-pairing.
\item \label{2''}The topological dual of $\mathcal{K}_{t}$ is isomorphic and isometric to $\overline{\mathcal{K}'_{t}}$. 

\item \label{3''}We have the following embeddings  for all $t>0$:
$$ \mathcal{K}'_{t}\subset L^{2}(\partial X,\nu_{o})\subset  \mathcal{K}_{t}.$$


\end{enumerate}

\end{prop}

\begin{proof}
We give a proof of Item (\ref{1''}). Observe first that $\langle \cdot ,\cdot \rangle$ is a well defined bilinear form on $\mathcal{F}_{t,\textbf{1}_{ \partial X}}/ N_{t}\times \overline{ \mathcal{F}_{-t,\sigma_{t}}} $. Then notice that the standard inner product is non degenerate on $\mathcal{F}_{t,\textbf{1}_{ \partial X}}/ N_{t}\times  \mathcal{F}_{-t,\sigma_{t}} $:  on one hand, if $v\in \mathcal{F}_{t,\textbf{1}_{\partial X}}$ satisfies $\langle v,w\rangle=0$ for all $w\in  \mathcal{F}_{-t,\sigma_{t}}$ it follows, since $\mathcal{I}_{t}$ is continuous and since $\mathcal{F}_{t,\textbf{1}_{ \partial X}}$ is $\|\cdot\|_{2}$-dense  in $L^{2}(\partial X,\nu_{o})$, that $v\perp \overline{Im(\mathcal{I}_{t})}^{\|\cdot\|_{2}}$. Hence, $v\in \mathcal{V}_{t}\cap \mathcal{F}_{t,\textbf{1}_{\partial X}}=N_{t}$ and  it follows that $v=0$ in $\mathcal{F}_{t,\textbf{1}_{ \partial X}}/ N_{t}$. \\
 On the other hand, using again the fact that $\mathcal{F}_{t,\textbf{1}_{ \partial X}}$ is $\|\cdot\|_{2}$-dense in $L^{2}(\partial X,\nu_{o})$, we obtain that if $w$ satisfies $\langle v,w\rangle =0$ for all $v\in \mathcal{F}_{t,\textbf{1}_{ \partial X}}/ N_{t}$, then $w=0$.\\ 
 Observe now the continuity property given for $(v,w)\in \mathcal{F}_{t,\textbf{1}_{ \partial X}}/ N_{t}\times  \mathcal{F}_{-t,\sigma_{t}}$ by : 
$$
|\langle v,w\rangle|=|\langle \mathcal{J}_{t} \mathcal{I}_{t} v,w\rangle|
=|\langle  \mathcal{I}_{t} v,\mathcal{J}_{t}w\rangle|
\leq \|v\|_{\mathcal{
K}_{t}}\|w\|_{\mathcal{K}'_{t}}.
$$
Hence, $\langle \cdot ,\cdot \rangle$ is a well defined non-degenerate bilinear form on $\mathcal{
K}_{t} \times \overline{\mathcal{
K}'_{t}}$.\\

We prove now Item (\ref{2''}).
 For any $w\in \overline{\mathcal{K}'_{t}}$, the map $\lambda_w:v\mapsto\langle v,w\rangle$ defines an element of $(\mathcal{K}_{t})^{*}$. We shall show that given $\ell \in (\mathcal{K}_{t})^{*}$, there exists $w\in \mathcal{K}'_{t}$ such that $\ell =\lambda_w$.
 Using the Riesz representation theorem there exists $a\in \overline{\mathcal{K}_{t}}$ such that for all $v\in \mathcal{K}_{t}$, $\ell(v)=\langle v,a\rangle_{\mathcal{K}_{t}}=\langle v,\mathcal{I}^{2}_{t}{a}\rangle.$  It turns out that 
 $\mathcal{I}^{2}_{t}{a}$  is an element of $ \overline{\mathcal{K}'_{t}}$ :
 $$\|\mathcal{I}^{2}_{t}{a}\|^{2}_{\mathcal{K}'_{t}}=\langle \mathcal{J}_{t} \mathcal{I}^{2}_{t}{a}, \mathcal{J}_{t}\mathcal{I}^{2}_{t}{a}\rangle=\langle \mathcal{I}_{t}{a}, \mathcal{I}_{t}{a}\rangle=\|a\|^{2}_{\mathcal{K}_{t}}.$$
 
 Hence any element of $(\mathcal{K}_{t})^{*}$ is represented by an element of $ \overline{\mathcal{K}'_{t}}$ and since the pairing is non degenerate, $(\mathcal{K}_{t})^{*}$ is isomorphic and isometric to $ \overline{\mathcal{K}'_{t}}$.\\
 We prove Item  (\ref{3''}).  Proposition \ref{bounded} implies for all $v\in L^{2} (\partial X,\nu_{o})$ $$ \|v\|_{\mathcal{K}_{t}}\leq \|\sigma_{t}\|_{\infty} \|v\|_{2}.$$ The other inclusion  follows from Lemma \ref{inverse}: Pick $v\in Im (\mathcal{I}_{t})$ represented uniquely by $v=\mathcal{I}_{t}(x)$ with $x$ in the closure of  $Im (\mathcal{I}_{t})$. We have  $$\|v\|^{2}_{\mathcal{K}'_{t}}=\langle  \mathcal{J}_{t}(v),\mathcal{J}_{t}(v)\rangle
=\langle  \mathcal{J}_{t}\mathcal{I}_{t}(x),\mathcal{J}_{t}\mathcal{I}_{t}(x)\rangle 
=\langle x,x \rangle 
\geq \frac{1}{\|\sigma_{t}\|^{2}_{\infty}}   \| \mathcal{I}_{t}(x)\|_{2}^{2}=\frac{1}{\|\sigma_{t}\|^{2}_{\infty}}  \|v\|^{2}_{2}.$$
and the proof is done.\\

 
 \end{proof}

\subsection{   If $\mathcal{I}_{t}$ is a positive operator: a unitary complementary series}
 Assume that there exists $t$  with $0<t<\frac{1}{2}$ such that $\mathcal{R}_{t}$ is positive. Therefore, Lemma \ref{positiv} implies that $\mathcal{I}_{t}$ is positive.
 Consider the sesquilinear form 
 \begin{equation}
 \langle \cdot,\cdot\rangle_{\mathcal{H}_{t} }:(v,w)\in \mathcal{F}_{t,\textbf{1}_{ \partial X}}\times \mathcal{F}_{t,\textbf{1}_{ \partial X}}\mapsto \langle v,\mathcal{I}_{t} (w)\rangle  \in \mathbb{C}.
 \end{equation}

 Let $L_{t}:=\{ v\in \mathcal{F}_{t,\textbf{1}_{ \partial X}}| \langle v,v\rangle_{\mathcal{I}_{t} }=0\}.$
 Consider the pre-Hilbert space $\mathcal{F}_{t,\textbf{1}_{ \partial X}}/L_{t}$ endowed with $\langle \cdot,\cdot \rangle_{\mathcal{I}_t}$ and turn it into a Hilbert space defined as:
 \begin{equation}
\overline{\mathcal{F}_{t,\textbf{1}_{ \partial X}}/L_{t}}^{\|\cdot\|_{\mathcal{I}_{t}}}:=\mathcal{H}_{t}.
\end{equation}
Now, we are ready to prove Proposition \ref{posintro}.

\begin{proof}
We only have to prove that $(\pi_{t},\mathcal{H}_{t})$ is a unitary representation.
Let $\gamma \in \Gamma$ and $v,w\in \mathcal{H}_{t}$.
We have
$$\begin{array}{lcl}
\langle \pi_{t}(\gamma)v,w\rangle_{\mathcal{H}_{t}}&=&\langle \pi_{t}(\gamma)v,w\rangle_{\mathcal{H}_{t}}
=\langle \pi_{t}(\gamma)v,\mathcal{I}_{t}(w)\rangle
=\langle v,\pi_{-t}(\gamma^{-1})\mathcal{I}_{t}(w)\rangle\\
&=&\langle v,\mathcal{I}_{t}(\pi_{t}(\gamma^{-1})w)\rangle
=\langle v,\pi_{t}(\gamma^{-1})w\rangle_{\mathcal{H }_{t}},\end{array}
$$
where the fourth equality follows from Proposition \ref{intertwin}. 
\end{proof}

Consider  the sesquilinear form 
 \begin{equation}
 \langle \cdot,\cdot\rangle_{\mathcal{J}_{t} }:(v,w)\in \mathcal{F}_{-t,\sigma_{t}}\times \mathcal{F}_{-t,\sigma_{t}}\mapsto \langle v,\mathcal{J}_{t} (w)\rangle  \in \mathbb{C},
 \end{equation}
that is an inner product on $\mathcal{F}_{-t,\sigma_{t}}$. Thus, consider the Hilbert space
\begin{equation}
\overline{\mathcal{F}_{-t,\sigma_{t}}}^{\|\cdot\|_{\mathcal{J}_{t}}}:=\mathcal{H}'_{t}.
\end{equation}

\begin{lemma}
The representation $(\pi_{-t},\mathcal{H}'_{t})$ is a well defined unitary representation.
\end{lemma}
The proof of the above lemma follows from Lemma \ref{intertwin2}.
The following is similar to Proposition \ref{duality} but deals with the space $\mathcal{H}_{t}$ and $\mathcal{H}'_{t}$.
\begin{prop}\label{duality'}
 Assume there exists $t\in ]0,\frac{1}{2}]$ such that $\mathcal{I}_{t}$ is positive operator.\\

\begin{enumerate}
\item  \label{1'}  The intertwiner  $\mathcal{I}_{t}$ defines an isometry from $\mathcal{H}_{t}\rightarrow \mathcal{H}'_{t} $ and $\mathcal{J}_{t}$ defines an isometry from $\mathcal{H}'_{t}\rightarrow \mathcal{H}_{t} $ such that $ \mathcal {I}_{t}\circ \mathcal{J}_{t}=I_{\mathcal{H}'_{t}}$ and $ \mathcal {J}_{t}\circ \mathcal{I}_{t}=I_{\mathcal{H}_{t}}$.

\item \label{2'} The pair $(\mathcal{H}_{t},\overline{\mathcal{H}'_{t}},\langle \cdot,\cdot\rangle)$ is a $\mathbb{C}$-pairing.

\item \label{3'}  The topological dual of $\mathcal{H}_{t}$ is isomorphic and isometric to $\overline{\mathcal{H}'_{t}}$.
\item  \label{4'} We have the following embeddings 
$$ \mathcal{H}'_{t}\subset L^{2}(\partial X,\nu_{o})\subset  \mathcal{H}_{t}.$$

\end{enumerate}

\end{prop}

\begin{proof}
For Item (\ref{1'}) notice  that for $v\in \mathcal{F}_{t,\textbf{1}_{\partial X}}/N_{t}$ we have $ \|\mathcal{I}_{t}(v)\|^{2}_{\mathcal{H}'_{t}}=\langle \mathcal{I}_{t}(v),\mathcal{J}_{t}\mathcal{I}_{t}(v) \rangle=\langle \mathcal{I}_{t}(v),v \rangle =\|v\|^{2}_{\mathcal{H}_{t}}$ and for $v\in \mathcal{F}_{-t,\sigma_{t}}$ we have $ \|\mathcal{J}_{t}(v)\|^{2}_{\mathcal{H}_{t}}=\langle \mathcal{J}_{t}(v),\mathcal{I}_{t}\mathcal{J}_{t}(v) \rangle=\langle \mathcal{J}_{t}(v),v \rangle =\|v\|^{2}_{\mathcal{H}'_{t}}.$ Extend the above equalities on the Hilbert spaces to obtain the result and use Lemma \ref{inverse}. \\
For Item (\ref{2'}), observe that for $v$ in $\mathcal{F}_{t,\textbf{1}_{\partial X}}$ and $w\in \mathcal{F}_{-t,\sigma_{t}}$ we have $$| \langle v,w\rangle | = | \langle v,\mathcal{I}_{t}\mathcal{J}_{t}(w)\rangle | \leq \|v\|_{\mathcal{H}_{t}}\|\mathcal{J}_{t}(w)\|_{\mathcal{H}_{t}}=\|v\|_{\mathcal{H}_{t}}\|w\|_{\mathcal{H}'_{t}}.$$
Then the proof is the same as the proof of Item (\ref{1''}) of Proposition \ref{duality}.\\
For Item (\ref{3'}), since the bilinear form is non degenerate we have an isometric inclusion of the map $w\in \mathcal{H}'_{t}\longmapsto \lambda_{w} \in (\mathcal{H}_{t})^{*} $, where $\lambda_{w}:v\in \mathcal{H}_{t}\mapsto \langle v,w\rangle \in \mathbb{C}$.
Pick now $\ell \in \mathcal{H}_{t}^{*}$. There exists $a\in \mathcal{H}_{t}$ such that $\ell(v)=\langle v,a\rangle_{\mathcal{H}_{t}}=\langle v,\mathcal{I}_{t}(a)\rangle $. Set $w=\mathcal{I}_{t}(a)\in \mathcal{H}'_{t}$ to conclude the proof.\\
Eventually for Item (\ref{4'}) we have for the right inclusion : $$\|v\|^{2}_{\mathcal{H}_{t}}=\langle v,\mathcal{I}_{t}(v) \rangle \leq  \|v\|_{2}\|\mathcal{I}_{t}(v)\|_{2}\leq \|\sigma_{t}\|_{\infty} \|v\|_{2}^{2}.$$ And for the left inclusion use Lemma \ref{inverse} to write $\|v\|_{\mathcal{H}'_{t}}=c^{2}\|v\|^{2}_{2}+q(v)$ where $q(v)$ is positive quadratic form and $c$ a non-zero positive constant. The left inclusion follows.
 \end{proof}

 \subsection{Basic properties of those Hilbertian representations}

We  sum up in what follows useful properties of the Hilbertian representations we are dealing with.
\begin{prop}  Let $t>0$.

\begin{enumerate}
\item \label{1} The intertwiner $\mathcal{I}_{t}$ extends  to an isometry from $\mathcal{K}_{t}$ to $\mathcal{W}_{-t}$. 
\item \label{2}  If $0<t<\frac{1}{2}$ the representations $(\pi_{t},\mathcal{K}_{t})$ and $(\pi_{-t},\mathcal{W}_{t})$ are infinite dimensional.
\item \label{3} The representations $(\pi_{\frac{1}{2}},\mathcal{K}_{\frac{1}{2}})$ and $(\pi_{-\frac{1}{2}},\mathcal{W}_{-\frac{1}{2}})$ are nothing but the trivial representation.
\item  \label{4} The representations $(\pi_{t},\mathcal{K}_{t})$ and $(\overline{\pi_{-t}},\overline{\mathcal{K}'_{t}})$ are a dual system representation with respect to the $L^{2}$-inner product.\\

\noindent And \textbf{if $\mathcal{I}_{t}$ is positive} then :\\

\item \label{5}  It implies $t\leq {1\over 2}$.
\item \label{6} If $t\neq \frac{1}{2}$, the representations $(\pi_{t},\mathcal{H}_{t})$ and $(\pi_{-t},\mathcal{H}_{t}')$ are infinite dimensional unitary representations.
\item \label{7} The representation $(\pi_{\frac{1}{2}},\mathcal{H}_{\frac{1}{2}})$ is nothing but the trivial representation.
\item \label{8} The representations of $(\pi_{t},\mathcal{H}_{t})$ and $(\overline{\pi_{-t}},{\mathcal{H}'_{t}})$ are a dual system representation of $\Gamma$ with respect to the $L^{2}$-inner product.

\end{enumerate}

\end{prop}
\begin{proof}
We first prove Item (\ref{1}). The operator $\mathcal{I}_{t}:\mathcal{F}_{t,\textbf{1}_{ \partial X}}\rightarrow \mathcal{F}_{-t,\sigma_{t}}$ extends as an isometry from $\mathcal{K}_{t}$ to $\mathcal{W}_{-t}$. Indeed  for all $v\in \mathcal{F}_{t,\textbf{1}_{ \partial X}}$ we have
$$ \|\mathcal{I}_{t}(v)\|_{2}=\|v\|_{\mathcal{K}_{t}}.$$ Hence, $\mathcal{I}_{t}$ extends to an isometry from $\overline{\mathcal{F}_{t,\textbf{1}_{ \partial X}}}^{\|\cdot\|_{\mathcal{K}_{t}}}$ to $\overline{\mathcal{F}_{-t},\sigma_{t}}^{\|\cdot\|_{2}}$.\\
We prove Item (\ref{2}): the vector space $\mathcal{F}_{t,\textbf{1}_{ \partial X}}/N_{t}$ is infinite dimensional as a $\pi_{-t}$ representation of $\Gamma$ by Corollary \ref{infinitedim}. Hence $\mathcal{K}_{t}$ and $\mathcal{W}_{t}$ are infinite dimensional Hilbert spaces.\\
For Item (\ref{3}), observe that $\mathcal{I}_{\frac{1}{2}}v=\big(\int_{\partial X}vd\nu_{o}\big)\textbf{1}_{\partial X},$ that is the orthogonal projection onto the one dimensional space of constant functions. Hence, $\mathcal{W}_{-\frac{1}{2}}$ is a one dimensional vector space on which $\pi_{-\frac{1}{2}}$ acts trivially. Observe also that $\ker \mathcal{I}_{\frac{1}{2}}=N_{\frac{1}{2}}=L^{2}_{0}(\partial X)$, where $L^{2}_{0}(\partial X)=L^{2}(\partial X,\nu_{o})\ominus \mathbb{C}\textbf{1}_{\partial X}$. Therefore, $(\pi_{\frac{1}{2}},\mathcal{K}_{\frac{1}{2}})$, $(\pi_{\frac{1}{2}},L^{2}(\partial X,\nu_{o}){/} \mathcal{V}_{\frac{1}{2}})$ are the trivial representations.\\
For Item (\ref{4}), notice that $(\mathcal{K}_{t},\overline{\mathcal{K}'_{t}},\langle\cdot,\cdot\rangle)$ is a $\mathbb{C}$-pairing with respect to $\langle\cdot,\cdot\rangle$. Moreover we have for all $\gamma \in \Gamma$ and for any $(v,w)\in \mathcal{K}_{t}\times \mathcal{K}'_{t}$ that $\langle\pi_{t}(\gamma)v,\overline{\pi_{-t}}(\gamma)w\rangle=\langle v,w\rangle$.\\
For Item (\ref{5}), see Proposition \ref{nonpos}. Proofs of (\ref{6}), (\ref{7}) and (\ref{8}) are analogous to that of (\ref{2}), (\ref{3}) and (\ref{4}) respectively.

\end{proof}

    \section{On positivity of $\mathcal{I}_{t}$}\label{section5}
We start proving that $\mathcal{I}_{t}$ can't be positive  for any  $t>\frac{1}{2}$.

 \subsection{Spherical functions on hyperbolic groups }\hfill
 
 \vspace*{0,2cm}
 
 Asymptotic behavior of matrix coefficients' functions : 
 $$\phi_t\,:\, \gamma\in \Gamma \mapsto \langle\pi_t(\gamma)\textbf{1}_{\partial X},\textbf{1}_{\partial X}\rangle$$
 is a central question in representation theory and harmonic analysis on Lie groups (see for example \cite{GV}). The above functions are also called spherical, as they can be estimated by functions of length of $\gamma$ (see below). For instance, in the Lie group framework,  $\phi_0$ is  Harish-Chandra's  function (denoted by $\Xi$ in that context) and has been estimated first by the eponymous author  and  then by Anker  who gave an optimal lower bound. In the context of hyperbolic groups one can prove the following estimates (\cite{FiPi}, \cite{BM}, \cite{Ga}, \cite{Boy2}), called \emph{ Harish-Chandra Anker estimates} referring to  \cite{Ank} : there exists $C>0$ such that for all $\gamma \in \Gamma:$
\begin{align}\label{HCHestimates}
C^{-1}(1+|\gamma|)e^{-\frac{1 }{2}Q |\gamma|} \leq \phi_{0}(\gamma)\leq C (1+|\gamma|)e^{-\frac{1 }{2}Q |\gamma|}.
\end{align}
\\
 The family of functions $(\omega_t)_{t\in \mathbb{R}^{*} }$ defined by
\begin{equation}\label{functiono}
\omega_{t}(x) =\frac{2 \sinh\big( tQ x\big) }{e^{2tQ}-1},
\end{equation}
are positive, converges uniformly on compact sets to $x\mapsto x$ (when $t\to 0$) and satisfy $\omega_{-t}(x)=e^{2tQ}\omega_{t}(x)$. It allows to extend Harish-Chandra Anker estimates for all representations $\pi_t$ (see \cite{Boy}). 
 \begin{prop}\label{estimatesspherical}
   There exists $C>0$, such that for any $t\in \mathbb{R}$, we have for all $\gamma \in \Gamma:$
     $$C^{-1}e^{-\frac{1}{2}Q|\gamma| }\big(1+\omega_{|t|}(|\gamma|)\big)
 \leq   \phi_{t}(\gamma)\leq 
 Ce^{-\frac{1}{2}Q|\gamma|}\big(1+\omega_{|t|}(|\gamma|)\big).$$
    \end{prop}

 For our purposes, it will be suitable to normalize spherical functions, so that we adopt a slightly different definition of such functions (which coincide to the mentioned ones in the context of semisimple Lie group). 
 Pick $t\in ]0,\frac{1}{2}]$ and set  \begin{equation}\label{sphericalfunction}
 \varphi_{t}:\Gamma \mapsto \frac{1}{\|\sigma_{t}\|_{1}}\langle  \pi_{t}(\gamma) \textbf{1}_{\partial X},\mathcal{I}_{t}(\textbf{1}_{\partial X} )\rangle\in \mathbb{R}.
 \end{equation}
 which is well defined since  for $t>0$ the operator $\mathcal{I}_{t}$ is well defined. 
 We obtain the following characterization :
 \begin{prop}\label{nonpos}
 If $t$ is not in  $]0,\frac{1}{2}]$, the intertwiner $\mathcal{I}_{t}$ cannot be positive.\end{prop}
 
 \begin{proof}
 Let $t>0$.
 Let $f=\sum_{g}c_{g}D_{g}$ be a finitely supported function in $\mathbb{C}[\Gamma]$ and let $v=\pi_{t}(f)=\sum_{g}c_g\pi_{t}(g) \textbf{1}_{\partial X}\in \mathcal{F}_{t,\textbf{1}_{\partial X}}$.
 Consider $\gamma \in \Gamma \mapsto \varphi_{t}(\gamma)$ and observe that : 
 \begin{align*}
 \sum_{g,h }\overline{c_{h}}c_{g} \varphi_{t}(h^{-1}g)&= \frac{1}{\|\sigma_{t}\|_{1}} \sum_{g,h }\overline{c_{h}}c_{g}  \langle  \pi_{t}(h^{-1}g) \textbf{1}_{\partial X},\mathcal{I}_{t}(\textbf{1}_{\partial X} )\rangle\\
 &= \frac{1}{\|\sigma_{t}\|_{1}} \sum_{g,h }\overline{c_{h}}c_{g}  \langle  \pi_{t}(h^{-1}) \pi_{t}(g) \textbf{1}_{\partial X},\mathcal{I}_{t}(\textbf{1}_{\partial X} )\rangle\\
 &=  \frac{1}{\|\sigma_{t}\|_{1}}\sum_{g,h }\overline{c_{h}}c_{g}  \langle  \pi_{t}(g) \textbf{1}_{\partial X}, \pi_{-t}(h)\mathcal{I}_{t}(\textbf{1}_{\partial X} )\rangle\\
 &=  \frac{1}{\|\sigma_{t}\|_{1}}\sum_{g,h }\overline{c_{h}}c_{g}  \langle  \pi_{t}(g) \textbf{1}_{\partial X}, \mathcal{I}_{t}(\pi_{t}(h)\textbf{1}_{\partial X} )\rangle\\
 &=  \frac{1}{\|\sigma_{t}\|_{1}} \langle  v, \mathcal{I}_{t}(v)\rangle,\\
 \end{align*}
 Hence, $\varphi_{t}$ is of  positive type if and only if $\mathcal{I}_{t}$ is positive on $L^{2}(\partial X,\nu_{o})$, since $\mathcal{F}_{t,\textbf{1}_{ \partial X}}$ is dense in $L^{2}(\partial X,\nu_{o})$ by Proposition \ref{density}. Write $$\varphi_{t}(\cdot)=\frac{1}{\|\sigma_{t}\|_{1}}\langle  \pi_{t}(\cdot) \textbf{1}_{\partial X},\sigma_{t}\rangle,$$ and use Proposition \ref{cont2} to obtain the existence of $C>0$ such that for any $t>0$ and for all $\gamma \in \Gamma:$ $$C^{-1}\phi_{t}(\gamma)\leq \varphi_{t}(\gamma)\leq C \phi_{t}(\gamma).$$ Hence $\varphi_{t}$ is unbounded if $t>\frac{1}{2}$ since $\phi_{t}$ is unbounded, and thus $\varphi_{t}$ cannot be of positive type \cite[(ii) Proposition C.4.2]{BDV}. Thus, the intertwiner  $\mathcal{I}_{t}$ cannot be positive.  \end{proof}
 
 \begin{remark} A function $\phi$ of positive type on $\Gamma$ gives rise to a unitary cyclic representation, that is a triple $(\pi_{\phi},H_{\phi},v_{\phi})$ with $(\pi_{\phi},H_{\phi})$  unitary, with  $v_{\phi}\in H_{\phi}$ satisfying  $H_{\phi}=\overline{\pi_{\phi}(\Gamma)v_{\phi}}$ and $\phi(\cdot)=\langle \pi_{\phi}(\cdot)v_{\phi},v_{\phi}\rangle.$\\
  This is the so called Gelfand-Naimark-Segal construction (GNS for short). 
 We refer to \cite[Theorem C.4.10]{BDV} for more details.\\
 If we perform the GNS construction with $\pi_{\varphi_{t}}$ for fixed $t\in [0,\frac{1}{2}]$,
 the representations  $(\pi_{\varphi_{t}},H_{\varphi_{t}})$ and $(\pi_{t},\mathcal{H}_{t})$ are unitarily equivalent. 
   \end{remark}

Let $\Gamma$ be a discrete group acting geometrically on a proper roughly geodesic $\delta$-hyperbolic $(X,d)$. Assume that $\mathcal{I}_{t}$ is positive for all  $0<t\leq\frac{1}{2}$. Consider the unitary representations $(\pi_{t},\mathcal{H}_{t})$.
We have

\begin{prop}\label{Haagerup}
The group $\Gamma$ satisfies Haagerup property.
\end{prop}
\begin{proof}
Consider the sequence of spherical functions $$\varphi_{t_{n}}:\gamma \in \Gamma \mapsto \frac{1}{\|\sigma_{t_{n}}\|_{1}}\langle  \pi_{t_{n}}(\gamma) \textbf{1}_{\partial X},\mathcal{I}_{t_{n}}(\textbf{1}_{\partial X} )\rangle\in \mathbb{R},$$
with $t_{n}=\frac{1}{2}-\frac{1}{n+2}$ for all $n\in \mathbb{N}^*$. Observe that 

\begin{enumerate}
\item $(\varphi_{t_{n}})_{n\in \mathbb{N}^{*}}$ is in $C_{0}(\Gamma)$ (the space of functions on $\Gamma$ vanishing at infinity) by Proposition \ref{estimatesspherical}.
\item $\varphi_{t_{n}}(e)=1$ for all $n\in \mathbb{N}^{*}.$
\item $\varphi_{t_{n}} \to 1$ as $n\to \infty$ on every finite subsets of $\Gamma$.
\item $(\varphi_{t_{n}})_{n\in \mathbb{N}^{*}}$ are positive definite functions on $\Gamma$ since $\mathcal{I}_{t}$ is assumed to be positive for $0<t\leq \frac{1}{2}$.
\end{enumerate}
Using  \cite[(2) of Theorem 2.1.1]{CJV} we deduce that $\Gamma$ satisfies Haagerup property.
\end{proof}

 \subsection{Equidistribution results}

Let $t\in \mathbb{R}$ and consider the measure:
\begin{equation}\label{themesure}
dm_{o,t}(\xi,\eta):=k_{t}(\xi,\eta)d\nu_{o}(\xi)d\nu_{o}(\eta) \in C^{*}(\overline{X} \times \overline{X}).
\end{equation}

 Let $A>0$ and consider the truncated \emph{continuous} kernel (denoted by $k_{t}^A$ with a slight abuse of notation) approximating $k_{t}$ on $\overline{X}\times \overline{X}$ defined as:
 
 $$
k_{t}^A:(\xi,\eta)\in \partial X \times \partial X\mapsto  \left\{
    \begin{array}{ll}
        e^{(1-2t)Q (\xi,\eta)_{o}}& \mbox{if $(\xi,\eta)_{o}\in [0,A]$ }  \\
        e^{(1-2t)Q A}(A+1-(\xi,\eta)_{o}) & \mbox{if $A<(\xi,\eta)_{o} <A+1.$}\\
       0 & \mbox{if $(\xi,\eta)_{o}\in ]A+1,+\infty[.$}
    \end{array}
\right.
$$
Thus define 
\begin{equation}
dm^{A}_{o,t}(\xi,\eta):=k_{t}^A(\xi,\eta)d\nu_{o}(\xi)d\nu_{o}(\eta) \in C^{*}(\overline{X} \times \overline{X}).
\end{equation}

\begin{lemma}\label{pos2}
For all  $t>0$, we have the following weak$^*$ convergence:
$m^{A}_{o,t}  \to m_{o,t}$ as $A\to \infty.$ uniformly in $t$.
\end{lemma}
\begin{proof}
For all $F\in C^{*}(\overline{X} \times \overline{X})$ we have:
\begin{align*}
|m^{t}_{o}(F)-m^{t,A}_{o}(F)|&=\bigg|\int_{\partial X \times \partial X}F(\xi,\eta)\, [ k_{t}(\xi,\eta)-k_{t}^A(\xi,\eta)]\, d\nu_{o}(\xi)d\nu_{o}(\eta)\bigg|\\
&\leq C \|F\|_{\infty} \int_{(\xi,\eta)_{o}\geq A+1} \frac{1}{d^{(1-2t)Q}_{o}(\xi,\eta)}d\nu_{o}(\xi)d\nu_{o}(\eta).
\end{align*}
Hence,
$$\|m_{o,t}-m^{A}_{o,t}\|_{C^{*}(\partial X \times \partial X)} \leq C \int_{(\xi,\eta)_{o}\geq A+1} \frac{1}{d^{(1-2t)Q}_{o}(\xi,\eta)}d\nu_{o}(\xi)d\nu_{o}(\eta).$$
Use Lemma \ref{limsup}, let $A\to +\infty$ to conclude the proof.
\end{proof}

Define now for $R>0$ large enough :

\begin{equation}
m^{A}_{o,t,N,R}:=\sum_{g,h\in S^{\Gamma}_{N,R}}\beta_{N,R}(g)\beta_{N,R}(h)k_{t}^A(go,ho)D_{g o}\otimes D_{ho} \in C^{*}(\overline{X} \times \overline{X}).
\end{equation}

Since $R$ is fixed, we write $m^{A}_{o,t,N}$ for $m^{A}_{o,t,N,R}$.
Consider  the two sequences of measures of $C^{*}(\overline{X} \times \overline{X})$  and defined for $t\in\mathbb R$ by :
 
 \begin{equation}
m^{t}_{o,N,R}:=\sum_{g,h\in S^{\Gamma}_{N,R}}\beta_{N,R}(g)\beta_{N,R}(h) e^{(1-2t)Q(go,ho)_{o}}D_{g o}\otimes D_{ho},\;\;\;\;\;\;\; m_{o,N,R}:=m_{o,N,R}^{1/2}
\end{equation}
 
\begin{lemma}\label{equidouble} For all $R>0$ large enough,
the sequence $(m_{o,N,R})_{N \in \mathbb{N}}$ converges to $\nu_{o}\otimes \nu_{o}$ with respect to the weak* topology on $C(\overline{X} \times \overline{X})$.  
\end{lemma} 
\begin{proof}
Since for all non negative integers $k$ and for all $g\in S_{k,R}$, we have $\beta_{k}(g)\leq \frac{C}{|S_{k,R}|}$, it follows that
$$\sup_{N}\|m_{o,N,R}\|_{C(\overline{X}\times\overline{X} )^{
*}}<+\infty.$$
Thus, up to subsequences, it follows that $ (m_{o,N,R})_{N \in \mathbb{N}}$ has a limit with respect to the weak* topology on $C(\overline{X}\times\overline{X} )$. But on the space generated by pure tensors written as $f\otimes g$ with $f,g\in C(\overline{X})$, we have $m_{o,N,R}(f\otimes g) \to (\nu_{o}\otimes \nu_{o})(f \otimes g)$ as $N \to +\infty$ 
by applying Corollary \ref{equi'} for $R>0$ large enough. Hence, we deduce that for all $F\in C(\overline{X}\times \overline{X})$ we have: $m_{o,N,R}(F)\to \nu_{o}\otimes \nu_{o}(F)$ as $N\to +\infty$.
 
\end{proof}

\subsubsection{Convergence of measures}
\begin{lemma}\label{pos1}
Let $R>0$ large enough. For any $t<\frac{1}{2}$, we have the following weak* convergence for all $A>0$
$m^{A}_{o,t,N}  \to m^{A}_{o,t}$ as $N \to \infty$ in $C^{*}(\overline{X} \times \overline{X}).$
\end{lemma}
\begin{proof}
The proof is a direct application of Lemma \ref{equidouble} since the map $(x,y)\in \overline{X}\times \overline{X}\mapsto k_{t,A}(x,y)$ is continuous.
\end{proof}

\begin{lemma}\label{pos3}
Assume that $t<\frac{1}{2}$.
We have the following uniform convergence:
$$\lim_{A \to +\infty}\sup_{N} \| m_{o,t,N} - m^{A}_{o,t,N}\|_{C^{*}(\overline{X} \times\overline{ X})}\to 0.$$
\end{lemma}
\begin{proof}
Let $F \in C(\overline{X} \times \overline{X})$. We have:
\begin{align*}
|m^{t}_{o,N}(F) - m^{t,A}_{o,N}(F)|&\leq 2 \|F\|_{\infty}\sum_{\{g,h\in S_{N}|(go,ho)>A\}}\beta_{N,R}(g) \beta_{N,R}(h)e^{(1-2t)Q (go,ho)}.
\end{align*}
For all $h\in S_{N,R}$, Lemma \ref{ineqsigma2} provides a constant $C>0$ such that: 

$$\sum_{\{g\in S^{\Gamma}_{N,R}|(go,ho)>A+1\}}\beta_{N,R}(g)e^{(1-2t)Q (go,ho)} \leq C \frac{e^{-2tQA}}{1-e^{-2tQR}}.$$
Therefore:
$$\|m_{o,t,N} - m^{A}_{o,t,N}\|_{C^{*}(\overline{X} \times\overline{ X})} \leq C\frac{e^{-2tQA}}{1-e^{-2tQR}}.$$
Let $A\to +\infty$ to conclude the proof.
\end{proof}

 We generalize the nice equidistribution result in the context of hyperbolic groups obtained in  \cite[Section 4]{Bou} for CAT(-1) spaces:
 \begin{prop}\label{equidistriboucher}
 For all $t<\frac{1}{2} $ we have:
 
 $$m_{o,t,N} \to m_{o,t},$$ as $N \to +\infty$ for the weak* convergence of $C^{*}(\overline{X} \times \overline{X})$
 \end{prop}
 
 \begin{proof}
 
For all $F\in C(\overline{X}\times \overline{X})$ such that $\|F\|_{\infty}\leq 1$ and for all non negative integers $N$:
 \begin{align*}
  |m_{o,t,N}(F) - m_{o,t}(F)|&\leq |m_{o,t,N}(F) - m^{A}_{o,t,N}(F)| +|m^{A}_{o,t,N}(F) - m^{A}_{o,t}(F)|+|m^{A}_{o,t}(F) - m_{o,t}(F)|\\
  &\leq \sup _{N}\|m_{o,t,N} - m^{A}_{o,t,N}\| +|m^{A}_{o,t,N}(F) - m^{A}_{o,t}(F)|+\|m^{A}_{o,t} - m_{o,t}\|.
\end{align*}
 Take the limit over $N$ applying Lemma \ref{pos1} for the second term of the above inequality, and take the limit over $A$ applying Lemma \ref{pos3} to the first term and Lemma \ref{pos2} to the last term to finish the proof.
 \end{proof}
 
 \subsection{Application: positivity of $\mathcal{I}_{t}$ for $0<t\leq\frac{1}{2}$}\hfill
 
 We prove Proposition \ref{positivcondneg}. Recall that the main assumption of this statement is conditional negativity of the distance. We suspect the latter can be relaxed, as it is a local and global geometric assumption and not only a macroscopic one like Gromov hyperbolicity - hence unsatisfactory.

 \begin{proof}
 Let $0<t<\frac{1}{2}$.
 Since the distance defines a conditionnaly negative kernel, the Gromov product $$(g,h)\in \Gamma \times \Gamma \mapsto (go ,ho)_{o}\in \mathbb{R}^{+},$$ is positive definite, see  \cite[Lemma C.3.1]{BDV}. It follows, \cite[Proposition C.1.6]{BDV}, that
 the kernel $k_{t}$ viewed on $\Gamma$  $$(g,h)\in \Gamma \times \Gamma \mapsto e^{(1-2t)Q(go ,ho)_{o}}\in \mathbb{R}^{+},$$ is positive definite for $t< \frac{1}{2}$ on $\Gamma$.
  Let $v\in C(\partial X)$ and extend $v$ to a continuous function  $\hat{v}$ in $\overline{X}$. Pick $R>0$ large enough. We have for all $N\in \mathbb{N}$:
  $$m_{o,t,N}( \hat{v} \otimes \overline{\hat{v}} )=\sum_{g,h\in S^{\Gamma}_{N,R}}\hat{v}(g)\overline{\hat{v}}(h)k_{t}(g,h)\geq 0.$$
Let $N \to +\infty$ ; Proposition \ref{equidistriboucher} implies:
  $$m_{o,t}(\hat{v} \otimes \overline{\hat{v}} )\geq 0.$$
 Finally, observe that for all $t\leq \frac{1}{2} $ and for all continuous functions $v\in C(\partial X):$
  $$m_{o,t}(\hat{v}\otimes \overline{\hat{v}} )=\langle v,\mathcal{I}_{t}(v)\rangle\geq 0.$$  We deduce that $\mathcal{I}_{t}$ is positive on $L^{2}(\partial X ,\nu_{o})$. By Lemma \ref{positiv}, we deduce that $\mathcal{R}_{t}$ is positive for $0<t< \frac{1}{2}$.\\
  The case $t=\frac{1}{2}$ gives $\mathcal{I}_{\frac{1}{2}}(v)=\mathcal{R}_{\frac{1}{2}}(v)=\langle v,\textbf{1}_{\partial X}\rangle \textbf{1}_{\partial X}$, that is the orthogonal projection onto the space of constant functions which is a positive operator. 
 
 \end{proof}

   \section{Spectral gap estimates and consequences}\label{section6}

Using duality arguments, we provide several spectral estimates  for the dual system representations studied in this paper. Such estimates will be useful to establish Schur's asymptotic relations in our context and their proofs  relie  on  \cite[Theorem 1.2]{Boy} (where representations $\pi_t$ are considered only from the operator algebra standpoint).  Notice that they are typical examples of \emph{slow growth representations} (see \cite[Section 8.3]{Ju}). The results can also be related to some of the results  of \cite{JN} in the setting of $Sp(n,1)$. \\
\subsection{Spectral inequalities}
The main result of \cite{Boy} is the following spectral inequality, generalizing the so called ``Haagerup Property" or Property RD.  Let $R>0$ large enough. There exists $C>0$ such that  for any $t \in \mathbb{R}$ and  for all $f_{n}\in \mathbb{C}[\Gamma]$, whose support is in $S^{\Gamma}_{n,R}$, we have:
\begin{equation}\label{RDgeneral}
\|\pi_{t}(f_{n})\|_{L^{2}\to L^{2}}\leq C \omega_{|t|}(nR)\|f_{n}\|_{\ell^{2}},
\end{equation}
(see (\ref{functiono}) for the definition of $\omega_{t}$).

Given $v,w\in L^{2}(\partial X,\nu_{o})$ two unit vectors, apply the above inequality to the sequence of functions $f_{n}=\sum_{\gamma \in S^{\Gamma}_{N,R}}\overline{\langle\pi_{t}(\gamma)v,w\rangle} D_{\gamma}$ to obtain a reformulation in terms of decay of matrix coefficients: there exists $C>0$  such that for all $t\in \mathbb{R}$ and for all unit vectors $v,w\in L^{2}(\partial X,\nu_{o})$
\begin{equation}\label{RDgeneralcoeff}
\big(\sum_{\gamma \in S^{\Gamma}_{N,R}}|\langle \pi_{t}(\gamma)v,w\rangle|^{2}\big)^{1/2}\leq C \omega_{|t|}(nR).
\end{equation}

Along the discussion below, $R>0$ will be large enough so that previous estimates from Section \ref{section2} and \ref{section3} are valid as well as inequality (\ref{RDgeneral}). We will write $f(x)\asymp g(x)$ for positive functions $f,g$ if there exists $C\geq 1$ such that $c^{-1}\leq f/g\leq C$. Last, as value of spherical functions $\phi_t(\gamma)$  depends essentially of length of $\gamma$ only, \textbf{we will write by abuse of notation $\phi_t(nR)$  instead of   $\phi_t(\gamma)$   for (indifferent)  $\gamma\in S_{n,R}^{\Gamma}$.}\\


\begin{prop}\label{radialestimates}
There exists $C>0$ such that for all    $t\in\mathbb R$ and $f_{n}=\sum_{\gamma \in S^{\Gamma}_{n,R}}\beta_{n,R}(\gamma)D_{\gamma}$,
$$\|\pi_{t}(f_{n})\|_{L^{2}\to L^{2}}\leq C \phi_{t}(nR).$$
\end{prop}


\begin{proof}
From $\beta_{n,R}(\gamma)\asymp e^{-nQR}$ when $\gamma\in S_{n,R}^{\Gamma}$ and $|S_{n,R}^{\Gamma}|\asymp e^{nQR}$, it follows $\|f_{n}\|^{2}_{\ell^{2}}=\sum_{\gamma \in S^{\Gamma}_{n,R}}\beta^{2}_{n,R}(\gamma)\asymp e^{-nQR}$. Hence, using (\ref{RDgeneral}), we get for $v,w\in L^2(\partial X,\nu_o)$ unit vectors, 
$$|\langle\pi_{t}(f_{n})v,w\rangle| \leq C \omega_{|t|}(nR)\|f_{n}\|_{\ell^{2}}\leq  C\omega_{|t|}(nR)e^{-nQR/2}\leq C\phi_t(nR)$$
the last inequality being estimate from below of $\phi_t$ given by Proposition \ref{estimatesspherical}. The proof is complete by definition of the operator norm. \end{proof}

\begin{theorem} \label{Hspectral}There exists $C>0$ such that for any $t>0$ and any $f_n\in\mathbb C[\Gamma]$ supported on $S_{n,R}^{\Gamma}$, 
$$ \|\pi_{t}(f_{n})\|_{\mathcal{K}_t\to \mathcal{K}_t}\leq C \|f_{n}\|_{\ell^{2}}\, \omega_{t}(nR)\;\;\; {\rm and}\;\;\;\;  \|\pi_{-t}(f_{n})\|_{\mathcal{K}_t'\to \mathcal{K}_t'}\leq C \|f_{n}\|_{\ell^{2}}\, \omega_{t}(nR)$$
\end{theorem}

\begin{proof}[Proof of Theorem \ref{Hspectral} for $\mathcal{K}_t$ and $\mathcal{K}'_t$]  We first prove norm operator estimates for $\mathcal {K}_t$ and ${\mathcal K}_t'$. Let $v,w\in \mathcal {K}_t\subset L^2(\partial X,\nu_o)$.

$$\begin{array}{lll}|\langle \pi_{t}(f_{n})v,w\rangle_{{\mathcal K}_t} | &=&|\langle {\mathcal I}_t \pi_{t}(f_{n})v,{\mathcal I}_tw\rangle |= |\langle  \pi_{-t}(f_{n}){\mathcal I}_t(v),{\mathcal I}_t(w)\rangle |\\
                                                                                                                  &\leq & \|  \pi_{t}(f_{n})\|_{L^2\to L^2}\| {\mathcal I}_t(v)\|_2\; \| {\mathcal I}_t(w)\|_2= \|  \pi_{t}(f_{n})\|_{L^2\to L^2}\|v \|_{{\mathcal K}_t}\; \| w\|_{{\mathcal K}_t}\\
                                                                                                                  &\leq & C\omega_t(nR) \|f_n\|_{\ell^2}\|v \|_{{\mathcal K}_t}\|w \|_{{\mathcal K}_t}       \end{array}$$
                                                                                                                  
For $v,w\in {\mathcal F}_{-t,\sigma_t}$, we get             
 $$\begin{array}{lll}|\langle \pi_{-t}(f_{n})v,w\rangle_{{\mathcal K}_t'} | &=&|\langle {\mathcal J}_t \pi_{t}(f_{n})v,{\mathcal J}_tw\rangle |= |\langle  \pi_{t}(f_{n}){\mathcal J}_t(v),{\mathcal J}_t(w)\rangle |\\
                                                                                                                  &\leq & \|  \pi_{t}(f_{n})\|_{L^2\to L^2}\| {\mathcal J}_t(v)\|_2\; \| {\mathcal J}_t(w)\|_2= \|  \pi_{t}(f_{n})\|_{L^2\to L^2}\|v \|_{{\mathcal K}_t'}\; \| w\|_{{\mathcal K}_t'}\\
                                                                                                                  &\leq & C\omega_t(nR) \|f_n\|_{\ell^2}\|v \|_{{\mathcal K}_t'}\|w \|_{{\mathcal K}_t'}       \end{array}$$
 the two last inequalities coming from \ref{RDgeneral} applied to $\pi_{-t}$ and $\pi_t$ respectively. The proofs are complete by definition of norm of operators.  \end{proof}
 
We prove now  norm operator estimates assuming there exists $t>0$ such that ${\mathcal I}_t>0$, that is estimates for  ${\mathcal H}_t$ and ${\mathcal H}_t'$. \\

\begin{theorem}\label{Uspectral} There exists $C>0$ such that for any $t>0$ and any $f_n\in\mathbb C[\Gamma]$ supported on $S_{n,R}^{\Gamma}$, 
$$ \|\mathcal{I}_{t}\pi_{t}(f_{n})\|_{\mathcal{H}_t\to \mathcal{H}_t}\leq C \|f_{n}\|_{\ell^{2}}\, \omega_{t}(nR)\;\;\; {\rm  and}\;\;\;\;  \|\pi_{-t}(f_{n})\|_{\mathcal{H}_t'\to \mathcal{H}_t}\leq C \| f_{n} \|_{\ell^{2}} \omega_{t}(nR).$$

\end{theorem}
\begin{proof} Let $v',w'\in {\mathcal H}_t'$. Recall that $\check{f}(\gamma)= \bar{f}(\gamma^{-1})$ and note that $\| \check{f_{n}}\|_{\ell^2}= \| f_n\|_{\ell^2}$. Then, we have
$$| \langle\pi_{t}(f_{n})v',w'\rangle|=|\langle v',\pi_{-t}(\check{f_{n}})w'\rangle|\leq C \|f_{n}\|_{\ell^{2}} \|v'\|_{2}\, \|w'\|_{2}\; \omega_{t}(nR),$$
the inequality being (\ref{RDgeneral}) applied to $\pi_{-t}$. Embeddings ${\mathcal H}_t'\subset L^2(\partial X,\nu_o)\subset {\mathcal H}_t$  (Proposition \ref{duality'}) implies 
$$|\langle v',\pi_{-t}(\check{f_{n}})w'\rangle|\leq C \|f_{n}\|_{\ell^{2}} \| v' \|_{{\mathcal H}_t'}\, \| w'  \|_{{\mathcal H}_t'}\; \omega_{t}(nR)$$
and the map $v'\in {\mathcal H}_t'\mapsto \langle v',\pi_{-t}(\check{f_{n}})w'\rangle \in \mathbb{C}$  defines an element of $ (\mathcal{H}_{t}')^{*}=\overline{{\mathcal H}_t}$. Thus
$$\| \pi_{-t}(\check{f_{n}})w'\|_{{\mathcal H}_t}\leq C  \|f_{n}\|_{\ell^{2}} \, \| w'  \|_{\mathcal{H}'_{t}}\; \omega_{t}(nR).$$ 

We obtain
$$\| \pi_{-t}(f_{n})w'\|_{{\mathcal H}_t}\leq C  \|f_{n}\|_{\ell^{2}} \, \| w'  \|_{\mathcal{H}'_{t}}\; \omega_{t}(nR).$$

Again from Proposition \ref{duality'}, we know that ${\mathcal I}_t\;:\; {\mathcal H}_t\to {\mathcal H}_t'$ is an isometry. If we set $w'={\mathcal I}_t(w)$ with $w\in {\mathcal H}_t$, using the intertwining relation gives for all $w\in {\mathcal H}_t$,
$$\|\mathcal{I}_{t}\pi_t(f_n)w\|_{{\mathcal H}_t}\leq C \|f_{n}\|_{\ell^{2}} \, \| w \|_{{\mathcal H}_t}\; \omega_t(nR).$$

 \end{proof}
   \red{ \begin{remark}
    We do not know if the bound  $\|\pi_{t}(f_{n})\|_{\mathcal{H}_t\to \mathcal{H}_t}\leq C \|f_{n}\|_{\ell^{2}}\, \omega_{t}(nR)$ is true .
    \end{remark}  }

\subsection{Mixing properties of  representations $\pi_{t}$ for $-\frac{1}{2}<t<\frac{1}{2}  $}

The notion of mixing  and weak mixing properties has been intensively used in ergodic theory, unitary representation theory, operator algebras. See for example \cite[p. 280 for definitions and Proposition A.1.12]{BDV}.  The main point here is to deduce from weak mixing property  that the \emph{non unitary} representations we are dealing with do not contain finite dimensional subrepresentations. \\
\begin{defi}
Let $(\langle\cdot,\cdot\rangle,H_{1},H_{2})$ a $\mathbb{C}$-pairing and let  $(\pi_{1},H_{1})$ and $(\pi_{2},H_{2})$ a dual system representation of $\Gamma$ on Hilbert spaces. We say that this dual system representation of $\Gamma$  is weak mixing if for all finite subsets $F_{1}\subset H_{1}$, $F_{2}\subset H_{2}$ and for all $\epsilon >0$ there exists $\gamma \in \Gamma$ such that $|\langle \pi_{1}(\gamma) v,w\rangle \langle  v',\pi_{2}(\gamma)w'\rangle| <\epsilon$ for all $v,v'\in F_{1}$ and $w,w'\in F_{2}.$  
\end{defi}

\begin{prop}\label{weakmixing}
If $t\in]-\frac{1}{2},\frac{1}{2}[ $, then the Hilbertian dual system representation of $\Gamma$ given by $(\pi_{t},L^{2}(\partial X,\nu_{o}) )$ and  $(\pi_{-t},\overline{L^{2}(\partial X,\nu_{o})} )$  is weak mixing. 
\end{prop}

 \begin{proof}

Let  $F\subset \mathcal{H}:=L^{2}(\partial X ,\nu_{o})$ and $\overline{F}\subset \overline{\mathcal{H}}:=\overline{L^{2}(\partial X ,\nu_{o})}$ be  finite subsets of vectors of $\mathcal{H}$ and $\overline{\mathcal{H}}$ and $\epsilon_{0}>0$ such that 
for all $\gamma \in \Gamma$ we have $|\langle\pi_{t}(\gamma)v,w\rangle \langle\pi_{-t}(\gamma)v',w'\rangle| \geq \epsilon_{0},$ for each pair of unit vectors $v,v'\in F$ and $w,w'\in \overline{F}$ .  Inequality (\ref{RDgeneralcoeff}) combined with Cauchy-Schwarz inequality implies for $R>0$ large enough that for all non negative integers $n$ and for all unit vectors $v,v'\in F$ and $w,w'\in \overline{F}$ unit vectors:

$$\sum_{\gamma \in S^{\Gamma}_{n,R}} |\langle \pi_{t}(\gamma)v,w \rangle \langle\pi_{-t}(\gamma)v',w'\rangle | \leq C e^{2tQnR},$$
where $C>0$ does depend on $t$.
It follows, by considering the growth of spheres $S^{\Gamma}_{n,R}$ that for all $n$: $$\epsilon_{0} e^{QnR}\leq Ce^{2|t|QnR}. $$
 and since $|t|<\frac{1}{2}$, the above inequality is absurd for $n$ large enough.  Hence the dual system representation provided by $\pi_{t}$ is weak mixing for $-\frac{1}{2}<t<\frac{1}{2}$.
\end{proof}
We deduce the following property that is  well known in the case of unitary representations.

\begin{prop}\label{infinitedim}
For $t\in ]-\frac{1}{2},\frac{1}{2}[$, the representations $(\pi_{t},L^{2}(\partial X ,\nu_{o}))$ cannot have a finite dimensional subrepresentations.
\end{prop}
 \begin{proof} By contradiction : 
  let $t\in ]-\frac{1}{2},\frac{1}{2}[$ and
  $\mathcal{V}\subset L^{2}(\partial X,\nu_{o})$ be $m$-dimensional, $\pi_{t}$-invariant and let $(v_{1},\dots,v_{m})$ an orthonormal basis of $\mathcal{V}$. Let $v\in \mathcal{V}$ a unit vector. Write for all $\gamma \in \Gamma$:
 \begin{align*}
\pi_{t}(\gamma) v&=\sum^{m}_{i=1}\langle \pi_{t}(\gamma)v,v_{i}\rangle v_{i}.
 \end{align*}
 Observe that $$ \langle\pi_{t}(\gamma) v, \pi_{-t}(\gamma) v\rangle=\|v\|^{2}=1.$$
 Now write for all $\gamma \in \Gamma$ and for all unit vectors $v\in \mathcal{V}$:
 
 \begin{align}\label{equality}
1=\langle\pi_{t}(\gamma) v, \pi_{-t}(\gamma) v\rangle&=\sum^{m}_{i=1}\langle \pi_{t}(\gamma)v,v_{i}\rangle \langle v_{i},\pi_{-t}(\gamma)v\rangle.
 \end{align}

 Now pick $0<\epsilon<\frac{1}{m}$. Therefore, Proposition \ref{weakmixing} provides the existence of $\gamma $ such that 
 $\bigg| \langle \pi_{t}(\gamma)v,v_{i}\rangle \langle v_{i},\pi_{-t}(\gamma)v\rangle \bigg|<\epsilon$. Hence, 
 $$\bigg|\sum^{m}_{i=1}\langle \pi_{t}(\gamma)v,v_{i}\rangle \langle v_{i},\pi_{-t}(\gamma)v\rangle \bigg|<1, $$
  and thus equality (\ref{equality}) is absurd.
  \end{proof}
 As a corollary of this mixing property, we have

      \begin{coro}\label{infinitedim}
 If $0<t<\frac{1}{2}$, then the spaces $\mathcal{F}_{-t,\sigma_{t}}$, $\mathcal{F}_{t,\textbf{1}_{\partial X} }/N_{t}$ and $Im(\mathcal{I}_{t})$ are infinite dimensional vector spaces. In particular the compact operator $\mathcal{I}_{t}$ has not finite rank.
 \end{coro}
\begin{proof}
 Proposition \ref{intertwin} implies that   $(\pi_{-t},\mathcal{F}_{-t,\sigma_{t}})$ is a subrepresentation of $(\pi_{-t},L^{2}(\partial X,\nu_{o}))$. Proposition \ref{infinitedim} implies that $\mathcal{F}_{-t,\sigma_{t}}$  and thus $Im(\mathcal{I}_{t})$ are infinite dimensional representations. Therefore $\mathcal{I}_{t}$ is not a finite rank operator. The isomorphism in  (\ref{isoalg}) ensures that $\mathcal{F}_{t,\textbf{1}_{\partial X} }/N_{t}$ in infinite dimensional. 
 \end{proof}

\subsection{The Poisson transform and the Riesz transform}

\subsubsection{Decay of matrix coefficients}

Let $A>0$. Define a subset of the boundary for $x,y\in X$ as

\begin{equation}\label{defsha}
 \mathcal{O}_{A}(x,y):= \{ \xi\in \partial X | (\xi,y) _{x}\geq A\}.
\end{equation} 
and set \textbf{$\hat{\gamma}$ for $\hat{\gamma.o_o}$}. Here  is the fundamental lemma concerning the matrix coefficients :
 
\begin{lemma}\label{decay2}
Let $t>0$ ; there exists $C>0$ such that for all $w\in Lip(\partial X)$ and for all $A>0$ and for all $\gamma \in \Gamma$:
 $$\frac{\langle \pi_{t}(\gamma)\textbf{1}_{\partial X} , |w-w(\hat{\gamma})\textbf{1}_{\partial X} |\rangle}{\phi_{t}(\gamma)}\leq C\|w\|_{Lip}\big((\varepsilon_{t}(|\gamma|)e^{2tQ A}+e^{-\epsilon A} \big),$$
 where $\varepsilon_{t}(x)=O_{+\infty}(e^{-\alpha x})$ for some $\alpha>0$.
\end{lemma}
\begin{proof}
 Decompose the integral in two terms as follows:
\begin{align*}
\langle \pi_{t}(\gamma)\textbf{1}_{\partial X} , |w-w(\hat{\gamma})\textbf{1}_{\partial X} |\rangle
&=\int _{\mathcal{O}_{A}(o,\gamma o)} e^{(\frac{1}{2}+t)Q\beta_{\xi}(o,\gamma o))}|w(\xi)-w(\hat{\gamma})|d\nu_{o}(\xi)\\&+\int _{\partial X \backslash \mathcal{O}_{A}(o,\gamma o)}e^{(\frac{1}{2}+t)Q\beta_{\xi}(o,\gamma o))}|w(\xi)-w(\hat{\gamma})| d\nu_{o}(\xi).
\end{align*}
First recall by (\ref{buseman}) that $\beta_{\xi}(o,\gamma o)=-|\gamma|+2(\xi,\gamma o)_{o}$.\\ We study the first integral: $$\int _{\mathcal{O}_{A}(o,\gamma o)} e^{(\frac{1}{2}+t)Q\beta_{\xi}(o,\gamma o))}|w(\xi)-w(\hat{\gamma})|d\nu_{o}(\xi)\leq \|w\|_{Lip} \int _{\mathcal{O}_{A}(o,\gamma o)} e^{(\frac{1}{2}+t)Q\beta_{\xi}(o,\gamma o))}d_{o,\epsilon}(\xi,\hat{\gamma})d\nu_{o}(\xi).$$

Recall that $d_{o,\epsilon}(\xi,\hat{\gamma})=e^{-\epsilon(\xi,\hat{\gamma})_{o}}.$
 We have $$(\hat{\gamma},\xi)_{o}\geq \min \{ (\hat{\gamma},\gamma o)_{o},(\gamma o,\xi)_{o})\}-\delta.$$ On one hand, if $(\hat{\gamma},\gamma o)_{o}\geq (\gamma o,\xi)_{o}$, then $(\hat{\gamma},\xi)_{o}\geq (\gamma o,\xi )_{o}-\delta\geq A-\delta$ by (\ref{defsha}) and if $(\hat{\gamma},\gamma o)_{o}\leq (\gamma o,\xi)_{o}$, then $(\hat{\gamma},\xi)_{o}\geq (\hat{\gamma} ,\gamma o)_{o}-\delta \geq |\gamma|- M_{X}-\delta$ using $(\ref{hat}).$  
 
 Hence for the first integral we have:
 $$ \int _{\mathcal{O}_{A}(o,\gamma o)} e^{(\frac{1}{2}+t)Q\beta_{\xi}(o,\gamma o))}|w(\xi)-w(\hat{\gamma})|d\nu_{o}(\xi) \leq C \|w\|_{Lip}(e^{-\epsilon |\gamma|+\epsilon \delta+\epsilon M_{X}}+ e^{-\epsilon A+\epsilon \delta}) \phi_{t}(\gamma).$$
 
 Concerning the second term we follow the method using the discretization process as in Lemma \ref{ineqsigma}. We have:
 \begin{align*}
 \int _{\partial X \backslash \mathcal{O}_{A}(o,\gamma o)}e^{(\frac{1}{2}+t)Q\beta_{\xi}(o,\gamma o))}|w(\xi)-w(\hat{\gamma})| d\nu_{o}(\xi) 
 &\leq \|w\|_{Lip}\int_{\partial X \backslash \mathcal{O}_{A}(o,\gamma o))}e^{(\frac{1}{2}+t)Q\beta_{\xi}(o,\gamma o)}d\nu_{o}(\xi) \\
 &= \|w\|_{Lip}e^{-(\frac{1}{2}+t)Q|\gamma |}\int_{\partial X \backslash U_{A} }e^{(1+2t)Q (\xi,\gamma o)_{o}} d\nu_{o}(\xi)\\
 & \leq C\|w\|_{Lip}e^{-(\frac{1}{2}+t)Q|\gamma |} \frac{1-e^{2tQ A}}{1-e^{2tQ}}\\
 &\leq C \|w\|_{Lip}  \frac{e^{-tQ|\gamma |}}{\sinh(tQ)} e^{2tQ A} \phi_{t}(\gamma),
  \end{align*}
  
 where the last inequality follows from the left-hand side of the lower bound (\ref{HCHestimates}). 
 Therefore, gathering both terms, using an absorbing constant $C>0$ we obtain:
\begin{align*}
 \langle \pi_{t}(\gamma)\textbf{1}_{\partial X} , |w-w(\hat{\gamma})\textbf{1}_{\partial X} |\rangle &\leq C \|w\|_{Lip} \bigg( \frac{e^{-tQ|\gamma |}}{\sinh(tQ)} e^{2tQ A}+ e^{-\epsilon |\gamma|+\epsilon \delta+ \epsilon M_{X}}+ e^{-\epsilon A+\epsilon\delta}\bigg)\phi_{t}(\gamma)\\
 & \leq C \|w\|_{Lip} \bigg( \big(\frac{e^{-tQ|\gamma |}}{\sinh(tQ)} + e^{-\epsilon |\gamma|}\big)e^{2tQ A}+ e^{- \epsilon A}\bigg)\phi_{t}(\gamma)\\
 &=C \|w\|_{Lip} \bigg( \varepsilon_{t}(|\gamma|)e^{2tQ A}+ e^{-\epsilon A}\bigg)\phi_{t}(\gamma).
 \end{align*}
where for $x\in \mathbb{R}_{+}$ we set $\varepsilon_{t}(x):= \frac{e^{-tQx}}{\sinh(tQ)}+e^{-\epsilon x
 },$ and the proof is done.
\end{proof}
The following lemma extracts the property of the decay of the matrix coefficient corresponding to the second term in the above proof.
\begin{lemma}\label{decay3}
Let $t, A >0$. There exists $C$ such that for all $w\in C(\partial X)$ and for all $\gamma \in \Gamma$:
 $$\frac{\langle \pi_{t}(\gamma)\textbf{1}_{\partial X} ,\textbf{1}_{\partial X\backslash{\mathcal{O}_{A}(o,\gamma o)}} \rangle}{\phi_{t}(\gamma)}\leq C \frac{e^{-tQ|\gamma |}}{\sinh(tQ)} e^{2tQ A} .$$
 \end{lemma}

\subsection{The generalized Poisson transform}

Inspired by \cite[Section 6.1]{Boy}, we recall first the notion of kernel we shall consider and the notion of Dirac-Weierstrass family.

Let  $\nu$ be a probability measure on $\partial X$.
Consider a kernel $K:(x,\xi)\in X\times \partial X \mapsto K(x,\xi)\in \mathbb{R}$. We say that the family $K(x,\cdot)_{x\in X}$ is a Dirac-Weierstrass family if it satisfies 
\begin{enumerate}
\item $K(x,\xi)\geq 0$ for all $(x,\xi)\in X\times \partial X$.
\item  $\int_{\partial X}K(x,\xi)d\nu(\xi)=1.$
\item For any $\eta \in \partial X$ and for any $r>0$, we have $$\int_{\partial X \backslash B(\eta,r)}K(x,\xi)d\nu(\xi) \to 0\mbox{ as }x\to \eta.$$
\end{enumerate} 
The useful property of a Dirac-Weierstrass family is the following property.\\
Define for $f\in L^{1}(\partial X,\nu)$ 
\begin{equation}
K(f)(x):=\int_{\partial X}K(x,\xi)f(\xi)d\nu(\xi),
\end{equation}
and define also
\begin{equation}
\bar{K}(f):x\mapsto \left\{
\begin{array}{l}
  K(f)(x) \mbox{ if $x\in X$}\\
 f(x) \mbox{ if $x\in \partial X$}\end{array}
\right.
\end{equation}
\begin{prop}\label{conti} Assume that 
 $K(x,\cdot)_{x\in X}$ is a Dirac-Weierstrass family. Let $f\in C(\partial X)$. Then $\bar{K}(f)$ is continuous on $\overline{X}$.
\end{prop}


After this general discussion, consider the kernel for any $t\in \mathbb{R}$ defined as :
\begin{equation}
P_{t}:(x,\xi)\in X \times \partial X \mapsto \frac{e^{(\frac{1}{2}+t)Q \beta_{\xi}(o,x)}}{\phi_{t}(x)}\in \mathbb{R}^{+}.
\end{equation}

Thus, the corresponding ``generalized Poisson transform" is defined as :

\begin{equation}
\mathcal{P}_{t}:f\in L^{1}(\partial X, \nu_{o}) \mapsto \mathcal{P}_{t}(f)(x):= \int_{\partial X}f(\xi)P_{t}(x,\xi)d\nu_{o}(\xi)= \int_{\partial X}f(\xi)\frac{e^{(\frac{1}{2}+t)Q \beta_{\xi}(o,x)}}{\phi_{t}(x)}d\nu_{o}(\xi),
\end{equation}
and $\bar{\mathcal{P}_{t}}$ satisfies 
 
\begin{equation}
\bar{\mathcal{P}_{t}}(f) = \left\{
    \begin{array}{ll}
        \mathcal{P}_{t}(f)(x)& \mbox{if } x \in X \\
       f(x) & \mbox{if } x \in \partial X.
    \end{array}
\right.
\end{equation}
 The key point is the following observation.
 \begin{lemma}\label{DW}
 Let $t>0$. The kernel $\big(P_{t}(x,\cdot)\big)_{x\in X}$ is a Dirac-Weierstrass family.
 \end{lemma}
 \begin{proof}
 The proof is the same as \cite[Proposition 6.2]{Boy2}, using here Lemma \ref{decay3}.
 \end{proof}
 Note that Proposition \ref{conti} together with Lemma \ref{DW} implies :
 
 \begin{prop}\label{contP} Let $t>0$.
 If $f$ is in $C(\partial X)$, then $\bar{\mathcal{P}_{t}}(f)$ is a continuous function on the whole space $\overline{X}$.
 \end{prop}
Moreover, observe that for all $f\in C(\partial X)$ and for all $\gamma \in \Gamma$:

\begin{equation}\label{dirac}
D_{\gamma o}(  \mathcal{P}_{t}(f))=\dfrac{\langle \pi_{t}(\gamma)\textbf{1}_{\partial X},f\rangle}{\phi_{t}(\gamma)}. 
\end{equation}

We deduce the following proposition:
\begin{prop}\label{density}

For any  $R>0$ large enough, there exists a sequence of  measures $\beta_{n,R}:\Gamma \rightarrow \mathbb{R}^{+}$,    supported on $S^{\Gamma}_{n,R}$, satisfying $\beta_{n,R}(\gamma)\leq C /|S^{\Gamma}_{n,R}|$ for some $C>0$ independent of $n$  such that
for all $t> 0  $,   for all $f,g\in C(\overline{X})$, for all $w\in L^{2}(\partial X,\nu_{o})$:
$$\sum_{\gamma \in S^{\Gamma}_{n,R}}\beta_{n,R}(\gamma) f(\gamma   o)  \frac{\langle \pi_{t}(\gamma)\textbf{1}_{\partial X},w\rangle }{\phi_{t}(\gamma)}\to  \langle f_{|_{\partial X}}  , w\rangle, $$
as $n\to +\infty$. Thus $\mathcal{F}_{t,\textbf{1}_{\partial X}}$ is dense in $L^{2}(\partial X,\nu_{o})$ for all $t>0.$
\end{prop}
\begin{proof}
Let $R>0$ large enough so that Theorem \ref{equi} holds and let $t>0$ ensuring that Proposition \ref{contP} holds. Let $f\in C(\overline{X})$ and consider the sequence of vectors $$f_{n}:=\sum_{\gamma \in S^{\Gamma}_{n,R}}\beta_{n,R}(\gamma) f(\gamma   o)  \frac{\pi_{t}(\gamma)\textbf{1}_{\partial X} }{\phi_{t}(\gamma)} \in \mathcal{F}_{t,\textbf{1}_{\partial X}}.$$
Observe by (\ref{dirac}) that for all $w\in L^{2}(\partial X,\nu_{o})$

$$ \langle f_{n},w\rangle =\sum_{\gamma \in S^{\Gamma}_{n,R}}\beta_{n,R}(\gamma) D_{\gamma o}\big(  f \cdot \bar{ \mathcal{P}_{t}}(w)\big),$$
where for all $x\in \overline{X}$, $\big(f \cdot \bar{ \mathcal{P}_{t}}(w)\big)(x)=f (x) \bar{ \mathcal{P}_{t}}(w)(x)$.\\
Therefore, for all $w\in C(\partial X)$, Theorem \ref{equi} combined with Proposition \ref{contP} imply that 
$$ \langle f_{n},w\rangle \to \langle f,w\rangle= \langle f_{|_{\partial X}}  , w\rangle$$ as $n\to +\infty$.
Since Inequality \ref{RDgeneral} implies that there exists $C>0$ such that for all $n\in \mathbb{N}$: $$|\langle f_{n},w \rangle |\leq C \|f\|_{\infty}\|w\|_{2},$$ the above convergence holds for any $w\in L^{2}(\partial X,\nu_{o})$, and the proof is done.

\end{proof}

 \subsection{The Riesz kernel}
 
 Define for $0<t<\frac{1}{2}$:
\begin{equation}
R_{t}:(\xi,x)\in \partial X\times X\mapsto \dfrac{e ^{(1-2t)Q (\xi,x)_{o}}}{\int_{\partial X}e ^{(1-2t)Q (\xi',x)_{o}}d\nu_{o}(\xi')}.
\end{equation}
Assuming that the Gromov product extends continuously to the bordification one can define:

\begin{equation}
R_{t}(\xi,\eta):= \dfrac{e ^{(1-2t)Q (\xi,\eta)_{o}}}{\int_{\partial X}e ^{(1-2t)Q (\xi',\eta)_{o}}d\nu_{o}(\xi')}=\dfrac{1}{d_{o,\epsilon} ^{(1-2t)D} (\xi,\eta)}\times\dfrac{1}{\sigma_{t}(\eta)}.
\end{equation}
Note that $R_{t}$ is continuous on $\partial X\times \partial X\backslash \Delta$ and $R_{t}\in L^{1}(\partial X \times \partial X,\nu_{o}\times \nu_{o})$ if and only if  $t>0$. Besides :
\begin{enumerate}
\item $\int_{\partial X}R_{t}(\xi,\eta) d\nu_{o}(\xi)=1.$
\item $R_{t}\geq 0.$
\item $R_{t}(\xi,\eta)=R_{t}(\eta,\xi).$
 \end{enumerate}
 
Now consider for $f\in L^{1}(\partial X,\nu_{o})$ the following kernel, as an extension of the Riesz operator to the whole space $X\cup \partial X$:
\begin{equation}
\mathcal{R}_{t}(f)(x):=\int_{\partial X} f(\xi)R_{t}(\xi,x)d\nu_{o}(\xi),
\end{equation}
with $x\in X\cup \partial X$.
Indeed, the assumption of working with $\epsilon$-good strongly hyperbolic spaces is crucial for the following proposition :
\begin{prop}\label{contR}
Let $t>0$.
Assume that $f\in C(\partial X)$. Then $\mathcal{R}_{t}(f)\in C(\overline {X})$.
\end{prop}
\begin{proof}
Pick $f\in C(\partial X)$ and consider the function $x\in X\mapsto \mathcal{R}_{t}(f)(x)=M_{\sigma^{-1}_{t} } \mathcal{I}_{t}(f).$ Proposition \ref{cont2}  and Proposition \ref{compactope} implies that the function $x\in X\mapsto \mathcal{R}_{t}(f)(x)$ is a continuous function. 
\end{proof}

Observe the fundamental correspondence between harmonic analysis on $\overline{X}$ and representation theory of $\Gamma$ given for $\gamma \in \Gamma$ for all $f\in C(\overline{X})$ via:
\begin{equation}\label{dirac}
D_{\gamma^{-1}o}(\mathcal{R}_{t}(f))=\dfrac{\langle \pi_{t}(\gamma)f,\textbf{1}_{\partial X}\rangle}{\phi_{t}(\gamma)}. 
\end{equation}

\section{Proofs }\label{section7}

\subsection{Proof of the main convergence theorem}

We give a single proof for Theorem \ref{theo1} and \ref{theo1'}.
\begin{proof}
For $t>0$, we set $\mathcal{E}_{t,0}=L^{2}(\partial X ,\nu_{o})$ and $\mathcal{E}_{t,2}=\mathcal{K}_{t}$. \\
If there exists $0<t\leq {1\over 2}$ such that $\mathcal{I}_{t}$ is positive, we set $\mathcal{E}_{t,1}=\mathcal{H}_{t}$.\\ 
\textbf{Step 1:} Uniform boundedness.\\
Let $t,t'>0$ and $i,j\in\{0,2\}.$
Consider for $R>0$ the quadrilinear form defined on $\mathcal{E}_{t,i}\times \overline{\mathcal{E}_{t,i}} \times \overline{\mathcal{E}_{t,j}} \times \mathcal{E}_{t,j}$  as:
$$B_{n,R}(v,w,v',w'):=\sum_{\gamma \in S^{\Gamma}_{n,R}} \beta_{n,R}(\gamma)\frac{\langle \pi_{t}(\gamma)v,w\rangle_{ \mathcal{E}_{t,i} }  }{\phi_{t}(\gamma)}\frac{\overline{\langle \pi_{t'}(\gamma)v',w'\rangle}_{ \mathcal{E}_{t',j} } }{\phi_{t'}(\gamma)}.$$
The  Property RD used in \cite{BG} to prove Schur's orthogonality relations is replaced by Inequality (\ref{RDgeneral}), Theorem \ref{Hspectral} and \ref{Uspectral}. Let $R>0$ large enough.
Pick two vectors $v,w\in \mathcal{E}_{t,i}$ of norm one, and consider for each $n$ the function $f_{n}$ supported on $S^{\Gamma}_{n,R}$ defined as: $$f_{n}=\sum_{\gamma \in S^{\Gamma}_{n,R}}\overline{\langle \pi_{t}(\gamma)v,w\rangle_{\mathcal{E}_{t,i}}} D_{\gamma}.$$ The spectral inequality (\ref{RDgeneral}) for $i=0$ and Theorem \ref{Hspectral} imply for $i=2$:
\begin{align}
\langle \pi_{t}(f_{n})v,w\rangle_{\mathcal{E}_{t,i}} \leq C \|f_{n}\|_{\ell^{2}} \omega_{t}(nR).
\end{align}
Therefore: $$ \bigg(\sum_{\gamma \in S^{\Gamma}_{n,R}} |\langle \pi_{t}(\gamma)v,w\rangle_{\mathcal{E}_{t,i}}|^{2}\bigg)^{\frac{1}{2}}\leq C  \omega_{t}(nR).$$

For all $n$, for all $v,w\in\mathcal{E}_{t,i}$ and $v,',w'\in\mathcal{E}_{t',j}$ of norm one for $t,t'> 0$ and for $i,j=0,2$ Cauchy-Schwarz inequality implies:
$$\bigg|\sum_{\gamma \in S^{\Gamma}_{n,R}}  \langle \pi_{t}(\gamma)v,w\rangle_{\mathcal{E}_{t,i}} \overline{\langle \pi_{t'}(\gamma)v',w'\rangle}_{\mathcal{E}_{t',j}}\bigg| \leq C  \omega_{t}(nR)  \omega_{t'}(nR).$$
 Since there exists $C>0$ such that for $\gamma \in S^{\Gamma}_{n,R}$ and for any $t>0$  we have $$ C^{-1}\frac{ \omega_{t}(nR)}{|S^{\Gamma}_{n,R}|} \leq \phi_{t}(\gamma) \leq C\frac{ \omega_{t}(nR)}{|S^{\Gamma}_{n,R}|} ,$$  and since there exists $C>0$ such that for any $\gamma \in S^{\Gamma}_{n,R}$ we have: $$\beta_{n,R}(\gamma)\leq \frac{C}{|S^{\Gamma}_{n,R}|},$$ we obtain for all non negative integers $n$ and for all positive functions $v,w\in \mathcal{E}_{t,i}$ and $ v',w'\in \mathcal{E}_{t',j}$ of norm one:
$$\bigg|\sum_{\gamma \in S^{\Gamma}_{n,R}} \beta_{n,R}(\gamma)\frac{\langle \pi_{t}(\gamma)v,w\rangle_{\mathcal{E}_{t,i}} }{\phi_{t}(\gamma)}\frac{\overline{\langle \pi_{t'}(\gamma)v',w'\rangle}_{\mathcal{E}_{t',j}}}{\phi_{t'}(\gamma)}\bigg|\leq C.$$

In an other words $$\sup_{n}\|B_{n,R}\|<+\infty,$$ where $\|\cdot\|$ denotes the standard operator norm on the space of complex quadrilinear form on $\mathcal{E}_{t,i}\times \overline{\mathcal{E}_{t,i}} \times \overline{\mathcal{E}_{t,j}} \times \mathcal{E}_{t,j}.$
\\ 
We study now the case $i=1$. Assume there exists $t,s>0$ such that $\mathcal{I}_{t},\mathcal{I}_{s}$ are positive. By the same arguments at the beginning of the proof of \textbf{Step 1} dealing with $\mathcal{E}_{t,i}$ for $i=0,2$,    Theorem \ref{Uspectral} implies that exists $C>0$ such that for all non negative integers $n$ and for all $(v,w) \in \mathcal{H}_{t}\times \mathcal{H}'_{t}$ and $(v',w' )\in \mathcal{H}_{s}\times \mathcal{H}'_{s}$ unit vectors
$$\bigg|\sum_{\gamma \in S^{\Gamma}_{n,R}} \beta_{n,R}(\gamma)\frac{\langle \mathcal{I}_{t}\pi_{t}(\gamma)v,w\rangle }{\phi_{t}(\gamma)}\frac{\overline{\langle \mathcal{I}_{s} \pi_{s}(\gamma)v',w'\rangle}}{\phi_{s}(\gamma)}\bigg|\leq C,$$ equivalently
$$\sup_{n}\bigg|\sum_{\gamma \in S^{\Gamma}_{n,R}} \beta_{n,R}(\gamma)\frac{\langle \pi_{t}(\gamma)v,w\rangle_{\mathcal{E}_{t,1}} }{\phi_{t}(\gamma)}\frac{\overline{\langle  \pi_{s}(\gamma)v',w'\rangle_{\mathcal{E}_{s,1}}}}{\phi_{s}(\gamma)}\bigg|\leq C.$$

\textbf{Step 2:} Computation of the limit.\\
Let $f,g\in C(\overline{ X})$ and let $v,w,v',w' \in Lip(\partial X)$ be four real positive functions. 

Write: 
$$\langle \pi_{t}(\gamma)v,w\rangle=\langle \pi_{t}(\gamma)v,w-\mathcal{P}_{t}(w)(\gamma o)\textbf{1}_{\partial X}\rangle + \mathcal{P}_{t}(w)(\gamma o)\langle \pi_{t}(\gamma)v,\textbf{1}_{\partial X}\rangle,$$
Then:
\begin{align*}
\frac{\langle\pi_{t}(\gamma)v,w\rangle}{\phi_{t}(\gamma)}&=\frac{\langle \pi_{t}(\gamma)v,w-\mathcal{P}_{t}(w)(\gamma o)\textbf{1}_{\partial X}\rangle }{\phi_{t}(\gamma)}+\mathcal{P}_{t}(w)(\gamma o)\mathcal{R}_{t}(v)(\gamma^{-1} o).
\end{align*}
Since $\pi_{t}$ preserves the cone of positive functions, we obtain: 
$$\bigg|\frac{\langle\pi_{t}(\gamma)v,w\rangle}{\phi_{t}(\gamma)}-\mathcal{P}_{t}(w)(\gamma o)\mathcal{R}_{t}(v)(\gamma^{-1} o) \bigg| \leq\|v\|_{\infty} \frac{\langle\pi_{t}(\gamma)\textbf{1}_{\partial X},|w-\mathcal{P}_{t}(w)(\gamma o)\cdot \textbf{1}_{\partial X}|\rangle}{\phi_{t}(\gamma)}.$$
Now concerning the term of the right-hand side of the above inequality: \begin{align*} \frac{\langle\pi_{t}(\gamma)\textbf{1}_{\partial X},|w-\mathcal{P}_{t}(w)(\gamma o)\cdot \textbf{1}_{\partial X}|\rangle}{\phi_{t}(\gamma)}&\leq  \frac{\langle\pi_{t}(\gamma)\textbf{1}_{\partial X},|w-w(\hat{\gamma} )\textbf{1}_{\partial X}|\rangle}{\phi_{t}(\gamma)}+ \frac{\langle\pi_{t}(\gamma)\textbf{1}_{\partial X},|\big(w(\hat{\gamma} )-\mathcal{P}_{t}(w)(\gamma o)\big)\textbf{1}_{\partial X} |\rangle}{\phi_{t}(\gamma)}\\
&=\frac{\langle\pi_{t}(\gamma)\textbf{1}_{\partial X},|w-w(\hat{\gamma} )\textbf{1}_{\partial X}|\rangle}{\phi_{t}(\gamma)}+|w(\hat{\gamma} )-\mathcal{P}_{t}(w)(\gamma o) |.
\end{align*}

Apply Lemma \ref{decay2} to get a constant $C>0$ such that for all $\gamma \in \Gamma$ and for all $A>0$:

$$\frac{\langle\pi_{t}(\gamma)\textbf{1}_{\partial X},|w-w(\hat{\gamma} )\textbf{1}_{\partial X}|\rangle}{\phi_{t}(\gamma)}\leq C\|w\|_{Lip}\big((\varepsilon_{t}(\gamma)e^{2tQ A}+e^{-\epsilon A} \big).$$
 Taking the average over $S^{\Gamma}_{n,R}$ we obtain, using Theorem \ref{equi} and \ref{equib}:
 $$\limsup_{n\to +\infty}\sum_{\gamma \in S^{\Gamma}_{n,R}}\beta_{n,R}(\gamma)|w(\hat{\gamma} )-\mathcal{P}_{t}(w)(\gamma o) |=0.$$
 Moreover, for all $A>0$:
 $$\limsup_{n\to +\infty}\sum_{\gamma \in S^{\Gamma}_{n,R}}\beta_{n,R}(\gamma)\frac{\langle\pi_{t}(\gamma)\textbf{1}_{\partial X},|w-w(\hat{\gamma} )\textbf{1}_{\partial X}|\rangle}{\phi_{t}(\gamma)}\leq C\|w\|_{Lip}e^{- \epsilon A},$$ and thus let $A\to +\infty$.
 
Hence, the error term disappears and we obtain:
\begin{align*}
\lim_{n\to +\infty}&\sum_{\gamma \in S^{\Gamma}_{n,R}}\beta_{n,R}(\gamma)f(\gamma  o)g(\gamma^{-1} o)\frac{\langle\pi_{t}(\gamma)v,w\rangle}{\phi_{t}(\gamma)}\frac{\langle\pi_{t'}(\gamma)v',w'\rangle}{\phi_{t'}(\gamma)}\\ & =\lim_{n\to +\infty}\sum_{\gamma \in S^{\Gamma}_{n,R}}\beta_{n,R}(\gamma) f(\gamma  o)g(\gamma^{-1} o)\mathcal{P}_{t}(w)(\gamma o)\mathcal{R}_{t}(v)(\gamma^{-1} o)\mathcal{P}_{t'}(w')(\gamma o)\mathcal{R}_{t'}(v')(\gamma^{-1} o).
\end{align*}
Using Proposition \ref{contP} and \ref{contR} with Theorem \ref{equi} we have for $f,g\in C(\overline{ X})$ and for $v,w,v',w' \in Lip(\partial X)$:

\begin{equation}\label{convergenceschur}
\lim_{n\to +\infty}\sum_{\gamma \in S^{\Gamma}_{n,R}}\beta_{n,R}(\gamma)f(\gamma  o)g(\gamma^{-1} o)\frac{\langle\pi_{t}(\gamma)v,w\rangle}{\phi_{t}(\gamma)}\frac{\overline{\langle\pi_{t'}(\gamma)v',w'\rangle}}{\phi_{t'}(\gamma)}=\langle g_{|_{\partial X}} \mathcal{R}_{t}(v), \mathcal{R}_{t'}(v')\rangle \overline{\langle  w,f_{|_{\partial X}} w' \rangle}.
\end{equation}
\textbf{Step 3:} Conclusion.\\

For $i,j=0,2$, the limit proved above in (\ref{convergenceschur}) together with the uniform bound proved in \textbf{Step 1} for $\mathcal{E}_{t,i},\mathcal{E}_{t',j}$ imply eventually that  for $f,g\in C(\overline{ X})$ and for 
 $v,w\in \mathcal{E}_{t,i}$ and $v',w' \in \mathcal{E}_{t',j}$: 
\begin{align*}
 \lim_{n\to +\infty}\sum_{\gamma \in S^{\Gamma}_{n,R}}\beta_{n,R}(\gamma)f(\gamma  o)g(\gamma^{-1} o)\frac{\langle\pi_{t}(\gamma)v,w\rangle_{\mathcal{E}_{t,i}}}{\phi_{t}(\gamma)}&\frac{\overline{\langle\pi_{t'}(\gamma)v',w'\rangle_{\mathcal{E}_{t',j}}}}{\phi_{t'}(\gamma)}\\
 &=\langle g_{|_{\partial X}} \mathcal{R}_{t}(v), \mathcal{R}_{t'}(v')\rangle \overline{ \langle  \mathcal{I}^{i}_{t}(w),f_{|_{\partial X}} \mathcal{I}^{j}_{t'}(w') \rangle}.
 \end{align*}
For the case $i=1$, we obtain for $f,g\in C(\overline{ X})$ and for 
 $(v,w)\in \mathcal{E}_{t,1} \times \mathcal{E}'_{t,1}$ and $(v',w') \in \mathcal{E}_{s,1}\times \mathcal{E}'_{s,1}$: 
\begin{align*}
 \lim_{n\to +\infty}\sum_{\gamma \in S^{\Gamma}_{n,R}}\beta_{n,R}(\gamma)f(\gamma  o)g(\gamma^{-1} o)\frac{\langle\pi_{t}(\gamma)v,w\rangle_{\mathcal{E}_{t,1}}}{\phi_{t}(\gamma)}&\frac{\overline{\langle\pi_{s}(\gamma)v',w'\rangle_{\mathcal{E}_{s,1}}}}{\phi_{s}(\gamma)}\\
 &=\langle g_{|_{\partial X}} \mathcal{R}_{t}(v), \mathcal{R}_{s}(v')\rangle \overline{ \langle  \mathcal{I}_{t}(w),f_{|_{\partial X}} \mathcal{I}_{s}(w') \rangle}.
 \end{align*}

\end{proof}

\subsection{Theorems \`a la Bader-Muchnik}

To prove irreducibility of representations our main tool will be results {\it \`a la Bader Muchnik} (Corollary \ref{BML2}). They are all consequences of Theorem \ref{theo1} and \ref{theo1'}. 
 \begin{proof}[Proof of Corollary \ref{BML2}.]
 Apply directly Theorem \ref{theo1} with $\mathcal{E}_{t',j}=L^{2}(\partial X,\nu_{o})$ and $v'=w'=\textbf{1}_{\partial X}$ to obtain for all $f,g\in C(\overline{X})$  and for all $v,w\in \mathcal{E}_{t,i}$ with $i\in \{ 0,2\}$ and for all $(v,w)\in \mathcal{E}_{t,1}\times \mathcal{E}'_{t,i}$:
 
$$
 \lim_{n\to +\infty}\sum_{\gamma \in S^{\Gamma}_{n,R}}\beta_{n,R}(\gamma)f(\gamma  o)g(\gamma^{-1} o)\frac{\langle\pi_{t}(\gamma)v,w\rangle_{\mathcal{E}_{t,i}}}{\phi_{t}(\gamma)}=\langle g_{|_{\partial X}} \mathcal{R}_{t}(v), \textbf{1}_{\partial X}\rangle \overline{ \langle  \mathcal{I}^{i}_{t}(w),f_{|_{\partial X}} \rangle}.$$

Since $ \mathcal{I}^{i}_{t}$ is self adjoint for $i\in \{ 0,1,2\}$, we obtain: $$ \langle  \mathcal{I}^{i}_{t}(w),f_{|_{\partial X}} \rangle= \langle  w,\mathcal{I}^{i}_{t}(f_{|_{\partial X}}) \rangle=\langle w,f_{|_{\partial X}}  \rangle_{\mathcal{E}_{t,i}},$$
and the proof is done.

 \end{proof}
 
\subsection{A first family of Irreducible Hilbertian representations}
Fix $t>0$.
\begin{proof}[Proof of Corollary \ref{coro1}.]
We will prove that $(\pi_{-t},\mathcal{W}_{t})$ is irreducible. Recall that 
 $ \pi_{-t}(f). \sigma_{t}=\mathcal{I}_{t}(\pi_{t}(f)\textbf{1}_{\partial X} )$, for all $f\in \mathbb{C}[\Gamma]$. \\
 Thus, since $\textbf{1}_{\partial X}$ is cyclic in $L^{2}(\partial X,\nu_{o})$ for $\pi_{t}$, the function $\sigma_{t}$ is cyclic in $\mathcal{W}_{t}$ for $\pi_{-t}$.\\
 Now, let $0\neq K\subset \mathcal{W}_{t}$ a $L^{2}$-closed subspace stable par $\pi_{-t}$. Let $R>0$ large enough. For all $w\in \mathcal{W}_{t}$ and $v\in L^{2}(\partial X,\nu_{o})$ define for all $n$
  $$w_{n}:=\sum_{\gamma \in S^{\Gamma}_{n,R}}\beta_{n,R}(\gamma)  \sigma_{t}(\gamma^{-1} o)  \frac{\pi_{-t}(\gamma^{-1}) w }{\phi_{t}(\gamma)} \in K.$$
  Theorem \ref{BML2} implies that as $n\to +\infty$:

$$ \sum_{\gamma \in S^{\Gamma}_{n,R}}\beta_{n,R}(\gamma)  \sigma_{t}(\gamma^{-1} o) \frac{\langle  v, \pi_{-t}(\gamma^{-1}) w\rangle }{\phi_{t}(\gamma)}\to \langle \mathcal{I}_{t}(v),\textbf{1}_{\partial X}\rangle \langle  \textbf{1}_{\partial X} ,w \rangle=\langle v,\mathcal{I}_{t}(\textbf{1}_{\partial X})\rangle \langle  \textbf{1}_{\partial X} ,w \rangle,$$

 writing 

$$\langle v,\mathcal{I}_{t}(\textbf{1}_{\partial X})\rangle \langle  \textbf{1}_{\partial X} ,w \rangle= \langle v,T_{t}(w)\rangle,$$ we have:
$$T_{t}:v \in \mathcal{W}_{t}\mapsto \langle \textbf{1}_{\partial X},v \rangle \sigma_{t}\in \mathcal{W}_{t}.$$
In other words, $w_{n}\to T_{t}(w)$ for the weak convergence.
Thus, for all $w\in K$ we have $T_{t}(w)=\langle \textbf{1}_{\partial X},w\rangle \sigma_{t}\in K$. If for all $w\in K$, $\langle \textbf{1}_{\partial X},w\rangle =0$, we would have that   for all $\gamma \in \Gamma$ that $\langle  \textbf{1}_{\partial X},\pi_{-t}(\gamma)w \rangle=0= \langle \pi_{t}(\gamma) \textbf{1}_{\partial X},w \rangle.$ And therefore, since $\textbf{1}_{\partial X}$ is cyclic for $(\pi_{t},L^{2}(\partial X,\nu_{o}))$, the vector $w$ has to be $0$. Hence, there exists $0\neq w\in K$ such that $\langle \textbf{1}_{\partial X},w\rangle \sigma_{t}\in K$ and so $\sigma_{t}\in K$. Since $\sigma_{t}$ is cyclic for $\pi_{-t}$ in $\mathcal{W}_{t}$, the proof is done.
\end{proof}
Since the dual representation of $(\pi_{t},L^{2}(\partial X,\nu_{o})/\mathcal{V}_{t})$ is $(\overline{\pi_{-t}},\overline{\mathcal{W}_{t}})$, we deduce immediately from the above theorem the irreducibility of $(\pi_{t},L^{2}(\partial X,\nu_{o})/\mathcal{V}_{t}))$.

\subsection{The Hilbertian  complementary series}

Fix $t>0$.
\begin{proof}[Proof of Corollary \ref{coro2}]

Apply Theorem \ref{BML2} with $\mathcal{E}_{t,i}=\mathcal{K}_{t}$, with $v=\textbf{1}_{\partial X}$ and $g=\textbf{1}_{\partial X}$ to obtain 
$$\sum_{\gamma \in S^{\Gamma}_{n,R}}\beta_{n,R}(\gamma) f(\gamma   o) \frac{\langle \pi_{t}(\gamma)\textbf{1}_{\partial X} ,w\rangle_{\mathcal{K}_{t}} }{\phi_{t}(\gamma)}\to  \overline{\langle  f_{|_{\partial X}} ,w \rangle_{\mathcal{K}_{t}}}. $$

In other words, $\textbf{1}_{\partial X}$ is cyclic for $\mathcal{K}_{t}$.

 Consider now  a proper subspace $F\subset \mathcal{K}_{t}$ invariant by $\pi_{t}$. Theorem \ref{BML2} implies, for $f=\textbf{1}_{\overline{X}}$ and $g=\sigma_{t}$ that for all $v\in F$ we have $\langle \textbf{1}_{\partial X} ,\mathcal{I}_{t}(v)\rangle \textbf{1}_{\partial X}\in F.$ If for all $v$ in $F$, $\langle \textbf{1}_{\partial X},\mathcal{I}_{t}(v)\rangle=0$, it would follow by Proposition \ref{intertwin} that for all $\gamma \in \Gamma$ that $\langle\pi_{t}(\gamma) \textbf{1}_{\partial X},\mathcal{I}_{t}(v)\rangle=0$. Then using the fact that $\textbf{1}_{\partial X}$ is cyclic for $(\pi_{t},L^{2}(\partial X,\nu_{o}))$, we would have $\mathcal{I}_{t} (v)=0$, that is $v=0$ in $\mathcal{K}_{t}$. Hence there exists $0\neq v\in F$ such that $\langle \textbf{1}_{\partial X} ,\mathcal{I}_{t}(v)\rangle\neq 0$ and thus $ \textbf{1}_{\partial X}\in F$. Since $\textbf{1}_{\partial X}$ is cyclic for $(\pi_{t},\mathcal{K}_{t})$ we obtain the irreducibility result.\\
 Let $t'\neq t$ ; we may assume $t'<t$. Suppose that $(\pi_{t'},\mathcal{K}_{t'})$ is equivalent to $(\pi_{t},\mathcal{K}_{t})$ via an invertible transformation $U:\mathcal{K}_{t'} \rightarrow \mathcal{K}_{t}$, namely $U^{-1}\pi_{t'}(\gamma)U=\pi_{t}(\gamma)$ for all $\gamma$.\\
On one hand, we have (for $R$ large enough)  for all $n:$  $$0<\sum_{\gamma \in S^{\Gamma}_{n,R}}\beta_{n,R}(\gamma)  \frac{\langle U^{-1} \pi_{t'}(\gamma)U \textbf{1}_{\partial X} ,\textbf{1}_{\partial X} \rangle_{\mathcal{K}_{t'}} }{\phi_{t}(\gamma)}=\sum_{\gamma \in S^{\Gamma}_{n,R}}\beta_{n,R}(\gamma)  \frac{\langle \pi_{t'}(\gamma)U(\textbf{1}_{\partial X}) ,U^{-1 *}(\textbf{1}_{\partial X}) \rangle_{\mathcal{K}_{t'}} }{\phi_{t}(\gamma)}.$$
And on the other hand:
\begin{align*}
\bigg|\sum_{\gamma \in S^{\Gamma}_{n,R}}\beta_{n,R}(\gamma)  \frac{\langle \pi_{t'}(\gamma)U(\textbf{1}_{\partial X}) ,U^{-1 *}(\textbf{1}_{\partial X}) \rangle_{\mathcal{K}_{t}} }{\phi_{t}(\gamma)}\bigg|&=\bigg|\sum_{\gamma \in S^{\Gamma}_{n,R}}\beta_{n,R}(\gamma)  \frac{\langle \pi_{t'}(\gamma)U(\textbf{1}_{\partial X}) ,U^{-1 *}(\textbf{1}_{\partial X}) \rangle_{\mathcal{K}_{t}} }{\phi_{t'}(\gamma)}\frac{\phi_{t'}(\gamma)}{\phi_{t}(\gamma)}\bigg|\\
&\leq C\frac{\phi_{t'}(nR)}{\phi_{t}(nR)}\to 0,
\end{align*}
 as $n \to +\infty$, where the last inequality follows from estimates of Theorem \ref{Hspectral}, and the limit follows from Proposition \ref{estimatesspherical} together with $t'<t$. Whence, the representations $(\pi_{t},\mathcal{K}_{t})$  and $(\pi_{t'},\mathcal{K}_{t'})$ are not equivalent for $t\neq t'$.

\end{proof}

\subsection{The unitary complementary series for $0<t\leq \frac{1}{2}$}
\begin{proof}[Proof of Corollary \ref{coro3}]
We only have to prove (\ref{3ofcoro}) of Corollary \ref{coro3}. 
 Theorem \ref{theo1'} implies for all $(v,w)\in \mathcal{H}_{t}\times \mathcal{H}'_{t}$  that
 $$\sum_{\gamma \in S^{\Gamma}_{n,R}}\beta_{n,R}(\gamma) \sigma_{t}(\gamma^{-1} o)    \frac{ \langle \mathcal{I}_{t}\pi_{t}(\gamma)v ,w\rangle }{\phi_{t}(\gamma)} \to \langle \mathcal{I}_{t}(v),\textbf{1}_{\partial X} \rangle  \langle \textbf{1}_{\partial X},\mathcal{I}_{t}(w) \rangle. $$
 Let $0\neq V\subset \mathcal{H}_{t}$ invariant by $\pi_{t}$ and closed. Then for all $v\in V, \mathcal{I}_{t}(v_{n}) \to \langle  \mathcal{I}_{t}(v),\textbf{1}_{\partial X} \rangle \sigma_{t}$ with respect to the weak* topology of $\mathcal{H}_{t}$ with $v_{n}=\sum_{\gamma \in S^{\Gamma}_{n,R}}\beta_{n,R}(\gamma) \sigma_{t}(\gamma^{-1} o)    \frac{ \mathcal{I}_{t}\pi_{t}(\gamma)v  }{\phi_{t}(\gamma)}\in V\subset \mathcal{H}_{t}.$ Therefore we have 
 $\mathcal{J}_{t}\mathcal{I}_{t}(v_{n})=v_{n}\to \mathcal{J}_{t}( \mathcal{I}_{t}(v),\textbf{1}_{\partial X} \rangle \sigma_{t})= \langle \mathcal{I}_{t}(v),\textbf{1}_{\partial X} \rangle \textbf{1}_{\partial X}\in V.$ In other words, $\langle v,\textbf{1}\rangle_{\mathcal{H}_{t}}\textbf{1}_{\partial X} \in V.$ Since $\textbf{1}_{\partial X}$ is cyclic by construction and since $\pi_{t}$ is unitary then $V$ has to be $\mathcal{H}_{t}$ and the proof is done.
\end{proof}

\appendix
\section{A unitary representation}
Let $\Gamma$ be a non elementary Gromov hyperbolic group acting geometrically  on $(X,d_\Sigma)$ its Cayley graph endowed with $d_\Sigma$ a word metric associated with some generating set $\Sigma$. Let $\partial X$ be its Gromov boundary endowed with some visual metric $d_{\partial X}$. Following the notation of \cite{Car}, denote by $J_{AR}(\partial X)$ its \emph{Ahlfors regular conformal gauge}, see also \cite{BP}. We recall briefly its definition: $$J_{AR}(\partial X):=\{ \mbox{$d'$ Ahlfors regular metric on $\partial X $ such that  $d' \sim_{q.s.} d_{\partial X}$}\}, $$ 
where two metrics $d_{\partial X}$ and $d'$ in $\partial X$ are quasisymmetrically equivalent (not. $d' \sim_{q.s.} d_{\partial X}$) if the
identity map $id: (\partial X, d_{\partial X}) \rightarrow (\partial X, d')$ is a quasisymmetric homeomorphism. Recall that a
homeomorphism $  h: ( Y, d) \rightarrow (Z, d') $ between two metric spaces is quasisymmetric if there
is an increasing homeomorphism $\eta:\mathbb{R}^{+}\rightarrow \mathbb{R}^{+}$ - called a distortion function - such that :

$$ \frac{d'(h(x),h(z))}{d'(h(y),h(z))}\leq \eta \bigg(\frac{d(x,z)}{d(y,z)}\bigg),$$ for all $x,y,z \in Y$ with $y\neq z$. Moreover, a metric $d'$ is Ahlfors regular of dimension $Q$ on $\partial X$ if there exists  $Q>0$ and a Radon measure on $\partial X$ such that there exists $C>0$  satisfying for every balls on the boundary $B_{\partial X}(\xi,r)$ the inequalities $C^{-1}r^{Q}\leq \mu(B_{\partial X}(\xi,r))\leq C r^{Q}$.

\subsection{Besov spaces }\label{besov}

For any $d_{o}\in J_{AR}(\partial X)$ of dimension $Q$, let $\nu_{o}$ be the corresponding Patterson-Sullivan measure. Define the Sobolev-Besov spaces \`a la Bourdon-Pajot \cite{BP}, given by
\begin{equation}
\mathcal{B}_{t}(d_{o}):= \lbrace v \mbox{ measurable } :\partial X \to \mathbb{C} | \|v\|^{2}_{\mathcal{B}_{t}}=\int_{\partial X \times \partial X} \frac{|v(\xi)-v(\eta)|^{2}}{d^{(1-2t)Q}_{o}(\xi,\eta)} d\nu_{o}(\xi)d\nu_{o}(\eta)<+\infty\rbrace / \sim,
\end{equation} 
where $v\sim w$ if and only if $v-w$ equals to a constant a.e. with respect to $dm_{o,t}(\xi,\eta)=d\nu_{o}(\xi)d\nu_{o}(\eta)/d^{(1-2t)Q}_{o}(\xi,\eta) $. Note that for any couple of measurable functions $(v,w)$  the inner product is given by

\begin{equation}
\langle v,w \rangle_{\mathcal{B}_{t}}= \int_{\partial X \times \partial X} \frac{(v(\xi)-v(\eta))\overline{(w(\xi)-w(\eta))}}{d^{(1-2t)Q}_{o}(\xi,\eta)}d\nu_{o}(\xi)d\nu_{o}(\eta).
\end{equation}

 For $t=-\frac{1}{2}$, we recover the space denoted by $B_{2}(d)$ in \cite[Section 0.1]{BP}.

 \begin{prop}\label{sobo} 
\mbox{  }
\begin{enumerate}
\item Let $t<t'$. We have the following embeddings:
$$\mathcal{B}_{t}\hookrightarrow \mathcal{B}_{t'}.$$
\item The pair $(\mathcal{B}_{t},\overline{\mathcal{B}_{-t}})$ is a $\mathbb{C}$-pairing with respect to $\langle \cdot,\cdot \rangle_{\mathcal{B}_{0}}$.
\item $\mathcal{B}_{\frac{1}{2}}$ is isomorphic  to the Hilbert space $L^{2}(\partial X)\ominus \mathbb{C}$.

\end{enumerate}
\end{prop}

Choose $d_{o,\epsilon}$ an $\epsilon$-good metric for some $\epsilon>0$. The fundamental observation is that for all $\gamma \in \Gamma$ and for all $v\in\mathcal{B}_{-\frac{1}{2}}(d_{o,\epsilon}):$
\begin{equation}
\|\pi_{-\frac{1}{2}}(\gamma)v\|_{\mathcal{B}_{-1/2}}=\|v\|_{\mathcal{B}_{-1/2}}.
\end{equation}
\\ 
In other words when $\mathcal{B}_{-1/2}$ is not trivial,  $(\pi_{-\frac{1}{2}},\mathcal{B}_{-1/2} )$ is  a unitary representation.

 \subsection{ $\ell^{2}$-cohomology of degree one}
We follow \cite{BouMV}. Let $\Gamma$ be a finitely generated group.
  Denote by $F(\Gamma)$ the space of all complex valued functions on $\Gamma$ and by $\lambda$ the left regular linear representation of $\Gamma$ on $F(\Gamma)$. Define 
then the space of $2$-Dirichlet finite functions on $\Gamma$:

$$D_{2}(\Gamma)=\{ f\in F(\Gamma)| \lambda(\gamma)f-f\in \ell^{2}(\Gamma),\forall \gamma \in \Gamma\} $$

The first $L^{2}$-cohomology of $\Gamma$ is defined as

$$\bar H^{1}_{(2)}(\Gamma):=D_{2}(\Gamma)/ i(\ell^{2}(\Gamma))+\mathbb{C},$$
where $i$ denotes the inclusion of $\ell^{2}(\Gamma)$ into $D_{2}(\Gamma)$. The space $\bar H^{1}_{(2)}(\Gamma)$ turns to be a Hilbert module over  the von Neumann algebra associated with $\Gamma$. It has a von Neumann dimension called  \emph{the first $L^{2}$-betti number} of $\Gamma$ denoted by $\beta^{1}_{(2)}(\Gamma).$ Eventually,  recall that 
$\beta^{1}_{(2)}(\Gamma)=0$ if and only if $\bar H^{1}_{(2)}(\Gamma)=0.$

The connection with Besov space is the following theorem, that is a particular case of a more general nice theorem due to Bourdon and Pajot in \cite[Théorème 0.1]{BP}:

\begin{theorem}\label{ThBP}
For any $d\in J_{AR}(\partial X)$, there exists canonical $\Gamma$-equivariant isomorphism of Hilbert spaces between $\bar H^{1}_{(2)}(\Gamma)$  and  $\mathcal{B}_{-1/2}(d)$. 
\end{theorem}

 \begin{prop}
 If $\beta_{(2)}^{1}(\Gamma)> 0$ then there exists $d\in J_{AR}(\partial X)$ such that $(\pi_{-\frac{1}{2}},\mathcal{B}_{-\frac{1}{2}}(d))$ is a non trivial, infinite dimensional, unitary representation of $\Gamma$.
 \end{prop}
 
 \begin{proof}
Since $\beta_{(2)}^{1}(\Gamma)> 0$ the space $\bar H^{1}_{(2)}(\Gamma)$ is not trivial.
Choose now $d_{o,\epsilon}$  an $\epsilon$-good visual metric on $\partial X$. Therefore, $\mathcal{B} _{-\frac{1}{2} } (d_{o,\epsilon})$ is a non trivial Hilbert space of infinite dimension, since its von Neumann dimension is non zero, and the choice of the Ahlfors regular metric turns $(\pi_{-\frac{1}{2}},\mathcal{B}_{-\frac{1}{2}}(d_{o,\epsilon}) )$ into a unitary representation of $\Gamma$.
\end{proof}
\begin{remark}
     We know that for $t>0$ the space $\mathcal{B}_{t}$ exists and is non empty: it contains $C(\partial X)$ for instance by the integrability property of $k_{t}$. But for $t<0$, the space $\mathcal{B}_{t}$ might be trivial: Assume that $\Gamma$ is hyperbolic without torsion with property (T). Then, using again Theorem \ref{ThBP}  together with the celebrated theorem of Delorme-Guichardet ensuring that property (T) groups have a trivial first group of cohomology \cite[Theorem 2.12.14]{BDV}, we obtain that $\mathcal{B}_{-\frac{1}{2}}(d)$ is trivial for any $d\in J_{AR}(\partial X)$. \end{remark}

\end{document}